\documentclass[10pt]{book}

\usepackage[utf8]{inputenc}

\usepackage{geometry} 

\usepackage{lmodern}
\usepackage{titling}
\usepackage[bottom]{footmisc}
\usepackage[shortlabels]{enumitem}

\usepackage{mathtools}
\usepackage{amsthm}
\usepackage{amsfonts}
\usepackage{textcomp}
\usepackage{gensymb}
\usepackage{amssymb}
\usepackage{mathrsfs}
\usepackage{shuffle}
\usepackage{empheq}

\usepackage[10pt]{mytrees}
\usepackage{color}

\usepackage{multirow}
\usepackage{boldline}
\usepackage{flafter}
\usepackage{placeins}

\usepackage{csquotes}
\usepackage[T1]{fontenc}
\usepackage[nottoc,notlof,notlot]{tocbibind}\usepackage[intoc,refpage]{nomencl}

\usepackage{makeidx}
\usepackage[pdfborder={0 0 1 [1 2]}]{hyperref}
\hypersetup{
pdfauthor={Rosa Preiß},
pdftitle={From Hopf algebras to rough paths and regularity structures},
pdfsubject={Master's Thesis}
}

\DeclareFontEncoding{LS1}{}{}
\DeclareFontSubstitution{LS1}{stixscr}{m}{n}
\DeclareFontSubstitution{LS1}{stixbb}{m}{it}
\DeclareFontFamily{LS1}{stixbb}{\skewchar\font127 }
\DeclareFontFamily{LS1}{stixscr}{\skewchar\font127 }
\DeclareFontShape{LS1}{stixbb}{m}{it}{<-> stix-mathbbit}{}
\DeclareFontShape{LS1}{stixscr}{m}{n}{<-> stix-mathscr}{}
\DeclareSymbolFont{stixsymbols}{LS1}{stixscr}{m}{n}
\DeclareSymbolFont{stixsymbols4}{LS1}{stixbb}{m}{it}
\DeclareMathSymbol{\smblkdiamond}{\mathbin}{stixsymbols4}{"E4}
\DeclareMathSymbol{\smwhtdiamond}{\mathbin}{stixsymbols}{"0B}

\newcommand*{\ie}{i.e.\@ }
\newcommand*{\resp}{resp.\@ }
\newcommand*{\cf}{cf.\@ }
\newcommand*{\eg}{e.g.\@ }

\newcommand*{\emphind}[1]{\emph{#1}\index{#1}}
\newcommand*{\eqrefequal}[1]{\stackrel{\mathclap{\text{\tiny\eqref{#1}}}}{=}}
\newcommand*{\diamondequal}{\stackrel{\mathclap{(\diamondsuit)}}{=}}

\newcommand*{\AND}{\;\wedge\;}

\newcommand*{\ofbracket}[1]{{(#1)}}
\newcommand*{\ofofbracket}[1]{{(\!\ofbracket{#1}\!)}}
\newcommand*{\tensoral}[1]{{\mathrm{T}\ofbracket{#1}}}
\newcommand*{\tensoralfs}[1]{{\mathrm{T}\ofofbracket{#1}}}
\newcommand*{\tensoraltc}[2]{{\mathrm{T}^{#1}\ofbracket{#2}}}

\newcommand*{\dualbracket}{{\langle{\,\cdotp,\cdotp\,}\rangle}}
\newcommand*{\tdual}[1]{#1'}
\newcommand*{\adual}[1]{{#1}^*}

\newcommand*{\pdual}[1]{\tilde{#1}}
\newcommand*{\Vpdual}{\pdual{V}}
\newcommand*{\vpdual}{\pdual{v}}
\newcommand*{\isdualvia}[2]{\mathrel{\mathop{{\uparrow}\!{\downarrow}}\limits_{#1}^{#2}}}
\newcommand*{\isdual}{\mathrel{{\uparrow}\!{\downarrow}}}

\newcommand*{\linear}{\mathrm{L}}
\newcommand*{\linearcont}{\mathcal{L}}

\newcommand*{\product}{\smwhtdiamond}
\newcommand*{\linofbilin}[1]{\mathord{\mathrm{m}_{\scriptscriptstyle{#1}}}}
\newcommand*{\mproduct}{\linofbilin{\product}}
\newcommand*{\producttwo}{\mathbin{\underline{\product}}}
\newcommand*{\mproducttwo}{\linofbilin{\producttwo}}
\newcommand*{\coproduct}{\mathord{\Delta}}
\newcommand*{\coproducttwo}{\mathord{\underline{\coproduct}}}
\newcommand*{\unitmap}{\mathrm{u}}
\newcommand*{\counit}{\varepsilon}
\newcommand*{\primitive}{\operatorname{Prim}}
\newcommand*{\grouplike}{\operatorname{GrLk}}
\newcommand*{\antipode}{\mathord{\mathcal{S}}}
\newcommand*{\rcoproduct}{\mathord{\tilde{\Delta}}}
\newcommand*{\convoproduct}{\ast}
\newcommand*{\rconvoproduct}{\mathbin{\tilde{\ast}}}
\newcommand*{\flip}{\hat{\tau}}

\newcommand*{\proj}{\mathord{\pi}}
\newcommand*{\projtwo}{\mathord{\underline{\proj}}}
\newcommand*{\projmult}[1]{\proj^{[#1]}}
\newcommand*{\copropro}{\mathord{\Delta_{\scriptscriptstyle\product}}}

\newcommand*{\productnum}[1]{\mathbin{\if 1#1 \smwhtdiamond\else\smblkdiamond \fi}}
\newcommand*{\mproductnum}[1]{\linofbilin{\productnum{#1}}}
\newcommand*{\copropronum}[1]{\mathord{\Delta_{\scriptscriptstyle\productnum{#1}}}}
\newcommand*{\antipodenum}[2]{\mathord{{\mathcal{S}_{\scriptscriptstyle\productnum{#1}}\hspace{-0.22em}}^{\scriptscriptstyle\productnum{#2}}}}
\newcommand*{\counitnum}[1]{\counit_{\scriptscriptstyle\productnum{#1}}}
\newcommand*{\unitmapnum}[1]{\unitmap_{\scriptscriptstyle\productnum{#1}}}

\newcommand*{\diffd}{\mathrm{d}}
\newcommand*{\intd}[0]{\,\diffd}
\newcommand*{\forallc}[0]{\forall\,}
\newcommand*{\existsc}[0]{\exists\,}

\newcommand*{\deltasymb}[1]{\delta_{#1}}
\newcommand*{\polynoms}{\mathcal{P}}
\newcommand*{\diffs}{\mathbb{D}}
\newcommand*{\propol}{\bullet}
\newcommand*{\propoltwo}{\mathbin{\underline{\bullet}}}
\newcommand*{\mpropol}{\linofbilin{\propol}}
\newcommand*{\prodiff}{\circ}
\newcommand*{\copropol}{\mathord{\Delta_{\scriptscriptstyle\prodiff}}}
\newcommand*{\coprodiff}{\mathord{\Delta_{\scriptscriptstyle\propol}}}
\newcommand*{\rcopropol}{\mathord{\rcoproduct_{\scriptscriptstyle\prodiff}}}

\newcommand*{\etensor}{\otimes}
\newcommand*{\itensor}{\mathbin{\dot\otimes}}
\newcommand*{\itensortwo}{\mathbin{\underline{\itensor}}}
\newcommand*{\X}{\mathrm{X}}
\newcommand*{\vspan}{\operatorname{span}}
\newcommand*{\naturals}{\mathbb{N}}
\newcommand*{\D}{\mathcal{D}}
\newcommand*{\antipol}{\mathord{\mathcal{S}_{\scriptscriptstyle\propol}^{\scriptscriptstyle\prodiff}}}
\newcommand*{\antidiff}{\mathord{\mathcal{S}_{\scriptscriptstyle\prodiff}^{\scriptscriptstyle\propol}}}
\newcommand*{\bigpropol}{\prod}
\newcommand*{\reals}{\mathbb{R}}

\newcommand*{\imorph}[1]{\iota_{#1}}
\newcommand*{\imorphpd}{\imorph{\polynoms_d\diffs_d}}

\newcommand*{\unit}{\mathbf{1}}
\newcommand*{\unitmappol}{\unitmap}
\newcommand*{\unitmapdiff}{\unitmap}
\newcommand*{\counitpol}{\counit}
\newcommand*{\counitdiff}{\counit}

\newcommand*{\counittens}{\varepsilon_{\scriptscriptstyle\shuffle}}
\newcommand*{\copropodu}{\mathord{\Delta_{\adual{\polynoms}}}}
\newcommand*{\antipodu}{\mathord{\mathcal{S}_{\polynoms'}}}

\newcommand*{\deconc}{\mathord{\Delta_{\scriptscriptstyle\itensor}}}
\newcommand*{\shuffleperm}{\operatorname{Sh}}
\newcommand*{\shuffletwo}{\mathbin{\underline{\shuffle}}}
\newcommand*{\symgroup}{\mathrm{S}}
\newcommand*{\deshuffle}{\mathord{\Delta_{\scriptscriptstyle\shuffle}}}
\newcommand*{\rdeshuffle}{\mathord{\rcoproduct_{\scriptscriptstyle\shuffle}}}
\newcommand*{\antitensor}{\mathord{\mathcal{S}_{\scriptscriptstyle\itensor}^{\scriptscriptstyle\shuffle}}}
\newcommand*{\mitensor}{\linofbilin{\itensor}}
\newcommand*{\mitensortwo}{\linofbilin{\itensortwo}}
\newcommand*{\id}{\mathord{\operatorname{id}}}
\newcommand*{\complexs}{\mathbb{C}}

\newcommand*{\antishuffle}{\mathord{\mathcal{S}_{\scriptscriptstyle\shuffle}^{\scriptscriptstyle\itensor}}}

\newcommand*{\itensortc}[1]{\mathbin{\itensor_{\scriptscriptstyle n}}}
\newcommand*{\shuffletc}[1]{\mathbin{\shuffle_{\scriptscriptstyle n}}}

\newcommand*{\liealg}{\mathfrak{l}}
\newcommand*{\liebracket}{\mathfrak{b}}
\newcommand*{\productindex}{_{\scriptscriptstyle\product}}
\newcommand*{\rnbracketing}{\mathfrak{R}}
\newcommand*{\derivgrade}[1]{\mathscr{D}}
\newcommand*{\lingen}[1]{\langle#1\rangle}
\newcommand*{\generate}[2]{\lingen{#1}_{#2}}

\newcommand*{\extproduct}{\mathbin{{\product}_{\scriptscriptstyle\mathrm{x}}}}
\newcommand*{\mextproduct}{\linofbilin{\extproduct}}
\newcommand*{\extcoproduct}{\mathord{{\coproduct}_{\scriptscriptstyle\mathrm{x}}}}
\newcommand*{\extproducttwo}{\mathord{{\producttwo}_{\scriptscriptstyle\mathrm{x}}}}
\newcommand*{\extcounit}{\counit_{\scriptscriptstyle\mathrm{x}}}
\newcommand*{\extantipode}{\mathord{{\antipode}_{\scriptscriptstyle\mathrm{x}}}}

\newcommand*{\directproduct}{\prod}

\newcommand*{\spacefs}[1]{#1_{\Pi}}
\newcommand*{\liefs}[1]{\mathfrak{g}_{#1}}
\newcommand*{\groupfs}[1]{\mathscr{G}_{#1}}
\newcommand*{\exppro}{\exp_{\product}}
\newcommand*{\logpro}{\log_{\product}}

\newcommand*{\quotient}{/}
\newcommand*{\quotientmap}{\mathord{\pi}}
\newcommand*{\quotientmaptwo}{\mathord{\underline{\quotientmap}}}
\newcommand*{\truncationop}[1]{\mathord{\quotientmap^{#1}}}
\newcommand*{\truncationoptwo}[1]{\mathord{\underline{\quotientmap}^{#1}}}

\newcommand*{\trees}{\mathscr{T}}
\newcommand*{\forests}{\mathscr{F}}
\newcommand*{\tonewroot}[1]{\left\lfloor #1 \right\rfloor}
\newcommand*{\lroot}{\lfloor}
\newcommand*{\rroot}{\rfloor}

\newcommand*{\forestspace}[1]{\lingen{\forests_{#1}}}

\newcommand*{\treepro}{\odot}
\newcommand*{\treeprotwo}{\mathbin{\underline{\treepro}}}
\newcommand*{\destar}{\mathord{\Delta_{\scriptscriptstyle\star}}}
\newcommand*{\rdestar}{\mathord{\tilde{\Delta}_{\scriptscriptstyle\star}}}
\newcommand*{\mtreepro}{\linofbilin{\treepro}}
\newcommand*{\antitree}{\mathord{\mathcal{S}_{\scriptscriptstyle\!\treepro}^{\star}}}
\newcommand*{\mstar}{\linofbilin{\star}}
\newcommand*{\antistar}{\mathord{\mathcal{S}_{\star}^{\scriptscriptstyle\treepro}}}
\newcommand*{\cuts}[1]{\operatorname{Cuts}\left(#1\right)}
\newcommand*{\cutfactor}[3]{\mathrm{c}_{#1}\!\left(#2,#3\right)}
\newcommand*{\detreepro}{\mathord{\coproduct}_{\scriptscriptstyle\treepro}}
\newcommand*{\splits}[1]{\operatorname{Splits}\left(#1\right)}
\newcommand*{\splitfactor}[2]{\mathrm{s}_{#1}\!\left(#2\right)}
\newcommand*{\butchergroup}{\mathscr{B}}

\newcommand*{\rX}{\mathbf{X}}
\newcommand*{\fullrX}{\bar{\rX}}
\newcommand*{\dualrX}{\check{\rX}}
\newcommand*{\dualfullrX}{\hat{\rX}}
\newcommand*{\timeT}{[0,T]}
\newcommand*{\level}{\mathbf{n}}
\newcommand*{\geogroup}[2]{{\mathrm{G}_{#2}^{#1}}}
\newcommand*{\geogroupfs}[1]{{\mathrm{G}_{#1}}}
\newcommand*{\geolie}[2]{{\mathrm{L}_{#2}^{#1}}}
\newcommand*{\geoliefs}[1]{{\mathrm{L}_{#1}}}
\newcommand*{\words}{\mathcal{W}}
\newcommand*{\geopaths}{\mathrm{R}}
\newcommand*{\geopathsfull}{\bar{\geopaths}}
\newcommand*{\geopathsdual}{\check{\geopaths}}
\newcommand*{\geopathsdualfull}{\hat{\geopaths}}

\newcommand*{\forestgroup}[2]{{\mathscr{G}_{#2}^{#1}}}
\newcommand*{\forestgroupfs}[1]{{\mathscr{G}_{#1}}}
\newcommand*{\branchedpaths}{\mathscr{R}}
\newcommand*{\branchedpathsfull}{\bar{\branchedpaths}}
\newcommand*{\branchedpathsdual}{\check{\branchedpaths}}
\newcommand*{\branchedpathsdualfull}{\hat{\branchedpaths}}

\newcommand*{\bra}[1]{\mathord{\langle #1|}}
\newcommand*{\structaction}[1]{\bar{\Gamma}_{#1}}
\newcommand*{\abstractint}{\mathcal{I}}
\newcommand*{\geomodelspace}{T^\mathrm{g}}
\newcommand*{\branchedmodelspace}{T^\mathrm{b}}
\newcommand*{\regstruc}{\mathcal{T}}
\newcommand*{\compact}{\mathfrak{K}}
\newcommand*{\schwartz}{\mathscr{S}}
\newcommand*{\schwartzd}{\tdual{\schwartz}}
\newcommand*{\spaceofmodels}{\mathscr{M}}
\newcommand*{\modeld}{\mathcal{D}}
\newcommand*{\contdiff}{\mathrm{C}}
\newcommand*{\abstractkernel}{\mathcal{K}}
\newcommand*{\lebesgue}{\mathscr{L}}
\newcommand*{\reconstr}{\mathcal{R}}
\newcommand*{\sectorY}{\lingen{\mathcal{Y}}}
\newcommand*{\sectorYdot}{\lingen{\dot{\mathcal{Y}}}}
\newcommand*{\hoelder}{\mathcal{C}}
\newcommand*{\lVERT}{\mathopen{\lvert\!\lvert\!\lvert}}
\newcommand*{\rVERT}{\mathclose{\rvert\!\rvert\!\rvert}}
\newcommand*{\testfcts}[2]{\mathcal{B}^#1(\reals^#2)}
\title{From Hopf algebras to rough paths and regularity structures}
\author{Rosa Preiß}
\date{\today}
         
\makenomenclature
\makeindex
\begin{document}

\newtheorem{lemma}{Lemma}[chapter]
\newtheorem{thm}[lemma]{Theorem}
\newtheorem{cor}[lemma]{Corollary}
\newtheorem{conj}[lemma]{Conjecture}
\theoremstyle{definition}
\newtheorem{defn}[lemma]{Definition}
\newtheorem{subdefn}{Definition}[lemma]
\theoremstyle{remark}
\newtheorem{rmk}[lemma]{Remark}
\newtheorem{ansatz}[lemma]{Ansatz}
\newtheorem{example}[lemma]{Example}
\numberwithin{equation}{chapter}
\numberwithin{table}{chapter}

\frontmatter
\begin{titlingpage}
 \begin{center}
  \textsc{\Large Master's Thesis}\\[0.5\baselineskip]
  \huge\textbf{\thetitle}\\[2\baselineskip]
  \LARGE\theauthor\\[0.5\baselineskip]
  \Large \textsc{Supervisors}: Peter K. Friz \textsc{and} Sylvie Paycha\\[2\baselineskip]
  \textsc{\Large Technische Universität Berlin}
  \vfill
  \Large July 29, 2016\\
 \end{center}
\end{titlingpage}
\chapter*{Eidesstattliche Erklärung}
Hiermit erkläre ich, dass ich die vorliegende Arbeit selbstständig und eigenhändig sowie ohne unerlaubte fremde Hilfe und ausschließlich unter Verwendung der aufgeführten Quellen und \mbox{Hilfsmittel} angefertigt habe.\\[1\baselineskip]
\noindent Die selbständige und eigenständige Anfertigung versichert an Eides statt:\\[1\baselineskip]
Berlin, den\\[2\baselineskip]
{\tiny Unterschrift}
\chapter*{Zusammenfassung in deutscher Sprache}
 Im Mittelpunkt dieser Arbeit stehen die von Lyons entwickelten Rough Paths (wörtlich 'raue Pfade'), die einen algebraischen und analytischen Rahmen für Stieltjes Integrale, also Integrale eines Pfades gegen einen anderen, bieten, in Fällen, in denen die betrachteten Pfade so rau sind, dass das klassische Riemann-Stieltjes Integral nicht mehr konvergiert. Bevor wir Rough Paths formal definieren können, beginnen wir mit einer Einführung in einige algebraische Grundbegriffe. Dazu gehören zunächst Algebren und Koalgebren, zwei Begriffe, die in einer gewissen Dualitätsbeziehung zueinander stehen. Als Kombination der beiden Begriffe erhalten wir Bialgebren, und als Spezialfall dieser wiederum Hopf Algebren, die für diese Arbeit von zentraler Bedeutung sind. Nachdem wir uns mit wichtigen Eigenschaften dieser Objekte vertraut gemacht haben, widmen wir uns einer ausführlichen Diskussion verschiedener Beispiele von Hopf Algebren. Dazu gehört unter anderem die Hopf Algebra der Polynome, bei der das Produkt durch die übliche Multiplikation von Polynomen gegeben ist und sich das Koprodukt durch eine vertraute Binomialformel ausdrücken lässt. Anschließend benutzen wir das Beispiel der Tensor Hopf Algebren zur Definition von Weakly Geometric Rough Paths (wörtlich 'schwach geometrische raue Pfade'), die solchen Formen von Stieltjes Integralen entsprechen, welche die übliche Regel der partiellen Integration erfüllen. Für Fälle wie die Itô-Integration, bei denen diese Regel nicht mehr anwendbar ist, betrachten wir die von Gubinelli eingeführten Branched Rough Paths (wörtlich 'verzweigte raue Pfade') und verwenden dazu das Beispiel der Hopf Algebren auf Räumen von als Bäume und Wälder bezeichneten Strukturen. Schließlich gehen wir noch auf einige Grundkonzepte der von Hairer entwickelten Theorie der Regularitätsstrukturen ein, die uns einen weiteren Zugang zu Weakly Geometric und Branched Rough Paths bietet. Dabei betrachten wir zuerst eine von Hairer beschriebene allgemeine Methode, wie sich aus bestimmten Hopf Algebren eine Regularitätsstruktur ableiten lässt, um dann auf Basis einer formalen Picard Iteration eine Regularitätstruktur zu entwickeln, welche besser zur Betrachtung sogenannter Rough Differential Equations (wörtlich 'raue Differentialgleichungen') geeignet ist.
\chapter*{Abstract}
Lyon's rough paths give an algebraic and analytic framework for Stieltjes integrals in a regime of low regularity where the usual Riemann-Stieltjes integral does not converge. Before we may rigorously define rough paths, we start with the introduction of some basic algebraic terminology. Among them are algebras and coalgebras, two notions which are in some sense dual to each other. As a combination of these notions we obtain bialgebras, and as a special case of them then Hopf algebras, which play a central role in this thesis. After further algebraic preliminaries, we give the examples of Hopf algebras we are interested in. Among them is the example of the polynomial Hopf algebra, whose product is nothing but the usual multiplication of polynomials and whose coproduct can be expressed very simply with the help of a binomial coefficient. We then use the dual pair of tensor Hopf algebras to introduce weakly geometric rough paths, which correspond to notions of Stieltjes integrals satisfying the usual integration by parts rule. For cases like Itô-integration where we need to give up integration by parts, we look at Gubinelli's branched rough paths based on the dual pair of Hopf algebras on trees and forests. Finally, we give some basic concepts of Hairer's theory of regularity structures and use them for a different approach to branched and weakly geometric rough paths. While we first look at a general method described by Hairer to derive a regularity structure from certain Hopf algebras, we then develop a regularity structure based on a formal Picard iteration which is more suitable for dealing with rough differential equations.

\tableofcontents

\chapter*{Acknowledgements\markboth{}{}}
\addcontentsline{toc}{chapter}{Acknowledgements}
 My first contact with rough paths and regularity structures was through the lecture of my main supervisor Peter K. Friz from \emph{Technische Universität Berlin} which was based on the book \cite{frizhairer14}. After I got interested in writing my master's thesis in the subject, Peter pointed me to the paper \cite{hairerkelly14} and introduced me to my second supervisor Sylvie Paycha from the \emph{University of Potsdam}. Peter gave the most important impulses for the scope of the thesis, and both Peter and Sylvie accompanied the writing process with regular meetings and valuable feedback. Sylvie also gave me the opportunity to speak about some of the topics at the mathematics institute in Potsdam.
 
 I furthermore want to thank Pierre Clavier from the working group of Sylvie for our discussions about the connections of our fields of research and for reading through an early version of the thesis, giving valueable remarks and feedback. Also, Ilya Chevyrev from the working group of Peter, for looking at one of the later drafts, for discussions and suggestions, and Yvain Bruned from the \emph{University of Warwick} for providing information and examples regarding the realization of tree symbols in LaTeX which helped creating the ones I use in this thesis, and Youness Boutaib from the working group of Sylvie for looking at a near-finished version of the thesis.
\mainmatter
\chapter{Introduction}
Consider a family $(X^i)_{i=1}^d$ of paths $X_i\in\hoelder^\gamma([0,T],\reals)$, the space of Hoelder continuous functions of exponent $\gamma>1/2$. By the integration theory of Young, introduced in the paper \cite{young36}, the Riemann-Stieltjes integrals (Equation (4.1) \cite{frizhairer14})
\begin{equation*}
 \int_s^t X_r^i\intd X_r^j:=\lim_{\lvert\mathcal{P}\rvert\rightarrow 0}\sum_{[u,v]\in\mathcal{P}}X^i_u(X^j_v-X^j_u),
\end{equation*}
where the limit is taken over partitions $\mathcal{P}$ of $[s,t]$, are well defined (\cf Section 3.3.2 \cite{lyonsqian02}). We may even compute \emph{iterated integrals} of the type (Equation (1.3) \cite{hairerkelly14})
\begin{equation*}
 \dualfullrX_{st}(e_{i_1\ldots i_n}):=\int_s^t\cdots\int_s^{r_1}\intd X^{i_1}_{r_1}\cdots\intd X^{i_n}_{r_n},
\end{equation*}
where $n\in\naturals,\,i_1,\ldots i_n\in\{1,\ldots,d\}$. The set of all these \emph{words} is denoted by $\words_d$. Together with later research, most notably by Chen (see \eg \cite{chen71}), it turns out that these iterated integrals fulfill the following algebraic and analytic conditions (Based on Definition 1.2.\@ \cite{hairerkelly14}):
\begin{enumerate}
 \item $\dualfullrX_{st}(\unit)=1$ and $\dualfullrX_{st}(w_1\shuffle w_2)=\dualfullrX_{st}(w_1)\dualfullrX_{st}(w_2)$ for all words $w_1,w_2\in\words_d$,
 \item $\dualfullrX_{tt}=\counit$ and $\dualfullrX_{st}(w)=(\dualfullrX_{su}\etensor\dualfullrX_{ut})\deconc w=\sum_{(w)}^{\itensor}\dualfullrX_{su}(w^1)\dualfullrX_{ut}(w^2)$ for all words $w\in\words_d$,
 \item $\sup_{s\neq t}\frac{\lvert\dualfullrX_{st}(w)\rvert}{\lvert t-s\rvert^{\gamma\lvert w\rvert}}<\infty$ for all words $w\in\words_d$,
\end{enumerate}
where $\shuffle$ is the \emph{shuffle product} and $\deconc$ is the \emph{deconcatenation coproduct} on the real vector space $\lingen{\words_d}$. The object $(\lingen{\words_d},\shuffle,\deconc)$ forms a \emph{bialgebra}. More precisely, it is a special kind of a bialgebra called a \emph{Hopf algebra}.

The crucial point, though, is that we assumed a Hölder exponent $\gamma>1/2$. Once we have paths of lower Hölder regularity, all of the above breaks down. Our iterated integrals aren't defined any more. This is where \emph{rough paths} come into play. Now that we have $\gamma\in(0,1/2]$, the idea is to take the above three conditions as a definition, and to say that if a map $\dualrX:\,[0,T]^2\rightarrow\adual{\lingen{\words_d}}$ obeys all of them, then we call it a \emph{weakly geometric rough path}. 

This is however not the last generalization we want to make. Condition 1.\@ from above encodes nothing but the integration by parts rule applied to iterated integrals. Since we also want to cover theories of integration like Itô-integration which do not satisfy integration by parts, we later give another collection of conditions based on a different Hopf algebra. This leads to the concept of \emph{branched rough paths}, which was introduced by Gubinelli in the paper \cite{gubinelli10}. Rough paths as such however were originally developed by Lyons in the papers \cite{lyons94}, \cite{lyons95} and \cite{lyons98}.

Before we can rigorously introduce rough paths, there are a lot of algebraic concepts to understand. In Chapter \ref{chapter:basics}, we do so in mainly following the notes \cite{manchon06}, but also repeatedly refer to the book \cite{abe80}. The standard reference for the following topics is the book \cite{sweedler69}. We start by defining \emph{algebras} and \emph{coalgebras}, two notions which are in some sense dual to each other. With certain \emph{compatibility requirements}, they can be combined to the notion of a bialgebra. For some bialgebras, there is a certain algebra and coalgebra antimorphism called the \emph{antipode}, which is characterized by a property involving the so-called \emph{convolution product} algebra given by the product and coproduct. Bialgebras with an antipode are called Hopf algebras. We introduce gradings which make it possible for us to even better understand the structure of the examples of Hopf algebras we are interested in even better. Finally, we formally describe the duality relation already mentioned, but also learn about duality between Hopf algebras as a whole, which was presented in almost the same way in \cite{chen71}.

Chapter \ref{chapter:connected_graded_Hopf} then presents these examples of Hopf algebras. First, we look at the polynomial Hopf algebra with the usual product of polynomials and a coproduct which is easy to express through a multi dimensional binomial coefficient. A dual Hopf algebra can be described as the algebra of differential operators with composition product, but it turns out to be isomorphic to the polynomial one. Referring to the presentations of the book \cite{reutenauer93}, the paper \cite{hairerkelly14} and again \cite{manchon06}, we then look at the \emph{tensor algebra} whose coproduct is nothing but the dual of the shuffle product mentioned above.

Afterwards, we take a short look at the technicalities of formal series and truncations, and furthermore state some well-known basics about Lie algebras. Also, we introduce the exponential map, which maps a Lie algebra to its Lie group, and its inverse, the logarithm. The exponential map then plays an important role for us by relating \emph{primitive} and \emph{group-like} elements in the case of a connected graded Hopf algebras, as explained in \cite{hairerkelly14} and for the special case of the tensor algebra in \cite{reutenauer93}.

As our last example, we look at the Connes-Kreimer Hopf algebra of trees and forests, mainly based on \cite{hairerkelly14} and the original paper \cite{conneskreimer98}. After introducing all the objects fully recursively, we give representation formulas which formalize the notion of admissible cuts already introduced in \cite{conneskreimer98} without relying on interpreting trees as graphs. These formulas make it easy to introduce the dual Hopf algebra whose product is the \emph{Grossman-Larson product}, the dual of the Connes-Kreimer coproduct.

In Chapter \ref{chapter:rough_paths}, we finally get to learn more about weakly geometric and branched rough paths. In both cases, we look at the equivalence of four different kinds of definitions. For that equivalence, the extension theorems of Lyons, which we took from the book \cite{lyonscaruanalevy07}, and of Gubinelli, found in \cite{gubinelli10}, play a crucial role. Citing the result from \cite{hairerkelly14}, we will see that weakly geometric rough paths are indeed included in the concept of branched rough paths.

The final Chapter \ref{chapter:regularity_structures} then gives an introduction to some concepts from Hairer's theory of regularity structures which was originally presented in its full scope in the paper \cite{hairer14}. Exactly as described in \cite{hairer14}, we look at how a regularity structure can be derived from a connected graded Hopf algebra, and we relate rough paths to \emph{models} for the regularity structure obtained from the respective Hopf algebra. The last section then generalizes a different approach to regularity structures for rough paths which \cite{frizhairer14} shortly looked at for the simplest non-trivial case. It is motivated by a formal Picard iteration, which is a standard method of solving partial differential equations, which we here use for the problem of a rough differential equation.
\chapter{Basic definitions and results}\label{chapter:basics}
For some vector space $V$ over some field $K$, we denote by $\adual{V}$\nomenclature[V*]{$\adual{V}$}{Algebraic dual space of the vector space $V$} its \emphind{algebraic dual space}\index{dual space!algebraic@\textit{algebraic}}, \ie the vector space of all linear maps ${V\rightarrow K}$. If $V$ is equipped with a topology, we write $\tdual{V}$ for its \emphind{topological dual space}\index{dual space!topological@\textit{topological}}, \ie the vector space of all continuous linear maps ${V\rightarrow K}$.
\begin{defn}(Section I.1.\@ \cite{manchon06})
 Let $V_1, V_2$ be two vector spaces over the same field $K$. A \emphind{tensor product} $V_1\etensor V_2$ is a $K$ vector space for which there is a bilinear map $\iota:\,V_1\times V_2\rightarrow V_1\etensor V_2$ with the property that for every $K$ vector space $W$ and every bilinear map $\rho:\,V_1\times V_2\rightarrow W$, there is a unique linear map $\linofbilin{\rho}:\,V_1\etensor V_2\rightarrow W$ such that
 \begin{equation}\label{eq:tensor_property}
  \rho=\linofbilin{\rho}\circ\iota.
 \end{equation}
\end{defn}
We write $v_1\etensor v_2:=\iota(v_1,v_2)$ for all $v_1\in V_1,\,v_2\in V_2$. With that notation, Equation \eqref{eq:tensor_property} can be formulated as
\begin{equation*}
 \rho(v_1,v_2)=\linofbilin{\rho}(v_1\etensor v_2)\quad\forallc v_1\in V_1,\, v_2\in V_2.
\end{equation*}

It turns out that for every pair of vector spaces $V_1, V_2$ over the same field, there exists a tensor product $V_1\etensor V_2$, and that the tensor product is unique up to linear isomorphisms (Proposition I.1.1.\@ \cite{manchon06}). For the proof of existence, one needs the axiom of choice. See the proof of Proposition I.1.1.\@ \cite{manchon06} for both existence and uniqueness.

What is very important for us is that for every element $x\in V_1\etensor V_2$, there is $n\in\naturals$ and $v_1^1,\ldots v_1^n\in V_1$, $v_2^1,\ldots,v_2^n\in V_2$ such that
\begin{equation*}
 x=\sum_{i=1}^n \iota(v_1^i,v_2^i)=\sum_{i=1}^n v_1^i\etensor v_2^i.
\end{equation*}
For this fact, also see the proof of Proposition I.1.1.\@ \cite{manchon06}. Note that such a representation is of course not unique. Also note that in the literature, one may sometimes find notions of completed tensor products, especially when the vector spaces are endowed with topologies, and for these objects, this fact does not hold true.

Due to the uniqueness of the tensor product up to isomorphisms, we may identify $(V_1\etensor V_2)\etensor V_3$ with $V_1\etensor(V_2\etensor V_3)$, and therefore simply write $V_1\etensor V_2\etensor V_3$ (Section I.1.\@ \cite{manchon06}). 

\section{Algebras and coalgebras}

\begin{defn}(Section I.2.1.\@ \cite{manchon06})
 A vector space $A$ together with an associative bilinear map $\product:\,A\times A\rightarrow A$ is called an \emphind{algebra}. A subspace $S\subseteq A$ is called \emphind{subalgebra} if $S\product S\subseteq S$ and \emphind{two-sided ideal}\index{ideal} if $S\product A+A\product S\subseteq S$ holds. $(A,\product)$ is \emph{unital}\index{unital algebra} if there is a \emphind{unit element} $\unit\in A$\nomenclature[001]{$\unit$}{Unit element of an algebra} such that ${\unit\product a=a\product\unit=a}$ for all $a\in A$. The function ${\unitmap:\,K\rightarrow A,\,t\mapsto t\unit}$\nomenclature[u]{$\unitmap$}{Unit map of an algebra}, where $K$ is the field under $A$, is called \emphind{unit map}.
\end{defn}
If an algebra $(A,\product)$ is given, we denote by ${\mproduct:\,A\etensor A\rightarrow A}$ the linear map generated by ${\mproduct(a_1\etensor a_2):=a_1\product a_2}$\nomenclature[m product]{$\mproduct$}{Linear map corresponding to the bilinear product $\product$}. Using this notation, associativity reads
\begin{equation}\label{eq:associativity}
 \mproduct(\id\etensor\mproduct)=\mproduct(\mproduct\etensor\id)
\end{equation}
and the unit map is characterized by
\begin{equation}\label{eq:unit_property}
 \mproduct(\id\etensor\unitmap)=\mproduct(\unitmap\etensor\id)=\id,
\end{equation}
where we use the identification $A=A\etensor K=K\etensor A$ implying $\id\etensor\unitmap,\unitmap\etensor\id\in\linear(A,A\etensor A)$.
\begin{defn}(Chapter 1 Section 2.1 \cite{abe80}\footnote{In \cite{abe80}, the unit property is included in the definition of an algebra, and thus also the definition of an algebra homomorphism, which is simply called an algebra morphism there, contains the condition that the unit of one algebra is mapped to the unit of the other algebra.})
 An \emph{algebra homomorphism}\index{algebra!homomorphism}\index{homomorphism!algebra@\textit{algebra}} between two algebras $(A_1,\productnum{1})$, $(A_2,\productnum{2})$ is a linear map $\varLambda:\,A_1\rightarrow A_2$ such that $\varLambda\mproductnum{1}=\mproductnum{2}(\varLambda\etensor\varLambda)$.
\end{defn}

\begin{defn}\label{defn:coalgebra}(Section I.3.1.\@ \cite{manchon06})
 Let $C$ be a vector space together with a linear map $\coproduct:\,C\rightarrow C\etensor C$. If $\coproduct$ is \emph{coassociative}\index{coassociativity}, \ie
 \begin{equation}\label{eq:coassociativity}
  (\coproduct\etensor\id)\coproduct=(\id\etensor\coproduct)\coproduct,
 \end{equation}
 we call $(C,\coproduct)$ a \emphind{coalgebra} and $\coproduct$ a \emphind{coproduct}. A subspace $S\subseteq C$ is called \emphind{two-sided coideal}\index{coideal} if $\coproduct S\subseteq S\etensor C+C\etensor S$ and \emphind{subcoalgebra} if $\coproduct S\subseteq S\etensor S$. The coalgebra $(C,\coproduct)$ is \emph{counital}\index{counital coalgebra} if there is $\counit\in\adual{C}$\nomenclature[00g e2]{$\counit$}{Counit of a coalgebra} such that
 \begin{equation}\label{eq:counit_property}
  (\counit\etensor\id)\coproduct=(\id\etensor\counit)\coproduct=\id,
 \end{equation}
 a so-called \emphind{counit}. $(C,\coproduct)$ is \emph{cocommutative}\index{cocommutativity} if $\flip\coproduct=\coproduct$, where $\flip:\,C\etensor C\rightarrow C\etensor C$\nomenclature[00g]{$\flip$}{Flip} is the \emphind{flip} linearly generated by $\flip(a\etensor b):=b\etensor a$.
\end{defn}
\begin{rmk}
 While every ideal of an algebra is a subalgebra but generally not every subalgebra is an ideal, every subcoalgebra is a coideal but generally not every coideal is a subcoalgebra.
\end{rmk}

We sometimes write 
\begin{equation*}
 \sum_{(c)}c_1\etensor c_2:=\coproduct c
\end{equation*}
\nomenclature[00g S (c) c_1 c_2]{$\sum_{(c)}c_1\etensor c_2$}{Sweedler's notation of some coproduct value $\coproduct c$} and for arbitrary linear $L:\,C\etensor C\rightarrow V$ then $\sum_{(c)}L(c_1\etensor c_2):=L\coproduct c$. This concept is called \emphind{Sweedler's notation} (Section I.3.1.\@ \cite{manchon06}). In Sweedler's notation, coassociativity reads (section I.3.1.\@ \cite{manchon06})
\begin{equation}\label{eq:coassociativity_sweedler}
 \sum_{(c)}\sum_{(c_1)}c_{1:1}\etensor c_{1:2}\etensor c_2=\sum_{(c)}\sum_{(c_2)}c_1\etensor c_{2:1}\etensor c_{2:2}.
\end{equation}
\begin{defn}(Chapter 2 Section 1.1 \cite{abe80}\footnote{As with the definition of an algebra, the definition of a coalgebra in \cite{abe80} already includes the counit property and thus also the definiton of a coalgebra morphism there includes an additional condition concerning the counits.})
A \emph{coalgebra homomorphism}\index{coalgebra!homomorphism}\index{homomorphism!coalgebra@\textit{coalgebra}} between two coalgebras $(C_1,\coproduct_1)$, $(C_2,\coproduct_2)$ is a linear map $\varLambda:\,C_1\rightarrow C_2$ such that $(\varLambda\etensor\varLambda)\coproduct_1=\coproduct_2\varLambda$.
\end{defn}

If we have an algebra $(A,\product)$ and a coalgebra $(C,\coproduct)$, the \emphind{convolution product} $\convoproduct$ (Section I.4.\@ \cite{manchon06}) on the space of linear maps $\linear(C,A)$ is defined by
\begin{equation}\label{eq:defn_convoproduct}
 S\convoproduct T:=\mproduct(S\etensor T)\coproduct.
\end{equation}
$(\linear(C,A),\convoproduct)$ forms an algebra, as
\begin{align*}
 S\convoproduct(T\convoproduct V)&=\mproduct\big(S\etensor(\mproduct(T\etensor V)\coproduct)\big)\coproduct\\
 &=\mproduct(\id\etensor\mproduct)(S\etensor T\etensor V)(\id\etensor\coproduct)\coproduct=\mproduct(\mproduct\etensor\id)(S\etensor T\etensor V)(\coproduct\etensor\id)\coproduct\\
 &=\mproduct\big((\mproduct(S\etensor T)\coproduct)\etensor V\big)\coproduct=(S\convoproduct T)\convoproduct V,
\end{align*}
where we used associativity of $\product$ and coassociativity of $\coproduct$ (based on Section I.4.\@ \cite{manchon06}). If $(A,\product)$ is unitary with unit map $\unitmap$ and $(C,\coproduct)$ is counitary with counit $\counit$, then $(\linear(C,A),\convoproduct)$ is unitary with unit element $\unitmap\counit$ (based on Section I.5.\@ \cite{manchon06}) as
\begin{equation}\label{eq:unit_convoproduct}
 T\convoproduct\unitmap\counit=\mproduct(T\etensor\unitmap\counit)\coproduct=\mproduct(\id\etensor\unitmap)T(\id\etensor\counit) \coproduct=T
\end{equation}
and analogously $\unitmap\counit\convoproduct T=T$.
\begin{defn}(Section I.3.2.\@ \cite{manchon06})
 Let $(C,\coproduct)$ be a coalgebra over the field $K$. Then, a \emph{right comodule}\index{comodule} is a $K$ vector space $\hat{C}$ together with a map $\hat{\coproduct}:\,\hat{C}\rightarrow C\etensor\hat{C}$ such that
 \begin{equation*}
  (\coproduct\etensor\id)\hat{\coproduct}=(\id\etensor\hat{\coproduct})\hat{\coproduct}.
 \end{equation*}

\end{defn}

\section{Bialgebras and Hopf algebras}
For an algebra $(A,\product)$ we introduce another algebra $(A\etensor A,\producttwo)$\nomenclature[ZZs product two]{$\producttwo$}{Canonical algebra product on $A\etensor A$ for some algebra $(A,\product)$} where the product is bilinearly generated by (Section I.2.2.\@ \cite{manchon06})
\begin{equation}\label{eq:producttwo}
{(a_1\etensor a_2)\producttwo(a_3\etensor a_4)}:={(a_1\product a_3)\etensor(a_2\product a_4)}.
\end{equation}
The corresponding linear map is given by (Section I.2.2.\@ \cite{manchon06})
\begin{equation*}
 \mproducttwo=(\mproduct\etensor\mproduct)\underbrace{(\id\etensor\flip\etensor\id)}_{=:\flip_{1324}}.
\end{equation*}

\begin{defn}\label{defn:bialgebra}(Section I.5.\@ \cite{manchon06}, Chapter 2 Section 1.1 \cite{abe80})
 Let $(B,\product)$ be an algebra over the field $K$ with unit element $\unit=\unitmap(1)$ and $(B,\coproduct)$ a coalgebra with counit $\counit$. If furthermore the \emphind{compatibility requirements}
 \begin{enumerate}
  \item $\coproduct\mproduct=(\mproduct\etensor\mproduct)(\id\etensor\flip\etensor\id)(\coproduct\etensor\coproduct)\quad$ ($\coproduct$ is an algebra homomorphism from $(B,\product)$ to $(B\etensor B,\producttwo)$ and $\product$ is a coalgebra homomorphism from $(B\etensor B,\flip_{1324}(\coproduct\etensor\coproduct))$ to $(B,\coproduct)$),
  \item $\coproduct\unitmap=\unitmap\etensor\unitmap\quad$ ($\unitmap$ is a coalgebra homomorphism from $(K,\id_K)$ to $(B,\coproduct)$),
  \item $\counit\mproduct=\counit\etensor\counit\quad$ ($\counit$ is an algebra homomorphism from $(B,\product)$ to $(K,\cdot)$),
 \end{enumerate}
 hold, we call $(B,\product,\coproduct)$ a \emphind{bialgebra}. For bialgebras $(B_1,\productnum{2}_1,\coproduct_1)$ with unit map $\unitmap_1$ and counit $\counit$, $(B_2,\productnum{2}_2,\coproduct_2)$ with unit map $\unitmap_1$ and counit $\counit$,  we call $\varLambda:\,B_1\rightarrow B_2$ a \emph{bialgebra homomorphism}\index{bialgebra!homomorphism}\index{homomorphism!bialgebra@\textit{bialgebra}} if $\varLambda$ is both an algebra homomorphism and a coalgebra homomorphism and fulfills the additional properties $\varLambda\unitmap_1=\unitmap_2$ and $\counit_2\varLambda=\counit_1$.
\end{defn}
\begin{rmk}\label{rmk:comp_req}
 Applied to elements of $B$, the compatibility requirements read
 \begin{enumerate}
  \item ${\coproduct (b_1\product b_2)={\coproduct b_1}\producttwo{\coproduct b_2}\quad\forallc b_1,b_2\in B}$,
  \item $\coproduct\unit=\unit\etensor\unit$,
  \item $\counit(b_1\product b_2)=\counit(b_1)\counit(b_2)\quad\forallc b_1,b_2\in B$.
 \end{enumerate}
\end{rmk}
\begin{rmk}\label{rmk:counit_unit}
 In a bialgebra, we always have $\counit(\unit)=1$, as the counit property \eqref{eq:counit_property} implies
 \begin{equation*}
  \unit=(\id\etensor\counit)\coproduct\unit=(\id\etensor\counit)(\unit\etensor\unit)=\counit(\unit)\unit.
 \end{equation*}
 Thus, $\counit\unitmap=\id_K$ and therefore $\counit=\counit\unitmap\counit$.
\end{rmk}
We introduce the following simple lemma to avoid having to show this separately in each of the examples of Hopf algebras we later want to look at.
\begin{lemma}\label{lemma:coassociative_subset} 
 Let $(A,\product)$ be an algebra and $\coproduct:\,A\rightarrow A\etensor A$ an algebra homomorphism. Assume that      
 \begin{equation}\label{eq:coassociative_subset}
  (\coproduct\etensor\id)\coproduct v=(\id\etensor\coproduct)\coproduct v
 \end{equation}
 holds for all $v\in M\subseteq A$. Then, we have \eqref{eq:coassociative_subset} for all $v$ in the subalgebra generated by $M$. If furthermore $\counit\in\adual{A}$ is an algebra homomorphism and
 \begin{equation}\label{eq:counit_subset}
  (\counit\etensor\id)\coproduct v=(\id\etensor\counit)\coproduct v=v
 \end{equation}
 holds for all $v\in M$, we even have \eqref{eq:counit_subset} for all $v$ in the subalgebra generated by $M$.
\end{lemma}
\begin{proof}
 Let $S$ denote the subalgebra generated by $M$. As $\coproduct$ is an algebra homomorphism, we have for all $v,w$ in $S$ that
 \begin{equation*}
  \coproduct(v\product w)=\sum_{(v)}\sum_{(w)}(v_1\product w_1)\etensor(v_2\product w_2)
 \end{equation*}
  and therefore, if $v$ and $w$ satisfy \eqref{eq:coassociative_subset}, we get, using the notation from \eqref{eq:coassociativity_sweedler},
 \begin{align*}
  (\coproduct\etensor\id)\coproduct(v\product w)&=\sum_{(v)}\sum_{(v_1)}\sum_{(w)}\sum_{(w_1)}(v_{1:1}\product w_{1:1})\etensor(v_{1:2}\product w_{1:2})\etensor(v_2\product w_2)\\
  &=\sum_{(v)}\sum_{(v_2)}\sum_{(w)}\sum_{(w_1)}(v_1\product w_{1:1})\etensor(v_{2:1}\product w_{1:2})\etensor(v_{2:2}\product w_2)\\
  &=\sum_{(v)}\sum_{(v_2)}\sum_{(w)}\sum_{(w_2)}(v_1\product w_1)\etensor(v_{2:1}\product w_{2:1})\etensor(v_{2:2}\product w_{2:2})\\
  &=(\id\etensor\coproduct)\coproduct(v\product w).
 \end{align*}
 As the equation \eqref{eq:coassociative_subset} is furthermore linear in $v$, which follows from $\coproduct$ being an algebra homomorphism, coassociativity indeed inductively extends from $M$ to the whole of $S$.
 
 \noindent If also $\counit$ is an algebra morphism, the counit property given for some $v,w\in S$ implies
 \begin{align*}
  (\counit\etensor\id)\coproduct(v\product w)&=\sum_{(v)}\sum_{(w)}\counit(v_1\product w_1)(v_2\product w_2)=\sum_{(v)}\sum_{(w)}(\counit(v_1)v_2)\product (\counit(w_1)w_2)\\&=\Big(\sum_{(v)}\counit(v_1)v_2\Big)\product\Big(\sum_{(w)}\counit(w_1)w_2\Big)=v\product w.
 \end{align*}
 Analogously, we get $(\id\etensor\counit)\coproduct(v\product w)=v\product w$. Hence, also the counit property inductively extends from $M$ to $S$.\\
\end{proof}

For a bialgebra $(B,\product,\coproduct)$ with unit element $\unit$, let 
\begin{equation*}
 \sum_{(h)}h'\etensor h'':=\rcoproduct h:=\coproduct h-\unit\etensor h-h\etensor\unit
\end{equation*}
\nomenclature[00g D tilde]{$\rcoproduct$}{The reduced coproduct for some coproduct $\coproduct$} denote the \emphind{reduced coproduct}\index{coproduct!reduced@\textit{reduced}} of $h$ (Proof of Proposition II.1.1.\@ \cite{manchon06}, Section 2.2 \cite{hairerkelly14}).
\begin{lemma}\textnormal{(Part of Proposition II.1.1.\@ \cite{manchon06})}
 $(B,\rcoproduct)$ is a coalgebra.
\end{lemma}
\begin{proof}(Part of the proof of Proposition II.1.1.\@ \cite{manchon06})
 As $\coproduct$ maps $B$ linearly to $B\etensor B$, the same is true for $\rcoproduct$. Thus, according to Definition \ref{defn:coalgebra}, the only thing left to show is coassociativity. By Definition \ref{defn:bialgebra} of a bialgebra we have $\coproduct\unitmap=\unitmap\etensor\unitmap$, implying
 \begin{equation}\label{eq:rcoproduct_unitmap}
  \rcoproduct\unitmap=-\unitmap\etensor\unitmap.
 \end{equation}
 Therefore, we compute
 \begin{multline*}
  (\rcoproduct\etensor\id)\rcoproduct=(\rcoproduct\etensor\id)(\coproduct-\unitmap\etensor\id-\id\etensor\unitmap)\eqrefequal{eq:rcoproduct_unitmap}(\rcoproduct\etensor\id)\coproduct+\unitmap\etensor\unitmap\etensor\id-\rcoproduct\etensor\unitmap\\
  \begin{aligned}
&=(\coproduct\etensor\id)\coproduct+\unitmap\etensor\unitmap\etensor\id+\unitmap\etensor\id\etensor\unitmap+\id\etensor\unitmap\etensor\unitmap-\coproduct\etensor\unitmap-(\id\etensor\unitmap\etensor\id)\coproduct-\unitmap\etensor\coproduct\\
   &\eqrefequal{eq:coassociativity}(\id\etensor\coproduct)\coproduct+\unitmap\etensor\unitmap\etensor\id+\unitmap\etensor\id\etensor\unitmap+\id\etensor\unitmap\etensor\unitmap-\coproduct\etensor\unitmap-(\id\etensor\unitmap\etensor\id)\coproduct-\unitmap\etensor\coproduct\\
   &=(\id\etensor\rcoproduct)\coproduct+\id\etensor\unitmap\etensor\unitmap-\unitmap\etensor\rcoproduct\eqrefequal{eq:rcoproduct_unitmap}(\id\etensor\rcoproduct)\rcoproduct.
  \end{aligned}
 \end{multline*}
\end{proof}
Having shown that $\rcoproduct$ is a coproduct, we may introduce the convolution product algebra $(\linear(B,B),\rconvoproduct)$ by 
\begin{equation}\label{eq:defn_rconvoproduct}
S\rconvoproduct T:=\mproduct(S\etensor T)\rcoproduct.
\end{equation}
\begin{defn}(Section I.5.\@ \cite{manchon06})
 A bialgebra $(H,\product,\coproduct)$ with unit map $\unitmap$ and counit $\counit$ together with an \emphind{antipode}, \ie a linear operator $\antipode:\,H\rightarrow H$\nomenclature[S]{$\antipode$}{Antipode of an Hopf algebra} fulfilling the \emph{antipode property}\index{antipode!property}
 \begin{equation}\label{eq:antipodeproperty}
  \mproduct(\antipode\etensor\id)\coproduct=\mproduct(\id\etensor\antipode)\coproduct=\underbrace{\unitmap\counit}_{\mathclap{=\counit(\cdotp)\unit}}\,,
 \end{equation}
 is called a \emphind{Hopf algebra}.
\end{defn}

Although \cite{abe80} gives the definition of a Hopf algebra homomorphism (morphism in the language used there), the following simple fact we came across is not mentioned, but it is for example stated as Exercise 9.\@ of Section III.8 \cite{kassel95}. The idea for the proof came from a similar argumentation in the proof of Proposition I.7.1.\@ \cite{manchon06}.
\begin{thm}\label{thm:hopf_morphism}
 Let $(H_1,\productnum{1},\coproduct_1,\antipode_1)$ and $(H_2,\productnum{2},\coproduct_2,\antipode_2)$ be Hopf algebras and $\varLambda:\,H_1\rightarrow H_2$ be a bialgebra homomorphism. Then,
 \begin{equation*}
  \varLambda\antipode_1=\antipode_2\varLambda.
 \end{equation*}
 Therefore, we call $\varLambda$ a \emph{Hopf algebra homomorphism}\index{Hopf algebra!homomorphism}\index{homomorphism!Hopf algebra@\textit{Hopf algebra}} \textnormal{(Chapter 2 Section 1.2 \cite{abe80})}.
\end{thm}
\begin{proof}
 Consider the algebra $(\linear(H_1,H_2),\convoproduct)$ defined by $S\convoproduct T:=\mproductnum{2}(S\etensor T)\coproduct_1$ with unit $\unitmap_2\counit_1$ (see \eqref{eq:defn_convoproduct} and \eqref{eq:unit_convoproduct}), where $\unitmap_i$ is the unit map and $\counit_i$ is the counit of $H_i$. We have
 \begin{align*}
  \varLambda\antipode_1\convoproduct\varLambda&=\mproductnum{2}(\varLambda\antipode_1\etensor\varLambda)\coproduct_1=\mproductnum{2}(\varLambda\etensor\varLambda)(\antipode_1\etensor\id)\coproduct_1=\varLambda\mproductnum{1}(\antipode_1\etensor\id)\coproduct_1=\varLambda\unitmap_1\counit_1=\unitmap_2\counit_1\\
  &=\unitmap_2\counit_2\varLambda=\mproductnum{2}(\id\etensor\antipode_2)\coproduct_2\varLambda=\mproductnum{2}(\id\etensor\antipode_2)(\varLambda\etensor\varLambda)\coproduct_1=\mproductnum{2}(\varLambda\etensor\antipode_2\varLambda)\coproduct_1=\varLambda\convoproduct\antipode_2\varLambda
 \end{align*}
 and therefore
 \begin{equation}\label{eq:varLambda_antipode_computation}
  \varLambda\antipode_1=\varLambda\antipode_1\convoproduct\unitmap_2\counit_1=\varLambda\antipode_1\convoproduct(\varLambda\convoproduct\antipode_2\varLambda)=(\varLambda\antipode_1\convoproduct\varLambda)\convoproduct\antipode_2\varLambda=\unitmap_2\counit_1\convoproduct\antipode_2\varLambda=\antipode_2\varLambda.
 \end{equation}

\end{proof}
\begin{cor}\textnormal{(Section 2.1 \cite{hairerkelly14})}
 If there is an antipode for a given bialgebra, then it is unique.
\end{cor}
\begin{proof}
 If for a bialgebra $(B_1,\product,\coproduct)$, there are antipodes $\antipode_1$ and $\antipode_2$, then both $(B_1,\product,\coproduct,\antipode_1)$ and $(B_2,\product,\coproduct,\antipode_2)$ are Hopf algebras, thus we have $\antipode_1$=$\antipode_2$ via Theorem \ref{thm:hopf_morphism} using the trivial bialgebra homomorphism $\id$.
\end{proof}

\begin{thm}\label{thm:antipode_antimorphism}\textnormal{(Proposition 4.0.1.\@ \cite{sweedler69}, Theorem III.3.4.\@ \cite{kassel95}, Proposition I.7.1.\@ \cite{manchon06})}
 Let $(H,\product,\coproduct,\antipode)$ be a Hopf algebra with unit element $\unit$, unit map $\unitmap$ and counit $\counit$. The antipode $\antipode$ is
 \begin{enumerate}[(i)]
  \item an algebra antimorphism, \ie
  \begin{equation}\label{eq:antipode_mproduct}
   \antipode\mproduct=\mproduct(\antipode\etensor\antipode)\flip,
  \end{equation}
  where $\flip$ denotes the flip, \resp
  \begin{equation*}
   \antipode(h_1\product h_2)=\antipode h_2\product\antipode h_1\quad\forallc h_1,h_2\in H.
  \end{equation*}
  Furthermore, $\antipode\unitmap=\unitmap$, \resp $\antipode\unit=\unit$.
  \item a coalgebra antimorphism, \ie
  \begin{equation}\label{eq:coproduct_antipode}
   \coproduct\antipode=\flip(\antipode\etensor\antipode)\coproduct.
  \end{equation}
  Furthermore, $\counit\antipode=\counit$.
 \end{enumerate}

\end{thm}
\begin{proof}(Based on the proofs of Proposition 4.0.1.\@ \cite{sweedler69}, of Theorem III.3.4.\@ \cite{kassel95} and of Proposition I.7.1.\@ and Lemma I.7.2.\@ \cite{manchon06})
 First of all, using associativity \eqref{eq:associativity}, we have
 \begin{equation}\label{eq:triassociativity}
  \begin{aligned}
  \mproduct(\mproduct\etensor\mproduct)&=\mproduct(\mproduct\etensor\id)(\id\etensor\id\etensor\mproduct)=\mproduct(\id\etensor\mproduct)(\id\etensor\id\etensor\mproduct)\\
  &=\mproduct(\id\etensor\mproduct)(\id\etensor\mproduct\etensor\id).
 \end{aligned}
 \end{equation}
 Likewise, using coassociativity \eqref{eq:coassociativity},
 \begin{equation}\label{eq:tricoassociativity}
  \begin{aligned}
  (\coproduct\etensor\coproduct)\coproduct&=(\id\etensor\id\etensor\coproduct)(\coproduct\etensor\id)\coproduct=(\id\etensor\id\etensor\coproduct)(\id\etensor\coproduct)\coproduct\\
  &=(\id\etensor\coproduct\etensor\id)(\id\etensor\coproduct)\coproduct.
  \end{aligned}
 \end{equation}
 Put $\coproducttwo:=(\id\etensor\flip\etensor\id)(\coproduct\etensor\coproduct)$. Furthermore, for any permutation $(i_1\ldots i_n)$, let $\flip_{i_1\ldots i_n}:\,H^{\etensor n}\rightarrow H^{\etensor n}$ be the linear map generated by (Section I.1.\@ \cite{manchon06})
 \begin{equation*}
  \flip_{i_1\ldots i_n}(h_{i_1}\etensor\cdots\etensor h_{i_n}):=h_{1}\etensor\cdots\etensor h_{n}.
 \end{equation*}

 \begin{enumerate}[(i)]
 \item Consider the algebra $(\linear(H\etensor H,H),\convoproduct)$ defined by $S\convoproduct T:=\mproduct(S\etensor T)\coproducttwo$ with unit $\unitmap(\counit\etensor\counit)$. Using the antipode property \eqref{eq:antipodeproperty} and the unit property \eqref{eq:unit_property}, we get
 \begin{align*}
  \mproduct\convoproduct\mproduct(\antipode\etensor\antipode)\flip&=\mproduct(\mproduct\etensor\mproduct(\antipode\etensor\antipode)\flip)\coproducttwo=\mproduct(\mproduct\etensor\mproduct(\antipode\etensor\antipode)\flip)\flip_{1423}(\coproduct\etensor\coproduct)\\
  &\eqrefequal{eq:triassociativity}\mproduct(\id\etensor\mproduct)(\id\etensor\mproduct\etensor\id)(\id\etensor\id\etensor\antipode\etensor\antipode)\flip_{1342}(\coproduct\etensor\coproduct)\\
  &=\mproduct(\id\etensor\mproduct)\flip_{132}(\id\etensor\id\etensor\mproduct)(\id\etensor\antipode\etensor\id\etensor\antipode)(\coproduct\etensor\coproduct)\\
  &\eqrefequal{eq:antipodeproperty}\mproduct(\id\etensor\mproduct)\flip_{132}(\id\etensor\antipode\etensor\unitmap\counit)(\coproduct\etensor\id)\\
  &\eqrefequal{eq:unit_property}\mproduct(\id\etensor\antipode)(\coproduct\etensor\counit)\eqrefequal{eq:antipodeproperty}\unitmap\counit\etensor\counit
  =\unitmap(\counit\etensor\counit).
 \end{align*}
 On the other hand, using the antipode property as well as the first and third compatibility requirement in Definition \ref{defn:bialgebra}, we have
 \begin{equation*}
  \antipode\mproduct\convoproduct\mproduct=\mproduct(\antipode\mproduct\etensor\mproduct)\coproducttwo=\mproduct(\antipode\etensor\id)(\mproduct\etensor\mproduct)\coproducttwo=\mproduct(\antipode\etensor\id)\coproduct\mproduct=\unitmap\counit\mproduct=\unitmap(\counit\etensor\counit).
 \end{equation*}
 Then, equation \eqref{eq:antipode_mproduct} follows completely analogous to the calculation in \eqref{eq:varLambda_antipode_computation}.
 
 \noindent By Remark \ref{rmk:counit_unit}, the antipode property, the second compatibility requirement and the unit property we also have
 \begin{equation*}
  \unit=\unitmap\counit\unit=\mproduct(\antipode\etensor\id)\coproduct\unit=\antipode\unit\product\unit=\antipode\unit.
 \end{equation*}
 \item Consider the algebra $(\linear(H,H\etensor H),\convoproduct)$ defined by $S\convoproduct T:=\mproducttwo(S\etensor T)\coproduct$ with unit $(\unitmap\etensor\unitmap)\counit$. Using the antipode property \eqref{eq:antipodeproperty} and the counit property \eqref{eq:unit_property}, we get
 \begin{align*}
  \coproduct\convoproduct\flip(\antipode\etensor\antipode)\coproduct&=\mproducttwo(\coproduct\etensor\flip(\antipode\etensor\antipode)\coproduct)\coproduct=(\mproduct\etensor\mproduct)\flip_{1324}(\coproduct\etensor\flip(\antipode\etensor\antipode)\coproduct)\coproduct\\
  &\eqrefequal{eq:tricoassociativity}(\mproduct\etensor\mproduct)\flip_{1342}(\id\etensor\id\etensor\antipode\etensor\antipode)(\id\etensor\coproduct\etensor\id)(\id\etensor\coproduct)\coproduct\\
  &=(\mproduct\etensor\mproduct)(\id\etensor\antipode\etensor\id\etensor\antipode)(\id\etensor\id\etensor\coproduct)\flip_{132}(\id\etensor\coproduct)\coproduct\\
  &\eqrefequal{eq:antipodeproperty}(\mproduct\etensor\id)(\id\etensor\antipode\etensor\unitmap\counit)\flip_{132}(\id\etensor\coproduct)\coproduct\\
  &\eqrefequal{eq:counit_property}(\mproduct\etensor\unitmap)(\id\etensor\antipode)\coproduct\eqrefequal{eq:antipodeproperty}\unitmap\counit\etensor\unitmap=(\unitmap\etensor\unitmap)\counit.
 \end{align*}
 On the other hand, using the antipode property as well as the first and second compatibility requirement in Definition \ref{defn:bialgebra}, we have
 \begin{equation*}
  \coproduct\antipode\convoproduct\coproduct=\mproducttwo(\coproduct\antipode\etensor\coproduct)\coproduct=\mproducttwo(\coproduct\etensor\coproduct)(\antipode\etensor\id)\coproduct=\coproduct\mproduct(\antipode\etensor\id)\coproduct=\coproduct\unitmap\counit=(\unitmap\etensor\unitmap)\counit.
 \end{equation*}
 Then, equation \eqref{eq:coproduct_antipode} follows completely analogous to the calculation in \eqref{eq:varLambda_antipode_computation}.
 
 \noindent By Remark \ref{rmk:counit_unit}, the antipode property, the third compatibility requirement and the counit property we also have
 \begin{equation*}
  \counit=\counit\unitmap\counit=\counit\mproduct(\antipode\etensor\id)\coproduct=(\counit\etensor\counit)(\antipode\etensor\id)\coproduct=\counit\antipode(\id\etensor\counit)\coproduct=\counit\antipode.
 \end{equation*}
 \end{enumerate}

\end{proof}

\section{Gradings}
\begin{defn}\label{defn:grading}(Chapter 1 Section 2.2 and Chapter 2 Section 4.1 \cite{abe80}, Section II.1.\@ \cite{manchon06})Let $G_i,\,i\in\naturals_0$ be vector spaces over the same field and put $G:=\bigoplus_{i=0}^{\infty}G_i$. Then,
 \begin{enumerate}[(i)]
  \item an algebra $(G,\product)$ is \emph{graded}\index{algebra!graded@\textit{graded}} if 
  \begin{equation*}
   G_n\product G_m\subseteq G_{n+m}\quad\forallc n,m\in\naturals_0,
  \end{equation*}
  \item a coalgebra $(G,\coproduct)$ is \emph{graded}\index{coalgebra!graded@\textit{graded}} if 
  \begin{equation*}
   \coproduct G_n\subseteq\bigoplus_{m\leq n}G_m\etensor G_{n-m}\quad\forallc n\in\naturals_0,
  \end{equation*}
  \item a bialgebra $(G,\product,\coproduct)$ is \emph{graded}\index{bialgebra!graded@\textit{graded}} if both $(G,\product)$ and $(G,\coproduct)$ are graded,
  \item a Hopf algebra $(G,\product,\coproduct,\antipode)$ is \emph{graded}\index{Hopf algebra!graded@\textit{graded}} if $(G,\product,\coproduct)$ is graded and 
  \begin{equation*}
   \antipode G_n\subseteq G_n\;\forallc n\in\naturals_0.
  \end{equation*}
 \end{enumerate}
 The family $(G_i)_{i\in\naturals_0}$ is called \emphind{grading} of $G$.
 The algebra, coalgebra, bialgebra or Hopf algebra is \emph{connected graded}\index{connected grading} if additionally $\dim G_0=1$.
\end{defn}
If $(G_i)_{i\in\naturals_0}$ is a grading, the family $(G^i)_{i\in\naturals_0},\,G^i:=\bigoplus_{j=0}^{i}G_i$
is a so-called \emphind{filtration} (Chapter 1 Section~2.2 and Chapter 2 Section 4.1 \cite{abe80}, Section~II.2.\@ \cite{manchon06}).
\begin{thm}\label{thm:rcoproduct_cgbialgebra}\textnormal{(Proposition II.1.1.\@ \cite{manchon06})}
 In a connected graded bialgebra $(B,\product,\coproduct)$ graded by $(B_i)_i$ we have $\ker\counit=\bigoplus_{i=1}^{\infty}B_i$ and
 \begin{equation}\label{eq:rcoproduct_grading}
  \rcoproduct B_n\subseteq\bigoplus_{0<m<n}B_m\etensor B_{n-m}\;\forallc n\in\naturals.
 \end{equation}
\end{thm}
\begin{proof}(Proof of Proposition II.1.1.\@ \cite{manchon06}) 
 Let $K$ be the field under $B$. Let $n\in\naturals$ and $x\in B_n$ be arbitrary. From Definition \ref{defn:grading} we know that $\coproduct x\in\bigoplus_{m\leq n}B_m\etensor B_{n-m}$. As furthermore $B_0=\vspan\{\unit\}$ due to connectedness, there are $y,z\in B_n$ and $w\in\bigoplus_{0<m<n}B_m\etensor B_{n-m}$ such that
 \begin{equation*}
  \coproduct x=\unit\etensor y+z\etensor\unit+w.
 \end{equation*}
 By the counit property \eqref{eq:counit_property} and $\counit(\unit)=1$ (see Remark \ref{rmk:counit_unit}), we have
 \begin{equation*}
  x=(\id\etensor\counit)\coproduct x=\counit(y)\unit+z+(\id\etensor\counit) w.
 \end{equation*}
  Due to linear independence of $z-x\in B_n$, $\counit(y)\unit\in B_0$ and $(\id\etensor\counit) w\in\bigoplus_{0<m<n}B_m$ this implies $x=z$, $\counit(y)=0$ and $(\id\etensor\counit) w=0$. Again applying the counit property, we get, using what we just found out,
  \begin{equation*}
   x=(\counit\etensor\id)\coproduct x=\counit(x)\unit+y+(\counit\etensor\id)w.
  \end{equation*}
  As before, linear independence of $y-x\in B_n$, $\counit(x)\unit\in B_0$ and $(\counit\etensor\id)w\in\bigoplus_{0<m<n}B_m$ implies $x=y$, $\counit(x)=0$ and $(\counit\etensor\id)w=0$.
  Hence, we conclude
  \begin{equation*}
   \coproduct x=\unit\etensor x+x\etensor\unit+w,
  \end{equation*}
  implying $\rcoproduct x=w\in\bigoplus_{0<m<n}B_m\etensor B_{n-m}$, as well as $\counit(x)=0$. As $n\in\naturals$ and $x\in B_n$ where arbitrary, this implies \eqref{eq:rcoproduct_grading} and, together with $\counit(\unit)=1$, also $\ker\counit=\bigoplus_{i=1}^{\infty}B_i$.
\end{proof}
\begin{rmk}\label{rmk:coproduct_gradeone}
 Due to Theorem \ref{thm:rcoproduct_cgbialgebra} we especially have
 \begin{equation*}
  \coproduct b_1=\unit\etensor b_1+b_1\etensor\unit\quad\forallc b_1\in B_1.
 \end{equation*}
\end{rmk}
The following theorem will be very important for some of our examples of Hopf algebraslater. We chose a different approach for the proof than the one given in \cite{manchon06}. For a function $f:\,M\rightarrow N$ and $M'\subseteq M$, let $f{\restriction_{M'}}$ denote the restriction of $f$ to $M'$.
\begin{thm}\label{thm:antipode_recursion}\textnormal{(Corollary II.3.2.\@ \cite{manchon06})}\index{antipode!recursion formula}
 Any connected graded bialgebra $(H,\product,\coproduct)$ over the field $K$ is a connected graded Hopf algebra with the antipode given by the recursion
 \begin{equation}\label{eq:recursion_antipode_right}
 \antipode h=-h-\mproduct(\id\etensor\antipode)\rcoproduct h=-h-\sum_{(h)}h'\product\antipode h''\quad\forallc h\in\ker\counit,\quad\antipode k\unit=k\unit\quad\forallc k\in K,
 \end{equation}
 or equivalently
 \begin{equation}\label{eq:recursion_antipode_left}
 \antipode h=-h-\mproduct(\antipode\etensor\id)\rcoproduct h=-h-\sum_{(h)}\antipode h'\product h''\quad\forallc h\in\ker\counit,\quad\antipode k\unit=k\unit\quad\forallc k\in K.
 \end{equation}
\end{thm}
\begin{proof}~
 \providecommand{\restrH}[1]{{\restriction_{\tilde{H}^{#1}}}}
 Let $(H_i)_i$ be the grading of $H$, $(H^i)_i$ the filtration and define the family $(\tilde{H}^i)_{i\in\naturals}$ by ${\tilde{H}^i:=\bigoplus_{j=1}^i H_j}$.
 \begin{enumerate}
  \item \emph{The individual recursions are well-defined.} We only look at \eqref{eq:recursion_antipode_right}, the other case is completely analogous. We proceed by induction. First of all, $\antipode{\restriction_{H_0}}$ is obviously a well-defined linear operator as $H$ is connected. Assume $\antipode{\restriction_{H^n}}$ is a well-defined linear operator. As \eqref{eq:rcoproduct_grading} implies $\rcoproduct H^{n+1}\subseteq\tilde{H}^n\etensor\tilde{H}^n$, the operator
  \begin{equation*}
   \antipode{\restriction_{H^{n+1}}}:=-\id{\restriction_{H^{n+1}}}-\mproduct(\id\etensor\antipode{\restriction_{H^n}})\rcoproduct{\restriction_{H^{n+1}}}
  \end{equation*}
  is well-defined in accordance with \eqref{eq:recursion_antipode_right} and linear due to the fact that it is a sum of compositions of linear operators.
  \item \emph{Both recursions yield the same results.} Let $\antipode_1$ be the linear operator obtained by \eqref{eq:recursion_antipode_right} and $\antipode_2$ the one obtained by \eqref{eq:recursion_antipode_left}. Again, we proceed by induction. As $H$ is connected, $H_0=\vspan\{\unit\}$ and hence $\antipode_1{\restriction_{H_0}}=\antipode_2{\restriction_{H_0}}$ by definition. Assume 
  \begin{equation}\label{eq:recursion_antipode_IA}
   \antipode_1{\restriction_{H^n}}=\antipode_2{\restriction_{H^n}}
  \end{equation}
  for some $n\in\naturals_0$. Then,
  \begin{align*}
   (\id\rconvoproduct\antipode_1)\restrH{n+1}&\diamondequal(\id\rconvoproduct\antipode_2)\restrH{n+1}=-(\id\rconvoproduct\id)\restrH{n+1}-(\id\rconvoproduct\antipode_2\rconvoproduct\id)\restrH{n+1}\\
   &\diamondequal-(\id\rconvoproduct\id)\restrH{n+1}-(\id\rconvoproduct\antipode_1\rconvoproduct\id)\restrH{n+1}=(\antipode_1\rconvoproduct\id)\restrH{n+1}\diamondequal(\antipode_2\rconvoproduct\id)\restrH{n+1}
  \end{align*}
  where we used \eqref{eq:recursion_antipode_IA} and the definition of $\rconvoproduct$ \eqref{eq:defn_rconvoproduct} combined with $\rcoproduct H^{n+1}\subseteq\tilde{H}^n\etensor\tilde{H}^n$ for the equalities labeled $(\diamondsuit)$. We conclude $\antipode_1{\restriction_{H^{n+1}}}=\antipode_2{\restriction_{H^{n+1}}}$.
  \item \emph{Both recursions yield the antipode.} Now, let $\antipode:=\antipode_1=\antipode_2$ be the operator given by each of the recursions. For $h\in\ker\counit$, we have
  \begin{equation*}
   \mproduct(\antipode\etensor\id)\coproduct h=\mproduct(\antipode\etensor\id)\rcoproduct h+\mproduct(\antipode\etensor\id)(\unit\etensor h+h\etensor\unit)=\mproduct(\antipode\etensor\id)\rcoproduct h+h+\antipode h\stackrel{\mathclap{\text{\eqref{eq:recursion_antipode_left}}}}{=}0
  \end{equation*}
  and analogously $\mproduct(\id\etensor\antipode)\coproduct h=0$ by \eqref{eq:recursion_antipode_left}. Also, $\unitmap\counit(h)=0$. Hence, the antipode property \eqref{eq:antipodeproperty} is fulfilled on $\ker\counit$. Finally, due to the fact that $\coproduct\unit=\unit\etensor\unit$ in a bialgebra, we get
  \begin{equation*}
   \mproduct(\antipode\etensor\id)\coproduct\unit=\mproduct(\antipode\unit\etensor\unit)=\unit=\mproduct(\id\etensor\antipode)\coproduct\unit,
  \end{equation*}
  just as $\unitmap\counit(\unit)=\unit$. By linearity, the antipode property is fulfilled on the whole of $H$.
 \end{enumerate}
 Finally, we show by yet another induction that $\antipode$ preserves the grading. First of all, we obviously have $\antipode H_0=H_0$. Assuming $\antipode H_m\subseteq H_m$ for all $m<n$, we get $\antipode H_n\subseteq H_n$ by \eqref{eq:recursion_antipode_right} and
 \begin{equation*}
  \mproduct(\id\etensor\antipode)\rcoproduct H_n\stackrel{\mathclap{\eqref{eq:rcoproduct_grading}}}{\subseteq}\mproduct(\id\etensor\antipode)\bigoplus_{0<m<n}H_m\etensor H_{n-m}\subseteq\mproduct\bigoplus_{0<m<n}H_m\etensor H_{n-m}\subseteq H_n
 \end{equation*}
 since $(H,\product,\coproduct)$ was assumed to be a graded bialgebra.\\
\end{proof}

\section{Pairs of dual vector spaces}
\begin{defn}\label{defn:dualbracket}(Chapter II Section 3.\@ \cite{robertson64}, Chapter 2 Section 2.1 \cite{abe80})
 Let $V,\Vpdual$ be two vector spaces over the same field $K$ and $\dualbracket:\,V\times\Vpdual\rightarrow K$ bilinear such that
 \begin{enumerate}
  \item\label{item:dualbracket_vzero} $\langle v,\vpdual\rangle=0\;\forallc\vpdual\in\Vpdual\implies v=0$,
  \item\label{item:dualbracket_vtildezero} $\langle v,\vpdual\rangle=0\;\forallc v\in V\implies \vpdual=0$.
 \end{enumerate}
 Then, we call $(V,\Vpdual)$ a \emphind{pair of dual vector spaces}\index{dual!pair} and $\dualbracket$\nomenclature[ZZs langle]{$\dualbracket$}{A duality pairing of two vector spaces} its \emphind{duality pairing}.
\end{defn}
\begin{lemma}\label{lemma:finite_dual}\textnormal{(Chapter II Section 5.\@ \cite{robertson64})}
 Let $(V,\Vpdual)$ be a pair of dual vector spaces and $V$ finite dimensional. Then, $\dim V=\dim\Vpdual$ and for each basis $(b_i)_i$ of $V$ there is a \emph{dual basis}\index{dual!basis} $(\pdual{b}_i)_i$ of $\Vpdual$, \ie ${\langle b_i,\pdual{b}_j\rangle=\deltasymb{i,j}}$.
\end{lemma}
\begin{proof}(Based on Chapter II Section 5.\@ \cite{robertson64})
 For each $\vpdual\in\Vpdual$ let $\iota(\vpdual):=\langle\,\cdotp,\vpdual\rangle\in V'$. Obviously, $\iota$ is linear. If $\vpdual\neq0$, then there is $v\in V$ such that $\langle v,\vpdual\rangle\neq0$ as $\dualbracket$ is a duality pairing. Hence $\iota(\vpdual)\neq0$, \ie $\iota$ is injective. Therefore $\dim\Vpdual\leq\dim V'=\dim V$. Thus $\Vpdual$ is also finite dimensional and by the same argument as before we get $\dim V\leq\dim\Vpdual'=\dim\Vpdual$. As $\iota:\,\Vpdual\rightarrow V'$ is injective and $\dim\Vpdual=\dim V'$ we have that $\iota$ is bijective. Let now $(b_i)_i$ be a basis of $V$ and define $\kappa_i\in V'$ by $\kappa_i(b_j):=\deltasymb{i,j}$. Then $(\iota^{-1}(\kappa_i))_i$ is a dual basis of $(b_i)_i$.
\end{proof}

\begin{thm}\label{thm:tensordual}
 For $i\in\{1,2\}$, let $\dualbracket_i$ be a duality pairing of a $K$-vector space $V_i$ with another $K$-vector space $\Vpdual_i$, where $K$ is a field of characteristic zero. Then, $(V_1\etensor V_2,\Vpdual_1\etensor\Vpdual_2)$ is a pair of dual vector spaces with the \emphind{induced duality pairing}\index{duality pairing!induced@\textit{induced}}
 \begin{equation}\label{eq:defn_tensordual}
  \left\langle\sum_{j=1}^{m}v_1^j\etensor v_2^j,\sum_{l=1}^{n}\vpdual_1^l\etensor\vpdual_2^l\right\rangle:=\sum_{j=1}^{m}\sum_{l=1}^{n}\langle v_1^j,\vpdual_1^l\rangle_1\langle v_2^j,\vpdual_2^l\rangle_2.
 \end{equation}
\end{thm}
\begin{proof}~
 \begin{enumerate}
  \item \emph{The map $\dualbracket$ is well-defined and bilinear.} Let $m_i:\,V_i\etensor\Vpdual_i\rightarrow K$ be the linear maps corresponding to $\dualbracket_i$. Put ${m:=(m_1\etensor m_2)(\id\etensor\flip\etensor\id)}$. Then, $\dualbracket:=m(\,{\cdot}\,{\etensor}\,{\cdot}\,)$ yields the expansion \eqref{eq:defn_tensordual} and is obviously bilinear.
  \item \emph{The map $\dualbracket$ is a duality pairing.} Let $0\neq u=\sum_{j=1}^{m}u_1^j\etensor u_2^j\in V_1\etensor V_2$ be arbitrary. Without loss of generality, we may assume that 
  \begin{equation*}
   \dim\vspan(u_1^j)_{1\leq j\leq m}=\dim\vspan(u_2^j)_{1\leq j\leq m}=m,
  \end{equation*}
  because otherwise the number of summands may be reduced. Put $W_i:=\vspan(u_1^j)_{1\leq j\leq m}$,
  \begin{equation*}
   W_i^\perp:=\{\vpdual\in\Vpdual_i|\langle w,\vpdual\rangle_i=0\;\forallc w\in W_i\}
  \end{equation*}
  and $F_i:=\Vpdual_i\quotient W_i^\perp$. For $\vpdual\in\Vpdual_i$ let $[\vpdual]\in F_i$ denote its equivalence class. Hence, for $\vpdual,\vpdual'\in\Vpdual_i$,
  \begin{equation*}
   [\vpdual]=[\vpdual']\iff\langle w,\vpdual\rangle_i=\langle w,\vpdual'\rangle_i\;\forallc w\in W_i.
  \end{equation*}
  This means that $\langle w,[\vpdual]\rangle_i':=\langle w,\vpdual\rangle_i$ is well-defined. $\langle w,[\vpdual]\rangle_i=0$ for all $w\in W_i$ implies by definition $[\vpdual]=[0]=0$. Also, $\langle w,[\vpdual]\rangle_i'=0$ for all $[\vpdual]\in F_i$ means by definition $\langle w,\vpdual\rangle_i=0$ for all $\vpdual\in\Vpdual_i$ and therefore $w=0$ because $(V_i,\Vpdual_i)$ is a pair of dual vector spaces. Hence, $(W_i,F_i)$ is a pair of dual vector spaces and since $W_i$ is finite dimensional, we may use Lemma \ref{lemma:finite_dual} to get that there is a basis $([\pdual{u}_i^j])_{1\leq j\leq n}$ of $F_i$ such that 
  \begin{equation*}
   \deltasymb{j,l}=\langle u_i^j,[\pdual{u}_i^l]\rangle_i'=\langle u_i^j,\pdual{u}_i^l\rangle_i.
  \end{equation*}
  Put $\pdual{u}:=\sum_{j=1}^{m}u_1^j\etensor u_2^j$. Then,
  \begin{equation*}
   \langle u,\pdual{u}\rangle=\sum_{j,l=1}^{m}\langle u_1^j,\pdual{u}_1^l\rangle_1\langle u_2^j,\pdual{u}_2^l\rangle_2=m\neq 0.
  \end{equation*}
  As $0\neq u\in V_1\etensor V_2$ was arbitrary, we just showed implication \ref{item:dualbracket_vzero} of Definition \ref{defn:dualbracket}. Showing implication \ref{item:dualbracket_vtildezero} of the definition is then completely analogous.
  \end{enumerate}

\end{proof}
\begin{defn}
 Let $(V_i,\Vpdual_i)$, $i\in\{1,2\}$ be pairs of dual vector spaces with duality pairings $\varpi_i=\dualbracket_i$. We then say that a linear map $S:\,V_1\rightarrow V_2$ is the \emph{dual operator}\index{dual!operator} of a linear map $\pdual{S}:\,\Vpdual_2\rightarrow \Vpdual_1$ if
 \begin{equation*}
  \langle S v,\vpdual\rangle_2=\langle v,\pdual{S}\vpdual\rangle_1
 \end{equation*}
 for all $v\in V_1$ and $\vpdual\in\Vpdual_2$. We write $S\isdualvia{\varpi_1}{\varpi_2}\pdual{S}$, or shortly $S\isdual\pdual{S}$ if the involved duality pairings are clear from the context\footnote{This notation is nonstandard and just used for the purpose of this thesis.}.
\end{defn}
\begin{lemma}\label{lemma:dual_operator} Let $(V_i,\Vpdual_i)$ and $(W_j,\pdual{W}_j)$ be pairs of dual vector spaces over the same field $K$.
 \begin{enumerate}[(i)]
  \item\label{item:dual_operator_unique} If both $S_1:\,V_1\rightarrow V_2$ and $S_2:\,V_1\rightarrow V_2$ are dual operators of $\pdual{S}:\,\Vpdual_2\rightarrow\Vpdual_1$, then $S_1=S_2$. 
  \item\label{item:dual_operator_tensor} If $K$ is of characteristic zero, $A:\,V_1\rightarrow V_2$ is the dual operator of $\pdual{A}:\,\Vpdual_2\rightarrow\Vpdual_1$ and $B:\,W_1\rightarrow W_2$ is the dual operator of $\pdual{B}:\,\pdual{W}_2\rightarrow\pdual{W}_1$, then $A\etensor B$ is the dual operator of $\pdual{A}\etensor\pdual{B}$ under the induced duality pairing.
  \item\label{item:dual_operator_compo} If $F:\,V_1\rightarrow V_2$ is the dual operator of $\pdual{F}:\,\Vpdual_2\rightarrow\Vpdual_1$ and $G:\,V_2\rightarrow V_3$ is the dual operator of $\pdual{G}:\Vpdual_3\rightarrow\Vpdual_2$, then $G F$ is the dual operator of $\pdual{F}\pdual{G}$.
 \end{enumerate}
\end{lemma}
\begin{proof} Let $\dualbracket_{\nu_i}$ be the duality pairing of $(V_i,\Vpdual_i)$ and $\dualbracket_{\omega_j}$ the duality pairing of $(W_j,\pdual{W}_j)$.
 \begin{enumerate}[(i)]
  \item $\langle S_1 v,\vpdual\rangle_{\nu_2}=\langle v,\pdual{S}\vpdual\rangle_{\nu_1}=\langle S_2 v,\vpdual\rangle_{\nu_2}$ for all $v\in V_1$, $\vpdual\in\Vpdual_2$ implies $\langle(S_1-S_2)v,\vpdual\rangle_{\nu_2}=0$ for all $v\in V_1$, $\vpdual\in\Vpdual_2$, hence $(S_1-S_2)v=0$ for all $v\in V_1$ by the Definition \ref{defn:dualbracket} of dual pairs.
  \item Let $\dualbracket_{\nu_i\mu_i}$ be the induced duality pairing of $(V_i\etensor W_i,\pdual{V}_i\etensor\pdual{W}_i)$. Under the given assumption,
  \begin{align*}
   \langle (A\etensor B)(v_1\etensor w_1),\vpdual_2\etensor\pdual{w}_2\rangle_{\nu_2\omega_2}&=\langle A v_1,\vpdual_2\rangle_{\nu_2}\langle B w_1,\pdual{w}_2\rangle_{\omega_2}=\langle v_1,\pdual{A}\vpdual_2\rangle_{\nu_1}\langle w_1,\pdual{B}\pdual{w}_2\rangle_{\nu_1}\\
   &= \langle v_1\etensor w_1,(\pdual{A}\etensor\pdual{B})(\vpdual_2\etensor\pdual{w}_2)\rangle_{\nu_1\omega_1}
  \end{align*}
  for all $v_1\in V_1$, $w_1\in W_1$, $\vpdual_2\in\Vpdual_2$, $\pdual{w}_2\in\pdual{W}_2$. The implication follows via bilinearity.
  \item Under the given assumption,
  \begin{equation*}
   \langle G F v_1,\vpdual_3\rangle_3=\langle F v_1,\pdual{G}\vpdual_3\rangle_2=\langle v_1,\pdual{F}\pdual{G}\vpdual_3\rangle_1
  \end{equation*}
  for all $v_1\in V_1$ and $\vpdual_3\in\Vpdual_3$.
 \end{enumerate}
\end{proof}

\begin{samepage}
\begin{thm}\label{thm:dual_hopf_bi_co_algebras}\textnormal{(Equations (3.4.1) to (3.4.5) and Section 3.5 \cite{chen71}, also based on Chapter 2 Sections 2 and 3 \cite{abe80}, Proposition I.3.1.\@ \cite{manchon06})}
 Let $(V,\Vpdual)$ be a pair of dual vector spaces over the field $K$ of characteristic zero.
 \begin{enumerate}
  \item If $\copropro:\,\Vpdual\rightarrow\Vpdual\etensor\Vpdual$ is the dual operator of $\mproduct:\,V\etensor V\rightarrow V$, then $(\Vpdual,\copropro)$ is a coalgebra iff $(V,\product)$ is a algebra. In that case, we say that $(\Vpdual,\copropro)$ is the \emph{dual coalgebra}\index{dual!coalgebra} of $(V,\product)$ and $(V,\product)$ is the \emph{dual algebra}\index{dual!algebra} of $(\Vpdual,\copropro)$.
  \item If $(V,\product)$ is the dual algebra of $(\Vpdual,\copropro)$ and $\unitmap:\,K\rightarrow V$ is the dual operator of $\counit:\,\Vpdual\rightarrow K$, then $\unitmap$ is a unit map iff $\counit$ is a counit.
  \item If $(V,\productnum{1})$ is the dual unital algebra of $(\Vpdual,\copropronum{1})$ and $(V,\copropronum{2})$ is the dual counital coalgebra of $(\Vpdual,\productnum{2})$, then $(V,\productnum{1},\copropronum{2})$ is a bialgebra iff $(\Vpdual,\productnum{2},\copropronum{1})$ is a bialgebra. In that case, we say that $(V,\productnum{1},\copropronum{2})$ is the \emph{dual bialgebra}\index{dual!bialgebra} of $(\Vpdual,\productnum{2},\copropronum{1})$ and vice versa.
  \item If $(V,\productnum{1},\copropronum{2})$ is the dual bialgebra of $(\Vpdual,\productnum{2},\copropronum{1})$ and $\antipodenum{1}{2}:\,V\rightarrow V$ is the dual map of $\antipodenum{2}{1}:\,\Vpdual\rightarrow\Vpdual$, then $(\Vpdual,\productnum{1},\copropronum{2},\antipodenum{1}{2})$ is a Hopf algebra iff $(\Vpdual,\productnum{2},\copropronum{1},\antipodenum{2}{1})$ is a Hopf algebra. In that case, we say that $(\Vpdual,\productnum{1},\copropronum{2},\antipodenum{1}{2})$ is the \emph{dual Hopf algebra}\index{dual!Hopf algebra} of $(\Vpdual,\productnum{2},\copropronum{1},\antipodenum{2}{1})$ and vice versa.
 \end{enumerate}
\end{thm}
\end{samepage}
\begin{proof}~
 \begin{enumerate}
  \item Due to Lemma \ref{lemma:dual_operator} \ref{item:dual_operator_tensor} and \ref{item:dual_operator_compo},
  \begin{equation*}
   (\id\etensor\copropro)\copropro\isdual\mproduct(\id\etensor\mproduct)\quad\text{and}\quad(\copropro\etensor\id)\copropro\isdual\mproduct(\mproduct\etensor\id).
  \end{equation*}
  Hence, by Lemma \ref{lemma:dual_operator} \ref{item:dual_operator_unique}, coassociativity of $\copropro$ (\cf \eqref{eq:coassociativity}) and associativity of $\product$ (\cf \eqref{eq:associativity}) are equivalent.
  \item Due to Lemma \ref{lemma:dual_operator} \ref{item:dual_operator_tensor} and \ref{item:dual_operator_compo}, 
  \begin{equation*}
   (\id\etensor\counit)\copropro\isdual\mproduct(\id\etensor\unitmap)\quad\text{and}\quad(\counit\etensor\id)\copropro\isdual\mproduct(\unitmap\etensor\id).
  \end{equation*}
  Hence, by Lemma \ref{lemma:dual_operator} \ref{item:dual_operator_unique}, the counit property of $\counit$ \mbox{(\cf \eqref{eq:counit_property})} and the unit property of $\unitmap$ (\cf \eqref{eq:unit_property}) are equivalent.
  \item Due to Lemma \ref{lemma:dual_operator} \ref{item:dual_operator_tensor} and \ref{item:dual_operator_compo}, 
  \begin{equation*}
   \copropronum{1}\mproductnum{2}\isdual\copropronum{2}\mproductnum{1}\quad\text{and}\quad(\mproductnum{2}\etensor\mproductnum{2})(\id\etensor\flip\etensor\id)(\copropronum{1}\etensor\copropronum{1})\isdual(\mproductnum{1}\etensor\mproductnum{1})(\id\etensor\flip\etensor\id)(\copropronum{2}\etensor\copropronum{2})
  \end{equation*}
  as well as
  \begin{equation*}
   \copropronum{1}\unitmapnum{2}\isdual\counitnum{2}\mproductnum{1}\quad\text{and}\quad\unitmapnum{2}\etensor\unitmapnum{2}\isdual\counitnum{1}\etensor\counitnum{1}
  \end{equation*}
  as well as
  \begin{equation*}
   \counitnum{1}\mproductnum{2}\isdual\copropronum{2}\unitmapnum{1}\quad\text{and}\quad\counitnum{2}\etensor\counitnum{2}\isdual\unitmapnum{1}\etensor\unitmapnum{1}.
  \end{equation*}

  Hence, by Lemma \ref{lemma:dual_operator} \ref{item:dual_operator_unique}, the compatibility requirements (\cf definition \ref{defn:bialgebra}) of $\productnum{2},\copropronum{1},\unitmapnum{2},\counitnum{1}$ and those of $\productnum{1},\copropronum{2},\unitmapnum{1},\counitnum{2}$ are equivalent. 
  \item Due to Lemma \ref{lemma:dual_operator} \ref{item:dual_operator_tensor} and \ref{item:dual_operator_compo}, 
  \begin{equation*}
   \mproductnum{2}(\id\etensor\antipodenum{2}{1})\copropronum{1}\isdual\mproductnum{1}(\id\etensor\antipodenum{1}{2})\copropronum{2},\,\,\,\mproductnum{2}(\antipodenum{2}{1}\etensor\id)\copropronum{1}\isdual\mproductnum{1}(\antipodenum{1}{2}\etensor\id)\copropronum{2}\,\,\,\text{and}\,\,\,\unitmapnum{2}\counitnum{1}\isdual\unitmapnum{1}\counitnum{2}.
  \end{equation*}
  Hence, by Lemma \ref{lemma:dual_operator} \ref{item:dual_operator_unique}, the antipode property (\cf \eqref{eq:antipodeproperty}) of $\antipodenum{2}{1}$ and the antipode property of $\antipodenum{1}{2}$ are equivalent.
 \end{enumerate}
\end{proof}

\chapter{Connected graded Hopf algebras}\label{chapter:connected_graded_Hopf}
\section{Polynomials and differential operators}
\begin{defn}
 For the set $\naturals_0^d=\{n=(n_1,\ldots,n_d)|n_i\in\naturals_0\,\forallc 1\leq i\leq d\}$ of $d$-dimensional \emphind{multi indices} we define
 \begin{enumerate}[(i)]
  \item the partial order $n\leq m \iff n_i\leq m_i\,\forallc i$ with the shorthand notation $n<m\iff {n\leq m\AND n\neq m}$,
  \item the absolute value $\lvert n\rvert:=\sum_{i=1}^d n_i$,
  \item the sum $+:\,\naturals_0^d\times\naturals_0^d\rightarrow\naturals_0^d,\,n+m:=(n_1+m_1,\ldots,n_d+m_d)$,
  \item the difference $n-m:=(n_1-m_1,\ldots,n_d-m_d)$ for $m\leq n$,
  \item the factorial $\cdotp!:\,\naturals_0^d\rightarrow\naturals,\,n!:=\prod_{i=1}^d n_i!$,
  \item the binomial coefficient $\binom{n}{m}:=\frac{n!}{(n-m)!m!}$ for $m\leq n$.
 \end{enumerate}
 By $(e^i)_{1\leq i\leq d}\subset\naturals_0^d$ we denote the family of multi indices given by $e_j^i:=\deltasymb{i,j}$.
\end{defn}
\subsection{Overview}
In this section, we want to introduce the following two structures and show that they are Hopf algebras:
\setlength{\fboxsep}{5pt}
\begin{empheq}[box=\fbox]{equation*}
  \begin{aligned}
   &\polynoms_d=\vspan\{\X^n,\,n\in\naturals_0^d\}=\reals[\X_1,\ldots,\X_d]\\
   &\begin{alignedat}{2}
    &\propol:\,\polynoms_d\times\polynoms_d\rightarrow\polynoms_d,&\,&\X^n\propol\X^m=\X^{n+m}\\
    &\unitmappol:\,\reals\rightarrow\polynoms_d,&\quad&\unitmappol(r)=r\X^0\\
    &\copropol:\,\polynoms_d\rightarrow\polynoms_d\etensor\polynoms_d,&\quad&\copropol\X^n=\sum_{0\leq m\leq n}\binom{n}{m}\,\X^m\etensor\X^{n-m}\\
    &\counitpol:\,\polynoms_d\rightarrow\reals,&\quad&\counitpol(\X^n)=\begin{cases}1 &\text{if }n=0\\0&\text{else}\end{cases}\\
    &\antipol:\,\polynoms_d\rightarrow\polynoms_d,&\quad&\antipol\X^n=(-1)^n\X^n
   \end{alignedat}
  \end{aligned}
\end{empheq}

\begin{empheq}[box=\fbox]{equation*}
  \begin{aligned}
   &\diffs_d=\vspan\{\D^n,\,n\in\naturals_0^d\}\\
   &\begin{alignedat}{2}
    &\prodiff:\,\diffs_d\times\diffs_d\rightarrow\diffs_d,&\,&\D^n\prodiff\D^m=\binom{n+m}{m}\D^{n+m}\\
    &\unitmapdiff:\,\reals\rightarrow\diffs_d,&\quad&\unitmapdiff(r)=r\D^0\\
    &\coprodiff:\,\diffs_d\rightarrow\diffs_d\etensor\diffs_d,&\quad&\coprodiff\D^n=\sum_{0\leq m\leq n}\D^m\etensor\D^{n-m}\\
    &\counitdiff:\,\diffs_d\rightarrow\reals,&\quad&\counitdiff(\D^n)=\begin{cases}1 &\text{if }n=0\\0&\text{else}\end{cases}\\
    &\antidiff:\,\diffs_d\rightarrow\diffs_d,&\quad&\antidiff\D^n=(-1)^n\D^n
   \end{alignedat}
  \end{aligned}
\end{empheq}

\subsection{Polynomial algebra}
The content of this subsection is standard knowledge and belongs to the folklore of the theory of graded Hopf algebras.

The \emphind{polynomial algebra} is defined as the real vector space $\polynoms_d:=\vspan\{\X^n,\,n\in\naturals_0^d\}$\nomenclature[P d]{$\polynoms_d$}{Space of polynomials in $d$ dimensions} together with the bilinear product ${\propol}$ generated by $\X^n\X^m:=\X^n\propol\X^m:=\X^{n+m}$. Put $\unit:=\X^0$ and $\X_i:=\X^{e^i}$.
Our goal is to extend $(\polynoms_d,{\propol})$ to a connected graded Hopf algebra. Intuitively, we choose the grade of some monomial $\X^n$ to be equal to its polynomial degree $|n|$. This grading yields the decomposition $\polynoms_d=\bigoplus_{n\in\naturals}\polynoms_{d,n}$, where $\polynoms_{d,n}:=\vspan\{\X^m,\,m\in\naturals_0^d:\,|m|=n\}$. Hence, $\polynoms_{d,n}\propol\polynoms_{d,m}\subseteq\polynoms_{d,n+m}$ and $\dim\polynoms_{d,0}=1$, which means that $(\polynoms_d,{\propol})$ is a connected graded algebra (see Definition \ref{defn:grading}). When defining a compatible coproduct ${\copropol}$ which preserves the grading, we have to have $\copropol\unit:=\unit\etensor\unit$, which is exactly the compatibility requirement for the unit element in a bialgebra (see Remark \ref{rmk:comp_req}), as well as $\copropol\X_i:=\unit\etensor\X_i+\X_i\etensor\unit$ for all $i\in\{1,\dots,d\}$ due to the grading and remark \ref{rmk:coproduct_gradeone}. Then, in order to get a coproduct which is an algebra 
homomorphism as required in a bialgebra (see Definition \ref{defn:bialgebra}), we need to define
\begin{equation}\label{eq:def_copropol}
 \copropol\X^n:=\bigpropol_{i\in\{1,\dots,d\}}(\copropol\X_i)^{\propoltwo n_i}=\sum_{0\leq m\leq n}\binom{n}{m}\,\X^m\etensor\X^{n-m}\quad\forallc n\in\naturals_0^d,
\end{equation}
where again $(x\etensor v)\propoltwo(y\etensor w):=(x\propol y)\etensor(v\propol w)$. The coproduct on the whole of $\polynoms_d$ is then generated by linearity.
\begin{rmk}\label{rmk:copropol_unique}
 As we just showed there is only one candidate for a coproduct which may extend $(\polynoms_d,{\propol})$ to a bialgebra under the predefined grading $(\polynoms_{d,n})_n$.
\end{rmk}
\begin{lemma}
 $(\polynoms_d,{\propol},{\copropol})$ is a connected graded bialgebra.
\end{lemma}
\begin{proof}
 $(\polynoms_d,{\propol})$ is an unital algebra as $\propol$ is obviously associative and by definition bilinear with unit $\unit$. Likewise, $\copropol$ is linear by definition. Its linearity, the bilinearity of ${\propol}$ and the first equality in \eqref{eq:def_copropol} then imply that $\copropol$ is an algebra homomorphism. Hence, by Lemma \ref{lemma:coassociative_subset} coassociativity already follows generally from the special cases
 \begin{align*}
  (\id\etensor{\copropol})\copropol\X_i&=\unit\etensor\copropol\X_i+\X_i\etensor\copropol\unit=\unit\etensor\unit\etensor\X_i+\unit\etensor\X_i\etensor\unit+\X_i\etensor\unit\etensor\unit\\
  &=\copropol\unit\etensor\X_i+\copropol\X_i\etensor\unit=({\copropol}\etensor\id)\copropol\X_i.
 \end{align*}
 Thus, $(\polynoms_d,{\copropol})$ is a coalgebra. The fact that $\counitpol\in\adual{\polynoms_d}:\,\counitpol(P)=P(0)$ is a counit is obvious from the expansion in \eqref{eq:def_copropol}. The validity of the compatibility requirements for the counit and the unit element, namely $\copropol\unit=\unit\etensor\unit$ and $\counitpol(P)\counitpol(Q)=P(0)Q(0)=(PQ)(0)=\counitpol(PQ)$ for all $P,Q\in\polynoms_d$ (\cf Definition \ref{defn:bialgebra} and Remark \ref{rmk:comp_req}) then completes the proof of $(\polynoms_d,{\propol},{\copropol})$ being a bialgebra. As $\polynoms_{d,n}\propol\polynoms_{d,m}\subseteq\polynoms_{d,n+m}$ and 
 \begin{equation*}
  \copropol\polynoms_{d,n}\subseteq\bigoplus_{m+l=n}\polynoms_{d,m}\etensor\polynoms_{d,l}
 \end{equation*}
 (see \eqref{eq:def_copropol}) as well as $\dim\polynoms_{d,0}=1$, it is also connected graded according to Definition \ref{defn:grading}.\\
\end{proof}
As the graded bialgebra $(\polynoms_d,{\propol},{\copropol})$ is connected, Theorem \ref{thm:antipode_recursion} shows that an antipode is given by the recursion
\begin{equation}\label{eq:rec}
 \antipol P=-P-\mpropol(\id\etensor\antipol)\rcopropol P=-P-\sum_{(P)}P'\propol\antipol P''\quad\forallc P\in\ker\counitpol,\quad\antipol\unit=\unit.
\end{equation}
Knowing that an antipode exists, we may deduce its form without performing the recursion, but simply from the algebra antimorphism property any antipode fulfills due to Theorem \ref{thm:antipode_antimorphism}. Indeed, we just need to compute
\begin{align*}
 0&=\unitmappol\counitpol(\X_i)=\mpropol(\id\etensor\antipol)\copropol\X_i=\mpropol(\id\etensor\antipol)(\unit\etensor\X_i+\X_i\etensor\unit)=\unit\propol\antipol\X_i+\X_i\propol\antipol\unit\\
 &=\antipol\X_i+\X_i,
\end{align*}
concluding $\antipol\X_i=-\X_i$ for all $i\in\{1,\ldots,d\}$, and the fact that $\antipol$ is a $\propol$ antimorphism, and thus a homomorphism since $\propol$ is commutative, already implies that
\begin{align*}
 \antipol\X^n&=\antipol(\X_1^{\propol n_1}\propol\cdots\propol\X_d^{\propol n_d})=(\antipol\X_1)^{\propol n_1}\propol\cdots\propol(\antipol\X_d)^{\propol n_d}=(-\X_1)^{\propol n_1}\propol\cdots\propol(-\X_d)^{\propol n_d}\\
 &=(-1)^{\lvert n\rvert}\X^n
\end{align*}
for all $n\in\naturals_0^d$. The map $\antipol$ is then obviously graded, which is also stated by Theorem \ref{thm:antipode_recursion}. We thus have the following.

\begin{thm}\label{thm:polhopf}
 $(\polynoms_d,{\propol},{\copropol},\antipol)$ is a connected graded Hopf algebra.
\end{thm}
\subsection{Algebra of differential operators}
Now, we consider the space of differential operators $\diffs_d:=\vspan\{\partial^n,\,n\in\naturals_0^d\}$ over the field $\reals$. Our treatment of the corresponding Hopf algebra is based on Section 2.1 \cite{brouder04} and on Example 2.2 \cite{hairerkelly14}, and also on some standard facts.

The tuple $(\polynoms_d,\diffs_d)$ is a pair of dual vector spaces via the duality pairing (Remark 4.19 \cite{hairer14})
\begin{equation}\label{eq:dualitypairing_diff_pol}
 \langle\partial^n,p\rangle:=(\partial^n p)(0),\,p\in\polynoms_d
\end{equation}
The dual basis of $\{\X^n,\,n\in\naturals_0^d\}$ is then given by $\{\D^n,\,n\in\naturals_0^d\}$, where $\D^n:=\frac{1}{n!}\partial^n$. Indeed, we have 
\begin{equation*}
 \langle\D^n,\X^m\rangle=\frac{1}{n!}(\partial^n\X^m)(0)=\deltasymb{n,m}\quad\forallc n,m\in\naturals_0^d.
\end{equation*}
Analogous to the case of the polynomials, we introduce the intuitive grading leading to the decomposition $\diffs_d=\bigoplus_{n\in\naturals}\diffs_{d,(n)}$ into subspaces $\diffs_{d,(n)}:={\vspan\{\D^m,\,m\in\naturals_0^d:\,|m|=n\}}$ as well as a bilinear product, the \mbox{composition ${\prodiff}$}, via $\partial^n\partial^m:=\partial^n\prodiff\partial^m:=\partial^{n+m}$. Obviously, there is an isomorphism ${\imorphpd:\,(\polynoms_d,{\propol})\rightarrow(\diffs_d,{\prodiff})}$ such that $\imorphpd(\X^n)=\partial^n$, \ie the grading is preserved by it as well. Hence, by Remark \ref{rmk:copropol_unique}, we also have to choose the coproduct ${\coprodiff}$ and the antipode $\antidiff$ isomorphic to ${\copropol}$ and $\antipol$ in order to extend $(\diffs_d,{\prodiff})$ to a connected graded Hopf algebra, meaning (Section 2.1 \cite{brouder04}, Example 2.2 \cite{hairerkelly14})
\begin{equation*}
 \coprodiff\partial^n:=\sum_{0\leq m\leq n}\binom{n}{m}\,\partial^m\etensor\partial^{n-m},\quad\antidiff\partial^n:=(-1)^{|n|}\partial^n.
\end{equation*}
When using the basis $(\D^n)_{n\in\naturals_0^d}$, the algebraic rules read
\begin{equation*}
 \D^n\prodiff\D^m=\binom{n+m}{m}\D^{n+m},\quad\coprodiff\D^n=\sum_{0\leq m\leq n}\D^m\etensor\D^{n-m},\quad\antidiff\D^n=(-1)^{|n|}\D^n.
\end{equation*}
\begin{thm}
 $(\diffs_d,{\prodiff},{\coprodiff},\antidiff)$ is a connected graded Hopf algebra and dual to $(\polynoms_d,{\propol},{\copropol},\antipol)$, \ie
 \begin{alignat}{2}
  \langle D_1\prodiff D_2,P\rangle&=\langle D_1\otimes D_2,\copropol P\rangle&\quad&\forallc D_1,D_2\in\diffs_d,\,P\in\polynoms_d,\label{eq:dual_prodiff}\\
  \langle\coprodiff D,P\otimes Q\rangle&=\langle D,P\propol Q\rangle&&\forallc D\in\diffs_d,\,P,Q\in\polynoms_d,\label{eq:dual_coprodiff}\\
  \langle\antidiff D,P\rangle&=\langle D,\antipol P\rangle&&\forallc D\in\diffs_d,\,P\in\polynoms_d\label{eq:dual_antidiff}.
 \end{alignat}
\end{thm}
\begin{proof}
  The first claim is guaranteed by Theorem \ref{thm:polhopf} and isomorphy by construction. As for the duality identities, it is sufficient to check their validity for the elements of the bases $(\X^n)_{n\in\naturals_0^d}$ and $(\D^n)_{n\in\naturals_0^d}$ due to the bilinearity of the duality pairing and the products and the linearity of the coproducts and the antipodes. Considering \eqref{eq:dual_prodiff}, we indeed compute
 \begin{multline*}
  \langle\D^n\prodiff\D^m,\X^l\rangle=\binom{n+m}{n}\deltasymb{n+m,l}=\sum_{0\leq k\leq l}\binom{l}{k}\deltasymb{n,k}\deltasymb{m,l-k}=\sum_{0\leq k\leq l}\binom{l}{k}\langle\D^n\otimes\D^m,\X^k\etensor\X^{l-k}\rangle\\
  =\langle\D^n\otimes\D^m,\copropol\X^l\rangle
 \end{multline*}
 for all $n,m,l\in\naturals_0^d$. Likewise,
 \begin{equation*}
  \langle\coprodiff\D^n,\X^m\otimes\X^l\rangle=\sum_{0\leq k\leq n}\langle\D^k\etensor\D^{n-k},\X^m\otimes\X^l\rangle=\sum_{0\leq k\leq n}\deltasymb{k,m}\deltasymb{n-k,l}=\deltasymb{n,m+l}=\langle\D^n,\X^m\propol\X^l\rangle
 \end{equation*}
 proves \eqref{eq:dual_coprodiff} and
 \begin{equation*}
  \langle\antidiff\D^n,\X^m\rangle=(-1)^{|n|}\deltasymb{n,m}=(-1)^{|m|}\deltasymb{n,m}=\langle\D^n,\antipol\X^m\rangle
 \end{equation*}
 validates \eqref{eq:dual_antidiff}.\\
\end{proof}

At this point, we would like to move to the algebraic dual space $\adual{\polynoms_d}$, the space of formal series $\sum_{n\in\naturals_0^d}a_n\D^n$ of which $\diffs_d$ is just the subspace of finite series. 
We can extend the Hopf algebra operations via
\begin{align*}
 \Big(\sum_{n\in\naturals_0^d}a_n\D^n\Big)\prodiff\Big(\sum_{m\in\naturals_0^d}b_m\D^m\Big)&:=\sum_{n,m\in\naturals_0^d}a_n b_m\D^n\prodiff\D^m=\sum_{k\in\naturals_0^d}\sum_{0\leq l\leq k}\binom{k}{l}a_l b_{k-l}\D^k,\\
 \copropodu\sum_{n\in\naturals_0^d}a_n\D^n&:=\sum_{n\in\naturals_0^d}a_n\coprodiff\D^n=\sum_{m,k\in\naturals_0^d}a_{m+k}\,\D^m\etensor\D^k,\\
 \antipodu\sum_{n\in\naturals_0^d}a_n\D^n&:=\sum_{n\in\naturals_0^d}a_n\antidiff\D^n=\sum_{n\in\naturals_0^d}(-1)^{|n|}a_n\D^n,
\end{align*}
but the problem is that $(\adual{\polynoms_d},\copropodu)$ does \emph{not} constitute a coalgebra, because we only have that ${\copropodu\adual{\polynoms_d}\subseteq\adual{(\polynoms_d\etensor\polynoms_d)}}$, yet
\begin{thm}
 $\copropodu\adual{\polynoms_d}\nsubseteq\adual{\polynoms_d}\etensor\adual{\polynoms_d}$.
\end{thm}
\begin{proof}
 At first, let $d=1$.
 Consider $D:=\sum_{n=0}^\infty\D^{n^2}\in\adual{\polynoms_1}$. We have
 \begin{equation}\label{eq:copropodu_D}
  \copropodu D=\sum_{n=0}^\infty\sum_{k\leq n^2}\D^k\etensor\D^{n^2-k}=\sum_{k=0}^\infty\D^k\etensor\Big(\sum_{n^2\geq k}\D^{n^2-k}\Big).
 \end{equation}
 Assume $\copropodu D\in\adual{\polynoms_1}\etensor\adual{\polynoms_1}$. Then, there is $l\in\naturals$, $(a_k^i)_{k,i}\subset\reals$ and $(L_i)_i\subset\adual{\polynoms_1}$ such that 
 \begin{equation*}
 \copropodu D=\sum_{i=1}^l\Bigg(\Big(\sum_{k=0}^\infty a_k^i\D^k\Big)\etensor L_i\Bigg)=\sum_{k=0}^\infty\D^k\etensor\Big(\sum_{i=1}^l a_k^i L_i\Big).
 \end{equation*}
 This is a contradiction to \eqref{eq:copropodu_D} as
 \begin{equation*}
  \dim\vspan\Big\{\sum_{i=1}^l a_k^i L_i,k\in\naturals_0\Big\}\leq l,\quad\text{but }\dim\vspan\Big\{\sum_{n^2\geq k}\D^{n^2-k},k\in\naturals_0\Big\}=\infty.
 \end{equation*}
 For $d>1$, considering $D:=\sum_{n\in\naturals_0}\D^{n^2 e_1}\in\adual{\polynoms_d}$ will lead to the same result.\\
\end{proof}
However, this is not a big issue for us, since in Section \ref{subsection:formal_series}, we will see that we can nevertheless work with the extended operation $\copropodu$ basically as usual, and that the defining identities of a Hopf algebra remain valid for formal series.

\section{Tensor Hopf algebras}
\subsection{Overview}
In this section, we want to introduce the following two structures and show that they are Hopf algebras:
 \begin{empheq}[box=\fbox]{equation*}
  \begin{aligned}
   &\tensoral{V}=\bigoplus_{n=0}^{\infty}V^{\itensor n}\\
   &\itensor:\,\tensoral{V}\times\tensoral{V}\rightarrow\tensoral{V}\\
   &\unitmap:\,K\rightarrow\tensoral{V},\quad\unitmap(r)=r\unit\\
   &\deshuffle:\,\tensoral{V}\rightarrow\tensoral{V}\etensor\tensoral{V},\quad u=v_1\itensor\ldots\itensor v_n\text{ is mapped to}\\
   &\deshuffle u=\unit\etensor u+u\etensor\unit+\sum_{0<i<n}\sum_{\sigma\in\shuffleperm(i,n)}v_{\sigma(1)}\itensor\dots\itensor v_{\sigma(i)}\etensor v_{\sigma(i+1)}\itensor\dots\itensor v_{\sigma(n)}\\
   &\counit:\,\tensoral{V}\rightarrow K,\quad\counit(\unit)=1,\quad\counit(V^{\itensor n})=\{0\}\text{ for }n\in\naturals\\
   &\antitensor:\,\tensoral{V}\rightarrow\tensoral{V},\quad\antitensor(v_1\itensor\dots\itensor v_n)=(-1)^n v_n\itensor\dots\itensor v_1
  \end{aligned}
\end{empheq}

\begin{empheq}[box=\fbox]{equation*}
  \begin{aligned}
   &\tensoral{\Vpdual}\\
   &\shuffle:\,\tensoral{\Vpdual}\times\tensoral{\Vpdual}\rightarrow\tensoral{\Vpdual},\\
   &(\vpdual_1\itensor\dots\itensor \vpdual_i)\shuffle(\vpdual_{i+1}\itensor\dots\itensor \vpdual_n):=\sum_{\sigma\in\shuffleperm(i,n)}\vpdual_{\sigma^{-1}(1)}\itensor\dots\itensor \vpdual_{\sigma^{-1}(n)}\\
   &\unitmap:\,K\rightarrow\tensoral{\Vpdual},\quad\unitmap(r)=r\unit\\
   &\deconc:\,\tensoral{\Vpdual}\rightarrow\tensoral{\Vpdual}\etensor\tensoral{\Vpdual},\\
   &\deconc(\vpdual_1\itensor\dots\itensor\vpdual_n)=\sum_{i=0}^n\vpdual_1\itensor\dots\itensor\vpdual_i\etensor\vpdual_{i+1}\itensor\dots\itensor\vpdual_n\\
   &\counit:\,\,\tensoral{\Vpdual}\rightarrow K,\quad\counit(\unit)=1,\quad\counit(\Vpdual^{\itensor n})=\{0\}\text{ for }n\in\naturals\\
   &\antishuffle:\,\,\tensoral{\Vpdual}\rightarrow\tensoral{\Vpdual},\quad\antishuffle(v_1\itensor\dots\itensor v_n)=(-1)^n v_n\itensor\dots\itensor v_1
  \end{aligned}
\end{empheq}
\subsection{Concatenation product Hopf algebra}
\begin{defn}\textnormal{(Section I.2.2.\@ \cite{manchon06}, Chapter 4 \cite{hairerkelly14})}
 For some real or complex vector space $V$ and tensor products $\itensor$, the \emph{tensor algebra of $V$}\index{tensor algebra} is given by 
 \begin{equation}\label{eq:tensoral}
  \tensoral{V}:=\bigoplus_{n=0}^{\infty}V^{\itensor n},
 \end{equation}
 \nomenclature[T(V)]{$\tensoral{V}$}{Tensor algebra of some vector space V}where $V^{\itensor 0}:=\vspan\{\unit\}$ is isomorphic to $\reals$ or $\complexs$, respectively. 
\end{defn}
We consider $\itensor:\,\tensoral{V}\times\tensoral{V}\rightarrow\tensoral{V}$ as the algebra product on $\tensoral{V}$ with unit $\unit$, also called \emphind{concatenation}, while $\tensoral{V}\etensor\tensoral{V}\nsubseteq\tensoral{V}$ reoccurs as the space to which the coproduct maps.
\begin{lemma}\textnormal{(Section I.2.2.\@ \cite{manchon06})}
 The couple $(\tensoral{V},{\itensor})$ is a connected graded unital algebra under the grading given by \eqref{eq:tensoral}.
\end{lemma}
\begin{proof}
 Tensor products are bilinear and associativity is fulfilled due to the usual identification $u_1\itensor u_2\itensor u_3:=(u_1\itensor u_2)\itensor u_3=u_1\itensor (u_2\itensor u_3)$. The unit is by definition given by $\unit$.
\end{proof}
\begin{defn}(Based on the definition of the shuffle product in Chapter 4 \cite{hairerkelly14})
 Let $\deshuffle:\,\tensoral{V}\rightarrow\tensoral{V}\etensor\tensoral{V}$\nomenclature[00g D shuffle]{$\deshuffle$}{The coproduct which is dual to the shuffle product $\shuffle$} be the linear map generated by ${\deshuffle\unit:=\unit\etensor\unit}$ and
 \begin{equation}\label{eq:defn_deshuffle}
  \deshuffle u:=\unit\etensor u+u\etensor\unit+\sum_{0<i<n}\sum_{\sigma\in\shuffleperm(i,n)}v_{\sigma(1)}\itensor\dots\itensor v_{\sigma(i)}\etensor v_{\sigma(i+1)}\itensor\dots\itensor v_{\sigma(n)}
 \end{equation}
 for all ${u=v_1\itensor\dots\itensor v_n\in V^{\itensor n}}$, where 
 \begin{equation*}
  \shuffleperm(i,n):=\{\sigma\in\symgroup_n|\,\sigma(1)<\ldots<\sigma(i),\,\sigma(i+1)<\ldots<\sigma(n)\}
 \end{equation*}
 and $\symgroup_n$ is the group of permutations of the set $\{1,\ldots,n\}$.
\end{defn}
\begin{example}
 Let $a,b,c,d\in V$ be arbitrary. Since
 \begin{align*}
  \shuffleperm(1,4)&=\{(1234),\,(2134),\,(3124),\,(4123)\},\\
  \shuffleperm(2,4)&=\{(1234),\,(1324),\,(1423),\,(2314),\,(2413),\,(3412)\},\\
  \shuffleperm(3,4)&=\{(1234),\,(1243),\,(1342),\,(2341)\},
 \end{align*}
 we get
 \begin{align*}
  \deshuffle(a\itensor b\itensor c\itensor d)&=\unit\etensor a\itensor b\itensor c\itensor d+a\itensor b\itensor c\itensor d\etensor\unit+a\etensor b\itensor c\itensor d+b\etensor a\itensor c\itensor d\\
  &\hphantom{\hphantom{}=\hphantom{}}+c\etensor a\itensor b\itensor d+d\etensor a\itensor b\itensor c+a\itensor b\etensor c\itensor d+a\itensor c\etensor b\itensor d\\
  &\hphantom{\hphantom{}=\hphantom{}}+a\itensor d\etensor b\itensor c+b\itensor c\etensor a\itensor d+b\itensor d\etensor a\itensor c+c\itensor d\etensor a\itensor b\\
  &\hphantom{\hphantom{}=\hphantom{}}+a\itensor b\itensor c\etensor d+a\itensor b\itensor d\etensor c+a\itensor c\itensor d\etensor b+b\itensor c\itensor d\etensor a.
 \end{align*}
 If $b=c$ and $a=d$, we thus have
 \begin{align*}
  \deshuffle(a\itensor b\itensor b\itensor a)&=\unit\etensor a\itensor b\itensor b\itensor a+a\itensor b\itensor b\itensor a\etensor\unit+ a\etensor b\itensor b\itensor a+ 2\, b\etensor a\itensor b\itensor a\\
  &\hphantom{\hphantom{}=\hphantom{}}+a\etensor a\itensor b\itensor b+ 2\,a\itensor b\etensor b\itensor a+a\itensor a\etensor b\itensor b+b\itensor b\etensor a\itensor a\\
  &\hphantom{\hphantom{}=\hphantom{}}+2\, b\itensor a\etensor a\itensor b+a\itensor b\itensor b\etensor a+2\,a\itensor b\itensor a\etensor b+b\itensor b\itensor a\etensor a.
 \end{align*}

\end{example}

\begin{lemma}\textnormal{(Proposition 1.9 \cite{reutenauer93}, Section I.6.2.\@ \cite{manchon06})}
 The operator $\deshuffle$ is the unique $(\tensoral{V},{\itensor})$ homomorphism to ${\tensoral{V}\etensor\tensoral{V}}$ satisfying 
 \begin{equation}\label{eq:deshuffle_on_V}
  \deshuffle v=\unit\etensor v+v\etensor\unit
 \end{equation}
 for all ${v\in V}$. The object $(\tensoral{V},{\itensor},{\deshuffle})$ is a connected graded bialgebra.
\end{lemma}
\begin{proof}
 The property \eqref{eq:deshuffle_on_V} is included in \eqref{eq:defn_deshuffle}, because the sum over $0<i<n$ is empty for $n=1$. 
 Let $n\in\naturals$ and $v_1,\ldots,v_{n+1}\in V$ be arbitrary. Since for all $\sigma\in\shuffleperm(i,n+1)$ we have $\sigma(i)=n+1$ or $\sigma(n+1)=n+1$, we get 
 \begin{align*}
  \rdeshuffle(u_n\itensor v_{n+1})&=\sum_{(u_n)}\big((u_1'\itensor v_{n+1})\etensor u_n''+u_1'\etensor(u_2''\itensor v_{n+1})\big)+u_n\itensor v_{n+1}+v_{n+1}\itensor u_n\\
  &=\rdeshuffle u_n\itensortwo(v_{n+1}\etensor\unit+\unit\etensor v_{n+1})+u_n\itensor v_{n+1}+v_{n+1}\itensor u_n\\
  &=\deshuffle u_n\itensortwo\deshuffle v_{n+1}-u_n\itensor v_{n+1}\etensor\unit-\unit\etensor u_n\itensor v_{n+1}.
 \end{align*}
 Hence, inductively we prove
 \begin{equation}\label{eq:deshuffle_morphism}
  \deshuffle(v_1\itensor\dots\itensor v_n)=\deshuffle v_1\itensortwo\dots\itensortwo\deshuffle v_n
 \end{equation}
 for all $n\in\naturals$ and $v_1,\ldots,v_n\in V$, showing that $\deshuffle$ is an algebra homomorphism. Since \eqref{eq:deshuffle_morphism} is a necessary condition for an algebra homomorphism, $\deshuffle$ is indeed the unique algebra homomorphism such that \eqref{eq:deshuffle_on_V} holds. It is then also a coproduct due to Lemma \ref{lemma:coassociative_subset}, since $(\id\etensor\deshuffle)\deshuffle v=(\deshuffle\etensor\id)\deshuffle v$ holds for all $v\in V$. Together with the counit given by $\counittens(\unit)=1$, $\counittens(V^{\itensor n})=0$, $n\geq 1$, all the compatibility requirements are fulfilled, hence $(\tensoral{V},{\itensor},{\deshuffle})$ is a bialgebra. Since from \eqref{eq:defn_deshuffle} we see that
 \begin{equation*}
  \deshuffle V^{\itensor n}\subseteq\bigoplus_{0\leq m\leq n}V^{\itensor m}\etensor V^{\itensor n-m},
 \end{equation*}
 it is also connected graded.\\
\end{proof}

\begin{lemma}\textnormal{(Lemma 1.5 \cite{reutenauer93}, Section I.6.2.\@ \cite{manchon06})}
 The linear operator $\antitensor:\tensoral{V}\rightarrow\tensoral{V}$ generated by
 \begin{equation*}
  \antitensor\unit:=\unit,\quad
  \antitensor(v_1\itensor\dots\itensor v_n):=(-1)^n v_n\itensor\dots\itensor v_1\quad\forallc v_i\in V
 \end{equation*}
 is an antipode of $(\tensoral{V},{\itensor},{\deshuffle})$.
\end{lemma}
\begin{proof}\textnormal{(Proof of Lemma 1.5 \cite{reutenauer93})}
As ${\mitensor({\id}\etensor{\antitensor})\deshuffle\unit}{=\mitensor({\antitensor}\etensor{\id})\deshuffle\unit}{=\unitmap\counit\unit}{=\unit}$, it only remains to show
 \begin{equation}\label{eq:ind_antitensor}
  \mitensor({\id}\etensor{\antitensor})\deshuffle u=\mitensor({\antitensor}\etensor{\id})\deshuffle u=0\quad\forallc u\in V^{\itensor m},m\in\naturals,
 \end{equation}
  where we proceed by induction. Regarding the case ${m=1}$, we obtain
 \begin{equation*}
  \mitensor({\id}\etensor{\antitensor})\deshuffle v=\unit\itensor\antitensor v+v\itensor\antitensor\unit=-v+v\itensor\unit=0
 \end{equation*}
 as well as
 \begin{equation*}
  \mitensor({\antitensor}\etensor{\id})\deshuffle v=\antitensor\unit\itensor v+\antitensor v\itensor\unit=\unit\itensor v-v=0.
 \end{equation*}
 for all $v\in V$. Assuming that \eqref{eq:ind_antitensor} holds for ${m=n}$, we have
 \begin{equation*}
  \sum_{(w)} w'\itensor\antitensor w''=\sum_{(w)}\antitensor w'\itensor w''=0
 \end{equation*}
 for arbitrary $w\in V^{\itensor n}$. Let $x\in V$ be arbitrary as well. Then,
 \begin{equation*}
  \deshuffle(w\itensor x)=\deshuffle w\itensortwo\deshuffle x=\sum_{(w)}\left(w'\etensor(w''\itensor x)+(w'\itensor x)\etensor w''\right)
 \end{equation*}
 and therefore
 \begin{align*}
  \mitensor({\antitensor}\etensor{\id})\deshuffle (w\itensor x)&=\sum_{(w)}\left(\antitensor w'\itensor(w''\itensor x)+\antitensor(w'\itensor x)\itensor w''\right)\\&=\Big(\sum_{(w)}\antitensor w'\itensor w''\Big)\itensor x+x\itensor\sum_{(w)}\antitensor w'\itensor w''=0.
 \end{align*}
 Analogously, we get
 \begin{align*}
  \mitensor({\id}\etensor\antitensor)\deshuffle (x\itensor w)&=\sum_{(w)}\left(w'\itensor\antitensor(x\itensor w'')+(x\itensor w')\itensor\antitensor w''\right)\\&=\Big(\sum_{(w)} w'\itensor\antitensor w''\Big)\itensor x+x\itensor\sum_{(w)} w'\itensor\antitensor w''=0.
 \end{align*}
 As $w\in V^{\itensor n}, x\in V$ were arbitrary, we get \eqref{eq:ind_antitensor} for ${m=n+1}$ by linearity of $\antitensor$.\\
\end{proof}
\begin{cor}\textnormal{(Section I.6.2.\@ \cite{manchon06})}
 $(\tensoral{V},{\itensor},{\deshuffle},{\antitensor})$ is a connected graded Hopf algebra.
\end{cor}
\subsection{Shuffle product Hopf algebra}

\begin{thm}
 Let $\dualbracket$ be a duality pairing of $V$ with some vector space $\Vpdual$. Then, the pair ${(\tensoral{V},\tensoral{\Vpdual})}$ is a pair of dual vector spaces with duality pairing
 \begin{equation}\label{eq:tensoral_dual}
  \left\langle\sum_{n=0}^{\infty}u_n,\sum_{m=0}^{\infty}\pdual{u}_m\right\rangle_{\mathrm{T}}:=\sum_{n=0}^{\infty}\langle u_n,\pdual{u}_n\rangle_{\itensor n},
 \end{equation}
 where $\langle a\unit,b\unit\rangle_{\itensor 0}:=ab$, $\langle v,\vpdual\rangle_{\itensor 1}:=\langle v,\vpdual\rangle$ and
 \begin{equation*}
  \left\langle\sum_{i=1}^{k}v_1^{i}\itensor\cdots\itensor v_n^{i},\sum_{j=1}^{l}\vpdual_1^{j}\itensor\cdots\itensor\vpdual_n^{j}\right\rangle_{\itensor n}:=\sum_{i=1}^{k}\sum_{j=1}^{l}\langle v_1^{i},\vpdual_1^{j}\rangle\cdots\langle v_n^{i},\vpdual_n^{j}\rangle.
 \end{equation*}

\end{thm}
\begin{proof}
 By repeatedly applying Theorem \ref{thm:tensordual}, we see that the pairings $\dualbracket_{\itensor n}, n\in\naturals_0$ are all well-defined and bilinear. Thus, also $\dualbracket_{\mathrm{T}}$ is well-defined and bilinear.
 For some arbitrary $0\neq{u=\sum_{n=0}^{\infty}u_n\in\tensoralfs{V}}$ with ${u_m\in V^{\itensor m}}$ for all ${{m\in\naturals_0}}$, there is $k$ such that ${u_k\neq0}$. If ${k=0}$, we are done as ${\langle u,\unit\rangle_{\mathrm{T}}\neq0}$. Otherwise, we see through repetitive application of Theorem \ref{thm:tensordual} that $(V^{\itensor k},\Vpdual^{\itensor k})$ is a pair of dual vector spaces and hence there is ${\pdual{u}\in\Vpdual^{\itensor k}}$ such that
 \begin{equation*}
  0\neq\langle u_k,\pdual{u}\rangle_{\itensor k}=\langle u,\pdual{u}\rangle_{\mathrm{T}}.
 \end{equation*}
 Analogously, we show that for each ${\pdual{x}\in\tensoral{\Vpdual}}$ there is ${x\in\tensoral{V}}$ such that ${\langle x,\pdual{x}\rangle_{\mathrm{T}}\neq0}$.
\\\end{proof}

\begin{lemma}\label{lemma:shuffle}\textnormal{(Sections 1.4 and 1.5 \cite{reutenauer93}, Chapter 4 \cite{hairerkelly14})}
 The \emphind{shuffle product} $\shuffle:\tensoral{\Vpdual}\times\tensoral{\Vpdual}\rightarrow\tensoral{\Vpdual}$\nomenclature[ZZs shuffle]{$\shuffle$}{Shuffle product} bilinearly generated by
 \begin{equation*}
  \unit\shuffle u:=u\shuffle\unit:=u
 \end{equation*}
 and
 \begin{equation*}
  (\vpdual_1\itensor\dots\itensor \vpdual_i)\shuffle(\vpdual_{i+1}\itensor\dots\itensor \vpdual_n):=\sum_{\sigma\in\shuffleperm(i,n)}\vpdual_{\sigma^{-1}(1)}\itensor\dots\itensor \vpdual_{\sigma^{-1}(n)}
 \end{equation*}
 is dual to $\deshuffle$.
\end{lemma}
\begin{proof}
 Let $v_1,\ldots,v_n\in V$ and $\vpdual_1,\ldots,\vpdual_{k+l}\in\Vpdual$ be arbitrary. Put $w:=v_1\itensor\cdots\itensor v_n$, $a:=\vpdual_1\itensor\cdots\itensor\vpdual_{k}$ and $b:=\vpdual_{k+1}\itensor\cdots\itensor\vpdual_{k+l}$. If $k+l\neq n$, then $\langle w,a\shuffle b\rangle_\mathrm{T}=0=\langle\deshuffle w,a\etensor b\rangle_{\mathrm{T}\etensor}$ as 
 \begin{equation}\label{eq:grading_shuffledual}
  w\in V^{\itensor n},\quad a\shuffle b\in\Vpdual^{\itensor (k+l)},\quad\deshuffle w\in\bigoplus_{i=0}^{n}V^{\itensor i}\etensor V^{\itensor (n-i)},\quad a\etensor b\in V^{\itensor k}\etensor V^{\itensor l}.
 \end{equation}
 If $n=k+l$, then
 \begin{multline*}
  \langle w,a\shuffle b\rangle_{\mathrm{T}}=\sum_{\sigma\in\shuffleperm(k,n)}\prod_{j=1}^n\langle v_j,\vpdual_{\sigma^{-1}(j)}\rangle=\sum_{\sigma\in\shuffleperm(k,n)}\prod_{j=1}^n\langle v_{\sigma(j)},\vpdual_j\rangle\\
  =\sum_{\sigma\in\shuffleperm(k,n)}\langle v_{\sigma(1)}\itensor\cdots\itensor v_{\sigma(k)},\vpdual_1\itensor\cdots\itensor\vpdual_k\rangle_{\itensor k}\langle v_{\sigma(k+1)}\itensor\cdots\itensor v_{\sigma(n)},\vpdual_{k+1}\itensor\cdots\itensor\vpdual_n\rangle_{\itensor l}\\
  =\langle\deshuffle w,a\etensor b\rangle_{\mathrm{T}\etensor}.
 \end{multline*}
 For $z\in\tensoral{V}$ we have 
 \begin{equation*}
  \langle z,\unit\shuffle u\rangle_{\mathrm{T}}=\langle z, u\shuffle\unit\rangle_{\mathrm{T}}=\langle z, u\rangle_{\mathrm{T}}=\langle\deshuffle z,\unit\etensor u\rangle_{\mathrm{T}\etensor}=\langle\deshuffle z,u\etensor\unit\rangle_{\mathrm{T}\etensor}
 \end{equation*}
 and for $x,y\in\tensoral{\Vpdual}$ also
 \begin{equation*}
  \langle\unit,x\shuffle y\rangle_{\mathrm{T}}=\langle\unit,x\rangle_{\mathrm{T}}\langle\unit,y\rangle_{\mathrm{T}}=\langle\unit\etensor\unit,x\etensor y\rangle_{\mathrm{T}\etensor}=\langle\deshuffle\unit,a\etensor b\rangle_{\mathrm{T}\etensor}.
 \end{equation*}
 The duality of $\shuffle$ and $\deshuffle$ then follows by (bi)linearity.\\
\end{proof}

\begin{lemma}\label{lemma:deconcatenation}\textnormal{(Chapter 4 \cite{hairerkelly14})}
 The \emphind{deconcatenation} $\deconc:\,\tensoral{\Vpdual}\rightarrow\tensoral{\Vpdual}\etensor\tensoral{\Vpdual}$ linearly generated by
 \begin{equation}\label{eq:defn_deconc}
  \deconc\unit:=\unit\etensor\unit,\quad
  \deconc(\vpdual_1\itensor\dots\itensor\vpdual_n)=\sum_{i=0}^n\vpdual_1\itensor\dots\itensor\vpdual_i\etensor\vpdual_{i+1}\itensor\dots\itensor\vpdual_n
 \end{equation}
 is dual to the concatenation $\itensor$.
\end{lemma}
\begin{proof}
 Consider $w=\vpdual_1\itensor\ldots\itensor\vpdual_n\in\Vpdual$ and $a=v_1\itensor\ldots\itensor v_k,b=v_{k+1}\itensor\ldots\itensor v_{k+l}\in V$, $0<k<n$. Then, 
 $\langle a\etensor b,\deconc w\rangle_{\mathrm{T}\etensor}=0=\langle a\itensor b,w\rangle_{\mathrm{T}}$ if $k+l\neq n$ (\cf \eqref{eq:grading_shuffledual}). If $k+l=n$, we have
 \begin{align*}
  \langle a\etensor b,\deconc w\rangle_{\mathrm{T}\etensor}&=\langle v_1\itensor\ldots\itensor v_k,\vpdual_1\itensor\ldots\itensor\vpdual_k\rangle_{\itensor k}\langle v_{k+1}\itensor\ldots\itensor v_n,\vpdual_{k+1}\itensor\ldots\itensor\vpdual_n\rangle_{\itensor (n-k)}\\
  &=\langle a\itensor b,w\rangle_{\mathrm{T}}.
 \end{align*}
 The remainder of the proof is completely analogous to that of Lemma \ref{lemma:shuffle}.\\
\end{proof}

\begin{lemma}\label{lemma:antitensor}
 The linear operator $\antishuffle:\tensoral{\Vpdual}\rightarrow\tensoral{\Vpdual}$ generated by
 \begin{equation*}
  \antishuffle\unit:=\unit,\quad\antishuffle(\vpdual_1\itensor\cdots\itensor\vpdual_n):=(-1)^n\vpdual_n\itensor\dots\itensor\vpdual_1
 \end{equation*}
 is dual to $\antitensor$.
\end{lemma}
\begin{proof} 
 We have for $u\in V^{\itensor m}$ and $v\in V^{\itensor n}$ that $\langle u,\antishuffle w\rangle_{\mathrm{T}}=0=\langle\antitensor u,w\rangle_{\mathrm{T}}$ if $m\neq n$. Also, $\langle\antitensor\unit,\unit\rangle_{\mathrm{T}}=1=\langle\unit,\antishuffle\unit\rangle_{\mathrm{T}}$ and for $u=v_1\itensor\cdots\itensor v_n,w=\vpdual_1\itensor\cdots\itensor\vpdual_n$ we get
 \begin{equation*}
  \langle u,\antishuffle w\rangle_{\mathrm{T}}=(-1)^n\prod_{i=1}^n\langle v_i,\vpdual_{n-i}\rangle=\langle\antitensor u,w\rangle_{\mathrm{T}}.
 \end{equation*}
\end{proof}
\begin{thm}\textnormal{(Proposition 1.9 \cite{reutenauer93}, Chapter 4 \cite{hairerkelly14})}
 ${(\tensoral{\Vpdual},{\shuffle},{\deconc},{\antishuffle})}$ is a connected graded Hopf algebra and dual to the Hopf algebra $(\tensoral{V},{\itensor},{\deshuffle},{\antitensor})$.
\end{thm}
\begin{proof}
 The operations $\shuffle$, $\deconc$ and $\antishuffle$ obviously are compatible with the grading. The fact that ${(\tensoral{\Vpdual},{\shuffle},{\deconc},{\antishuffle})}$ follows from Theorem \ref{thm:dual_hopf_bi_co_algebras} applied to the duality relations of Lemmas \ref{lemma:shuffle}, \ref{lemma:deconcatenation} and \ref{lemma:antitensor} and the obvious dualities of unitmaps and counits.
\end{proof}
\section{Formal series, Lie algebra and truncations}

\subsection{Formal series}\label{subsection:formal_series}
This section is based on the use of formal series and infinite linearity in \cite{reutenauer93}, though our definitions are a bit different.

For any vector space $G$ graded by $(G_i)_i$, \ie $G=\bigoplus_{i=0}^\infty G_i$ for subspaces $G_i$, the corresponding space of formal series\index{formal series} is given by
\begin{equation*}
  \spacefs{G}:=\directproduct_{i=0}^\infty G_i:=\left\{\sum_{n=0}^{\infty}u_n\middle|\,u_m\in G_m\,\forallc m\in\naturals_0\right\},
\end{equation*}
where the product is the Cartesian product of vector spaces (Chapter V Section 5.\@ \cite{robertson64}). 
For the space $G\etensor G$, consider the induced grading 
\begin{equation}\label{eq:induced_grading}
 ([G\etensor G]_i)_i:=\Big(\bigoplus_{m\leq i}G_m\etensor G_{i-m}\Big)_i.
\end{equation}
Let $\proj_n:\,\spacefs{G}\rightarrow G_n$ and $\projtwo_n:\,\spacefs{[G\etensor G]}\rightarrow[G\etensor G]_n$ be the canonical projections.
Generally, for $k\in\naturals$, $k\geq 2$, the grading of the space $G^{\etensor k}$ is given by
\begin{equation*}
 ([G^{\etensor k}]_i)_i:=\Big(\bigoplus_{\substack{m_1,\ldots,m_k\in\naturals_0:\\m_1+\cdots+m_k=i}}G_{m_1}\etensor\cdots\etensor G_{m_k}\Big)_i,
\end{equation*}
and we denote the canonical projection on $[G^{\etensor k}]_n$ by $\proj_n^{[k]}$.

\begin{thm}~
 \begin{enumerate}[(i)]
  \item If $(G,\product)$ is a graded algebra, then $\spacefs{G}$ is an algebra together with the \emph{extended product} $\extproduct:\,\spacefs{G}\times\spacefs{G}\rightarrow\spacefs{G}$ given by
  \begin{equation*}
   \proj_n\linofbilin{\extproduct}(a\etensor b)=\proj_n(a\extproduct b):=\sum_{i=0}^n \proj_i a\product\proj_{n-i}b=\mproduct\projtwo_n(a\etensor b).
  \end{equation*}
  \item If $(G,\coproduct)$ is a graded coalgebra, then the \emph{extended coproduct} $\extcoproduct:\,\spacefs{G}\rightarrow\spacefs{[G\etensor G]}$ given by
  \begin{equation*}
   \projtwo_n\extcoproduct:=\coproduct\proj_n
  \end{equation*}
  is coassociative. If $(G,\coproduct)$ is counitary with counit $\counit$, then $\extcounit:\,\spacefs{G}\rightarrow K$ given by
  \begin{equation*}
   \extcounit:=\counit\proj_0.
  \end{equation*}
  fullfills the counit property\index{counit!property}
  \begin{equation*}
   (\extcounit\etensor\id)\extcoproduct=(\id\etensor\extcounit)\extcoproduct=\id.
  \end{equation*}
  \item If $(G,\product,\coproduct)$ is a graded bialgebra with unit map $\unitmap$ and counit $\counit$ over the field $K$ such that $B^{>0}\subseteq\ker\counit$, then $(\spacefs{G},\extproduct,\extcoproduct)$ fulfills the compatibility requirements\index{compatibility requirements}
  \begin{enumerate}
   \item $\extcoproduct\mextproduct=(\mextproduct\etensor\mextproduct)(\id\etensor\flip\etensor\id)(\extcoproduct\etensor\extcoproduct)\quad$ ($\extcoproduct$ is an algebra homomorphism from $(\spacefs{G},\extproduct)$ to $(\spacefs{[G\etensor G]},\extproducttwo)$),
   \item $\extcoproduct\unitmap=\unitmap\etensor\unitmap$,
   \item $\extcounit\mproduct=\extcounit\etensor\extcounit\quad$ ($\extcounit$ is an algebra homomorphism from $(\spacefs{G},\extproduct)$ to $(K,\cdot)$),
  \end{enumerate}
  where $\producttwo$ is defined as in \eqref{eq:producttwo}.
  \item If $(G,\product,\coproduct,\antipode)$ is a graded Hopf algebra such that $B^{>0}\subseteq\ker\counit$ and $\extantipode:\,\spacefs{G}\rightarrow\spacefs{G}$ is the \emph{extended antipode} given by
  \begin{equation*}
   \proj_n\extantipode:=\antipode\proj_n,
  \end{equation*}
  then $\extantipode$ fulfills the antipode property\index{antipode!property}
  \begin{equation*}
   \mextproduct(\extantipode\etensor\id)\extcoproduct=\mextproduct(\id\etensor\extantipode)\extcoproduct=\unitmap\extcounit.
  \end{equation*}
 \end{enumerate}
\end{thm}
\begin{proof}~
 \begin{enumerate}[(i)]
  \item From associativity of $\product$, we conclude
  \begin{align*}
   \proj_n((a\extproduct b)\extproduct c)&=\sum_{i=0}^n  \proj_i(a\extproduct b)\product\proj_{n-i} c=\sum_{i=0}^n\sum_{j=0}^i (\proj_j a\product\proj_{i-j} b)\product\proj_{n-i} c\\
   &=\sum_{i=0}^n\sum_{j=0}^j \proj_i a\product(\proj_{i-j} b\product\proj_{n-i} c)=\sum_{k=0}^n\sum_{l=0}^{n-k} \proj_k a\product(\proj_l b\product\proj_{n-k-l} c)\\
   &=\sum_{k=0}^n\proj_k a\product\proj_{n-k}(b\extproduct c)=\proj_n(a\extproduct(b\extproduct c)).
  \end{align*}
  \item Coassociativity of $\coproduct$ implies
  \begin{align*}
   \projmult{3}_n(\extcoproduct\etensor\id)\extcoproduct&=(\coproduct\etensor\id)\projtwo_n\extcoproduct=(\coproduct\etensor\id)\coproduct\proj_n=(\id\etensor\coproduct)\coproduct\proj_n=(\id\etensor\coproduct)\projtwo_n\extcoproduct\\
   &=\projmult{3}_n(\id\etensor\extcoproduct)\extcoproduct.
  \end{align*}

  \item Due to the first compatibility requirement of Definition \ref{defn:bialgebra}, we have
  \begin{align*}
   \projtwo_n\extcoproduct\mextproduct&=\coproduct\proj_n\mextproduct=\coproduct\mproduct\projtwo_n=(\mproduct\etensor\mproduct)\flip_{1324}(\coproduct\etensor\coproduct)\projtwo_n=(\mproduct\etensor\mproduct)\flip_{1324}\projmult{4}_n(\extcoproduct\etensor\extcoproduct)\\
   &=
   (\mproduct\etensor\mproduct)\projmult{4}_n\flip_{1324}(\extcoproduct\etensor\extcoproduct)=\projtwo_n
   (\mextproduct\etensor\mextproduct)\flip_{1324}(\extcoproduct\etensor\extcoproduct)
  \end{align*}
  The second compatibility requirement follows from
  \begin{equation*}
   \extcoproduct\unit=\extcoproduct\proj_0\unit=\projtwo_0\coproduct\unit=\projtwo_0(\unit\etensor\unit)=\unit\etensor\unit.
  \end{equation*}
  Finally, from the original third compatibility requirement of Definition \ref{defn:bialgebra}, we conclude
  \begin{equation*}
   \extcounit\mextproduct=\counit\proj_0\mextproduct=\counit\mproduct\projtwo_0=(\counit\etensor\counit)\projtwo_0=\extcounit\etensor\extcounit.
  \end{equation*}
  \item From the original antipode property \eqref{eq:antipodeproperty}, it follows
  \begin{equation*}
   \proj_n\mextproduct(\id\etensor\extantipode)\extcoproduct=\mproduct\projtwo_n(\id\etensor\extantipode)=\mproduct(\id\etensor\antipode)\projtwo_n\extcoproduct=\mproduct(\id\etensor\antipode)\coproduct\proj_n=\unitmap\counit\proj_n=\proj_n\unitmap\extcounit,
  \end{equation*}
  since $\unitmap\counit\proj_n=0=\proj_n\unitmap\extcounit$ if $n\neq 0$ and otherwise $\unitmap\counit\proj_0=\unitmap\extcounit=\proj_0\unitmap\extcounit$. Likewise, $\proj_n\mextproduct(\extantipode\etensor\id)\extcoproduct=\proj_n\unitmap\extcounit$.
 \end{enumerate}
\end{proof}

In the following, we will simply also write $\product$, $\coproduct$, $\counit$, $\antipode$ for $\extproduct$, $\extcoproduct$, $\extcounit$, $\extantipode$ when it is clear from the context that we mean the extended operations on $\spacefs{G}$. 
\subsection{Lie Algebras}
\begin{defn}(Section 0.1 \cite{reutenauer93})
 A \emph{Lie algebra}\index{Lie!algebra} is a vector space $\liealg$ together with a bilinear map $\liebracket=[{\,\cdotp,\cdotp\,}]:\,\liealg\times\liealg\rightarrow\liealg$, the \emph{Lie bracket}\index{Lie!bracket}, with the properties
 \begin{enumerate}
  \item $[x,x]=0$ for all $x\in\liealg$,
  \item $\big[[a,b],c\big]+\big[[b,c],a\big]+\big[[c,a],b\big]=0$ for all $a,b,c\in\liealg$ (\emphind{Jacobi's identity}).
 \end{enumerate}
\end{defn}
\begin{rmk}(Section 0.1 \cite{reutenauer93})
 As for all $a,b\in\liealg$ we have
 \begin{equation*}
  0=[a+b,a+b]=[a,a]+[a,b]+[b,a]+[b,b]=[a,b]+[b,a],
 \end{equation*}
 the Lie bracket is \emph{antisymmetric}\index{antisymmetry}, \ie
 \begin{equation*}
  [a,b]=-[b,a]\quad\forallc a,b\in\liealg.
 \end{equation*}
\end{rmk}
\begin{defn}
 A sub Lie algebra of a Lie algebra $(\liealg,\liebracket)$ is a subspace $V\subseteq\liealg$ such that
 \begin{equation*}
  [V,V]\subseteq V.
 \end{equation*}
 The sub Lie algebra generated by a subspace $S\subseteq\liealg$, \ie the intersection of all sub Lie algebras containing $S$, is denoted by $\generate{S}{\liebracket}$.
\end{defn}
\begin{rmk}(Section 0.1 \cite{reutenauer93})
 For any algebra $(A,\product)$, there is a Lie algebra $(A,\liebracket\productindex)$ given by
 \begin{equation*}
  \liebracket\productindex(a,b):=[a,b]\productindex:=a\product b-b\product a\quad\forallc a,b\in A.
 \end{equation*}
 If A is an algebra graded by $(A_i)_i$, put $\liefs{A}:=\directproduct_{i=1}^\infty A_i$, which is a Lie algebra together with the Lie bracket $\liebracket\productindex$ of the extended product $\product$.
\end{rmk}
\begin{defn}(Based on Section 1.3 \cite{reutenauer93})
 For a Lie algebra $(\liealg,\liebracket)$, the linear operator $\rnbracketing:\,\tensoral{\liealg}\rightarrow\liealg$ recursively generated by
 \begin{equation*}
  \rnbracketing\unit=0,\quad\rnbracketing x=x,\quad\rnbracketing (x\itensor u)=[x,\rnbracketing u]\quad\forallc x\in\liealg,u\in\bigoplus_{n=1}^\infty\liealg^{\itensor n}
 \end{equation*}
 is called \emph{right norm bracketing}.
\end{defn}
\begin{thm}\textnormal{(Based on Section 0.4.1 \cite{reutenauer93})}
 For a subspace $W$ of a Lie algebra $(\liealg,\liebracket)$, we have
 \begin{equation*}
  \rnbracketing\tensoral{W}=\generate{W}{\liebracket}.
 \end{equation*}
\end{thm}
\begin{proof}('$\supseteq$' part of proof is taken from Section 0.4.1 \cite{reutenauer93})
 On the one hand, we have $\rnbracketing\tensoral{W}\subseteq\generate{W}{\liebracket}$, since if $\rnbracketing u\in\generate{W}{\liebracket}$ and $w\in W$, then also $\rnbracketing(w\itensor u)=[w,\rnbracketing u]\in\generate{W}{\liebracket}$, because $\generate{W}{\liebracket}$ is a sub Lie algebra containing $W$. On the other hand, each $v\in\generate{W}{\liebracket}$ can be written as a linear combination of nested Lie brackets. Since
 \begin{equation*}
  [[a,b],c]=-[[b,c],a]-[[c,a],b]=[a,[b,c]]+[b,[c,a]]=[a,[b,c]]-[b,[a,c]]\quad\forallc a,b,c\in\generate{W}{\liebracket}
 \end{equation*}
 due to Jacobi's identity and antisymmetry, we can inductively transform all nested Lie brackets into a linear combination of right norm bracketings, and thus we also have $\rnbracketing\tensoral{W}\supseteq\generate{W}{\liebracket}$.
\end{proof}

\subsection{Exponential map}
Let $A$ be a connectedly graded algebra, this time over a field of characteristic zero. Put $\groupfs{A}:=\unit+\liefs{A}\subseteq\spacefs{A}$ (Section II.3.\@ \cite{manchon06}). Define the maps $\exppro:\liefs{A}\rightarrow\groupfs{A}$ and $\logpro:\groupfs{A}\rightarrow\liefs{A}$ by
\begin{equation}\label{eq:defn_exppro_logpro}
  \exppro(u):=\sum_{n=0}^\infty\frac{u^{\product n}}{n!}\quad\text{and}\quad\logpro(\unit+u):=\sum_{n=1}^\infty\frac{(-1)^{n-1} u^{\product n}}{n}
\end{equation}
(Equations (3.1.1) and (3.1.2) \cite{reutenauer93}, Section II.3.\@ \cite{manchon06}).
\begin{thm}\label{thm:exppro_logpro_inverse_addition}~
 \begin{enumerate}
  \item \textnormal{(Equation (3.1.3) \cite{reutenauer93}, Section II.3.\@ \cite{manchon06}, Section 2.3 \cite{hairerkelly14})} The maps $\exppro$ and $\logpro$ are bijective and $\logpro$ is the inverse function of $\exppro$.
  \item \textnormal{(Based on the proof of Theorem 3.2 \cite{reutenauer93})} If $u_1,u_2\in\liefs{A}$ with $u_1\product u_2=u_2\product u_1$, then
  \begin{equation*}
   \exppro(u_1)\product\exppro(u_2)=\exppro(u_2)\product\exppro(u_1)=\exppro(u_1+u_2).
  \end{equation*}
  If $v_1,v_2\in\groupfs{A}$ with $v_1\product v_2=v_2\product v_1$, then
  \begin{equation*}
   \logpro(v_1)+\logpro(v_2)=\logpro(v_1\product v_2).
  \end{equation*}
  \item \textnormal{(Based on Proposition II.3.1.\@ \cite{manchon06})} The pair $(\groupfs{A},\product)$ forms a group. 
 \end{enumerate}
\end{thm}
We do not proof Statement 1.\@ of the theorem. It is a standard fact for the case of the tensor algebra $(\tensoral{V},\itensor)$, and the general case follows easily from it through looking at suitable homomorphisms.
\begin{proof}[Proof of the remainder]~
 \begin{enumerate}
 \setcounter{enumi}{1}
\item Let $u_1,u_2\in\liefs{A}$ with $u_1\product u_2=u_2\product u_1$ be arbitrary. Then,
  \begin{align*}
   \exppro(u_1+u_2)&=\sum_{n=0}^\infty\frac{(u_1+u_2)^{\product n}}{n!}=\sum_{n=0}^\infty\sum_{m=0}^n\binom{n}{m}\frac{u_1^{\product m}\product u_2^{\product(n-m)}}{n!}=\sum_{n=0}^\infty\sum_{m=0}^n\frac{u_1^{\product m}\product u_2^{\product(n-m)}}{m!(n-m)!}\\
   &=\sum_{k=0}^\infty\sum_{l=0}^\infty\frac{u_1^{\product k}\product u_2^{\product l}}{k!l!}=\exppro(u_1)\product\exppro(u_2).
  \end{align*}
  Let $v_1,v_2\in\groupfs{A}$ with $v_1\product v_2=v_2\product v_1$ be arbitrary. Then, we obviously have $\logpro(v_1)\product\logpro(v_2)=\logpro(v_2)\product\logpro(v_1)$, and thus
  \begin{equation*}
   \exppro(\logpro(v_1)+\logpro(v_2))=\exppro(\logpro(v_1))\product\exppro(\logpro(v_2))=v_1\product v_2=\exppro(\logpro(v_1+v_2)).
  \end{equation*}
  Due to injectivity of $\exppro$, we get $\logpro(v_1)+\logpro(v_2)=\logpro(v_1\product v_2)$.
  \item We have $\liefs{A}\product\liefs{A}\subseteq\liefs{A}$ due to the grading, and thus also $\groupfs{A}\product\groupfs{A}\subseteq\groupfs{A}$ since $\groupfs{A}=\unit+\liefs{A}$. Furthermore, since $\exppro$ is surjective, for any $y\in\groupfs{A}$, there is $x\in\liefs{A}$ such that $y=\exppro(x)$, and due to 2.,
  \begin{equation*}
   y\product\exppro(-x)=\exppro(x)\product\exppro(-x)=\exppro(x-x)=\exppro(0)=\unit.
  \end{equation*}
 \end{enumerate}

\end{proof}
\begin{thm}\label{thm:exppro_logpro_morphism}
 Let $(A,\productnum{1})$ and $(\hat{A},\productnum{2})$ be connectedly graded algebras and $\varLambda:\,\spacefs{A}\rightarrow\spacefs{\hat{A}}$ be an algebra homomorphism such that $\varLambda\unit=\unit$ and $\varLambda\proj_n v=\proj_n\varLambda v$ for all $v\in\spacefs{A},\,n\in\naturals_0$. Then,
 \begin{equation*}
  \varLambda\exp_{\productnum{1}}(u)=\exp_{\productnum{2}}(\varLambda u)\quad\forallc u\in\liefs{A}
 \end{equation*}
 and
 \begin{equation*}
  \varLambda\log_{\productnum{1}}(v)=\log_{\productnum{2}}(\varLambda v)\quad\forallc v\in\groupfs{A}.
 \end{equation*}
\end{thm}
\begin{proof} For each $u\in\liefs{A}$, we have
 \begin{equation*}
  \varLambda\exp_{\productnum{1}}(u)=\varLambda\sum_{n=0}^\infty\frac{u^{\productnum{1} n}}{n!}=\sum_{n=0}^\infty\frac{\varLambda u^{\productnum{1} n}}{n!}=\sum_{n=0}^\infty\frac{(\varLambda u)^{\productnum{2} n}}{n!}=\exp_{\productnum{2}}(\varLambda u),
 \end{equation*}
 were the interchangeability of the infinite sum and $\varLambda$ comes from the fact that $\varLambda\proj_n v=\proj_n\varLambda v$ for all $v\in\spacefs{A},\,n\in\naturals_0$. For each $\unit+u\in\liefs{A}$, we have
 \begin{align*}
  \varLambda\log_{\productnum{1}}(\unit+u)&=\varLambda\sum_{n=1}^\infty\frac{(-1)^{n-1} u^{\productnum{1} n}}{n}=\sum_{n=1}^\infty\frac{(-1)^{n-1} \varLambda u^{\productnum{1} n}}{n}=\sum_{n=1}^\infty\frac{(-1)^{n-1}(\varLambda u)^{\productnum{2} n}}{n}=\log_{\productnum{2}}(\unit+\varLambda u)\\
  &=\log_{\productnum{2}}(\varLambda(\unit+u)).
 \end{align*}

\end{proof}

\subsection{Truncations}
This subsection prepares a detailed treatment of the notion of truncations as they are used in Section 2.4 and in Chapter 4 \cite{hairerkelly14}. For an algebra $(A,\product)$ graded by $(A_i)_i$, the space $A^{>n}:=\bigoplus_{m=n+1}^\infty A_m$ is a two-sided ideal of $A$ as well as the space $\spacefs{A}^{>n}:=\directproduct_{m=n+1}^\infty A_m$ is an ideal of $(\spacefs{A},\product)$.

We assume that the following statement is standard, but did not find it in the literature we used. Part of the statement on quotient coalgebras is given in Section III.1 \cite{kassel95}, though.
\begin{thm}
 Let $I$ be a subspace of a vector space $V$ and $\quotientmap:\,V\rightarrow V\quotient I$ be the quotient map.
 \begin{enumerate}[(i)]
  \item If $(V,\product)$ is an algebra, there is a map $\product_{\quotientmap}:\,V\quotient I\times V\quotient I\rightarrow V\quotient I$ such that
  \begin{equation}\label{eq:defn_product_quotientmap}
   \linofbilin{\product_{\quotientmap}}(\quotientmap\etensor\quotientmap)=\quotientmap\mproduct
  \end{equation}
  iff $I$ is a two-sided ideal of $(V,\product)$. In this case, the map is unique and $(V\quotient I,\product_{\quotientmap})$ is an algebra, the \emph{quotient algebra}\index{quotient!algebra}.
  \item If $(V,\coproduct)$ is a coalgebra, there is a map $\coproduct_{\quotientmap}:\,V\quotient I\rightarrow V\quotient I\etensor V\quotient I$ such that
  \begin{equation*}
   \coproduct_{\quotientmap}\quotientmap=(\quotientmap\etensor\quotientmap)\coproduct
  \end{equation*}
  iff $I$ is a two-sided coideal of $(V,\coproduct)$. In this case, the map is unique and $(V\quotient I,\coproduct_{\quotientmap})$ is a coalgebra, the \emph{quotient coalgebra}\index{quotient!coalgebra}.
 \end{enumerate}
\end{thm}
\begin{proof}~
 \begin{enumerate}[(i)]
  \item If there is such a map $\product_{\quotientmap}$, then
  \begin{equation*}
   \quotientmap(I\product V+V\product I)=\quotientmap\mproduct(I\etensor V+V\etensor I)=\linofbilin{\product_{\quotientmap}}(\quotientmap\etensor\quotientmap)(I\etensor V+V\etensor I)=\{0\},
  \end{equation*}
  hence $I\product V+V\product I\in\ker\quotientmap=I$, thus $I$ is a two-sided ideal.
  
  \noindent If $I$ is a two sided ideal, then a map $\product_{\quotientmap}$ is well-defined through
  \begin{equation*}
   (\quotientmap v_1\product_{\quotientmap}\quotientmap v_2):=\quotientmap(v_1\product v_2).
  \end{equation*}
  Indeed, if $\quotientmap v_1=\quotientmap w_1$ and $\quotientmap v_2=\quotientmap w_2$, then
  \begin{equation*}
   v_1\etensor v_2-w_1\etensor w_2\in\ker(\quotientmap\etensor\quotientmap)=\ker\quotientmap\etensor V+V\etensor\ker\quotientmap=I\etensor V+V\etensor I 
  \end{equation*}
  hence
  \begin{equation*}
   v_1\product v_2-w_1\product w_2\in I\product V+V\product I\subseteq I
  \end{equation*}
  hence $\quotientmap(v_1\product v_2)=\quotientmap(w_1\product w_2)$.
  
  \noindent Uniqueness is clear, since we just showed that the demanded equation \eqref{eq:defn_product_quotientmap} is sufficient as a definition. The only thing left to show is associativity of $\product_{\quotientmap}$. If $v_1,v_2,v_3\in V$, then
  \begin{align*}
   (\quotientmap v_1\product_{\quotientmap}\quotientmap v_2)\product_{\quotientmap}\quotientmap v_3&=\quotientmap(v_1\product v_2)\product_{\quotientmap}\quotientmap v_3=\quotientmap((v_1\product v_2)\product v_3)=\quotientmap(v_1\product(v_2\product v_3))\\
   &=\quotientmap v_1\product_{\quotientmap}\quotientmap(v_2\product v_3)=\quotientmap v_1\product_{\quotientmap}(\quotientmap v_2\product_{\quotientmap}\quotientmap v_3).
  \end{align*}

  \item Due to the formal duality of algebra products and coproducts, the proof of the second statement is analogous to the proof of the first one.
 \end{enumerate}

\end{proof}

Due to the grading, we may identify both $\spacefs{A}\quotient\spacefs{A}^{>n}$ and $A\quotient A^{>n}$ canonically with $A^n:=\bigoplus_{m=0}^n A_m$.
The quotient map $\quotientmap^n:\spacefs{A}\rightarrow A^n$ then restricts to the other quotient map $\quotientmap^n{\restriction_A}:A\rightarrow A^n$ and is in both cases just the canonical projection onto $A^n$. In either case, we obtain the same quotient algebra $(A^n,\product_n)$, which is called the \emph{$n$-th grade truncation}\index{truncated!algebra} of the graded algebra $A$, or shortly a \emph{($n$-th grade) truncated algebra}.

On the other hand, starting from a graded coalgebra $(C,\coproduct)$ graded by $(C_i)_i$, the space $C^{>n}$ is not generally a coideal of $C$, as well as $\spacefs{C}^{>n}$ does not generally fulfill the coideal property with respect to the extended coalgebra $\spacefs{C}$. Fortunately though, $C^n$ is both a subcoalgebra of $C$ and fulfills the subcoalgebra property with respect to the extended coalgebra $\spacefs{C}$. We call $(C^n,\coproduct_n)$ where $\coproduct_n:=\coproduct{\restriction_{C^n}}$ the \emph{$n$-th grade truncation}\index{truncated!coalgebra} of the graded coalgebra $(C,\coproduct)$, or shortly a \emph{($n$-th grade) truncated coalgebra}. We will often simply write $(C^n,\coproduct)$ instead of $(C^n,\coproduct_n)$.

In the case of a graded bialgebra, we will combine the notions of truncated algebra and truncated coalgebra to that of a truncated bialgebra.

\begin{thm}
 Let $(B,\product,\coproduct)$ be a bialgebra graded by $(B_i)_i$ over the field $K$ with unit map $\unitmap$ and counit $\counit$. Then its \emph{$n$-th grade truncation} $(B^n,\product_n,\coproduct_n)$, called a \emph{($n$-th grade) truncated bialgebra}\index{truncated!bialgebra}\index{bialgebra!truncated@\textit{truncated}}, fulfills the \emph{truncated compatibility requirements}\index{compatibility requirements!truncated@\textit{truncated}}
 \begin{enumerate}
  \item $\coproduct_n\linofbilin{\product n}=\linofbilin{\producttwo n}(\coproduct_n\etensor\coproduct_n)\quad$ ($\coproduct_n$ is an algebra homomorphism from $(B^n,\product_n)$ to $\big([B\etensor B]^n,\producttwo_n\big)$),
  \item $\coproduct_n\unitmap=\unitmap\etensor\unitmap\quad$ ($\unitmap$ is a coalgebra homomorphism from $(K,\id_K)$ to $(B^n,\coproduct_n)$),
  \item $\counit_n\linofbilin{\product n}=\counit_n\etensor\counit_n\quad$ ($\counit_n$ is an algebra homomorphism from $(B^n,\product_n)$ to $(K,\cdot)$),
 \end{enumerate}
 where $\linofbilin{\product n}:\,B^n\etensor B^n\rightarrow B^n$ is the linear operator corresponding to $\product_n$, the pair $\big([B\etensor B]^n,\producttwo_n\big)$ is the truncation of the algebra $(B\etensor B,\producttwo)$ defined in \eqref{eq:producttwo} under the grading $([B\etensor B]_i)_i$ as defined in \eqref{eq:induced_grading} and $\linofbilin{\producttwo n}$ is the linear operator corresponding to $\producttwo_n$.
\end{thm}
\begin{proof}
 Let $\truncationoptwo{n}$ denote the projection of $B\etensor B$ onto $[B\etensor B]^n$.
 \begin{align*}
  \coproduct_n\linofbilin{\product n}(\truncationop{n}\etensor\truncationop{n})&=\coproduct_n\truncationop{n}\mproduct=\truncationoptwo{n}\coproduct\mproduct=\truncationoptwo{n}\mproducttwo(\coproduct\etensor\coproduct)=\linofbilin{\producttwo n}(\truncationoptwo{n}\etensor\truncationoptwo{n})(\coproduct\etensor\coproduct)\\
  &=\linofbilin{\producttwo n}(\coproduct_n\etensor\coproduct_n)(\truncationop{n}\etensor\truncationop{n})
 \end{align*}
 Since $\truncationop{n}\etensor\truncationop{n}$ is a surjective map into $B^n\etensor B^n$, we get $\coproduct_n\linofbilin{\product n}=\linofbilin{\producttwo n}(\coproduct_n\etensor\coproduct_n)$. The equation $\coproduct_n\unitmap=\unitmap\etensor\unitmap$ is clear since obviously $\coproduct_n\unit=\coproduct\unit=\unit\etensor\unit$. Finally,
 \begin{equation*}
  \counit_n\linofbilin{\product n}(\truncationop{n}\etensor\truncationop{n})=\counit_n\truncationop{n}\mproduct=\counit\mproduct=(\counit\etensor\counit)=(\counit_n\etensor\counit_n)(\truncationop{n}\etensor\truncationop{n}).
 \end{equation*}

\end{proof}

\begin{thm}
 Let $(H,\product,\coproduct,\antipode)$ be a Hopf algebra graded by $(H_i)_i$ with unit map $\unitmap$ and counit $\counit$. Then $\antipode_n:=\antipode{\restriction_{H^n}}$ fulfills the antipode property\index{antipode!property}
 \begin{equation*}
   \linofbilin{\product n}(\antipode_n\etensor\id)\coproduct_n=\linofbilin{\product n}(\id\etensor\antipode_n)\coproduct_n=\unitmap\counit_n
 \end{equation*}
 in its \emph{$n$-th grade truncation} $(B^n,\product_n,\coproduct_n,\antipode_n)$, which we call a \emph{($n$-th grade) truncated Hopf algebra}\index{truncated!Hopf algebra}\index{Hopf algebra!truncated@\textit{truncated}}.
\end{thm}
\begin{proof}
 \begin{align*}
  \linofbilin{\product n}(\antipode_n\etensor\id)\coproduct_n\truncationop{n}&=\linofbilin{\product n}(\antipode_n\etensor\id)\truncationoptwo{n}\coproduct=\linofbilin{\product n}\truncationoptwo{n}(\antipode\etensor\id)\coproduct=\linofbilin{\product n}(\truncationop{n}\etensor\truncationop{n})\truncationoptwo{n}(\antipode\etensor\id)\coproduct\\
  &=\truncationop{n}\mproduct\truncationoptwo{n}(\antipode\etensor\id)\coproduct=\truncationop{n}\mproduct(\antipode\etensor\id)\coproduct=\truncationop{n}\unitmap\counit=\unitmap\counit=\unitmap\counit\truncationop{n}
 \end{align*}

\end{proof}

Note that the $n$-th truncated algebra $(A^n,\product_n)$ is again a graded algebra. The grading is given by $(A^n_i)_i$ where $A^n_i=A_i$ for $i\leq n$ and $A^n_i=\{0\}$ for $i>n$. Hence, we get functions $\exp_{{\product n}}$ and $\log_{{\product n}}$ as described in the previous section.

\section{Primitive and group-like elements}

\begin{defn}(Based on Section I.5.\@ \cite{manchon06}, Section 2.3 \cite{hairerkelly14})
 \begin{enumerate}
  \item
  \begin{enumerate}[(i)]
   \item The subset of \emphind{group-like elements} of a coalgebra $(C,\coproduct)$ is given by
   \begin{equation*}
    \grouplike C:=\{c\in C\setminus\{0\}|\coproduct c=c\etensor c\}.
   \end{equation*}
   \item The subset of \emphind{group-like elements} of the space of formal series $\spacefs{C}$ of a graded coalgebra $(C,\coproduct)$ is given by
   \begin{equation*}
    \grouplike\spacefs{C}:=\{c\in\spacefs{C}\setminus\{0\}|\coproduct c=c\etensor c\}.
   \end{equation*}
   \item The subset of \emph{$n$-th grade truncated group-like elements}\index{group-like elements!truncated@\textit{truncated}} of a coalgebra $(C,\coproduct)$ graded by $(C_i)_i$ is given by
   \begin{equation*}
    \grouplike_n C:=\{c\in C\setminus\{0\}|\coproduct c=\quotientmaptwo^n(c\etensor c)\},
   \end{equation*}
   where $\quotientmaptwo^n:\,C\etensor C\rightarrow[C\etensor C]^n$ is the canonical projection.
  \end{enumerate}
  \item
  \begin{enumerate}[(i)]
   \item The subspace of \emphind{primitive elements} of a bialgebra $(B,\product,\coproduct)$ with unit element $\unit$ is given by
   \begin{equation*}
    \primitive B:=\{b\in B|\coproduct b=\unit\etensor b+b\etensor\unit\}.
   \end{equation*}
   We put $\primitive S:=S\cap\primitive B$ for any subspace $S$ of $B$.
   \item The subspace of \emphind{primitive elements} of the space of formal series $\spacefs{B}$ of a graded bialgebra $(B,\product,\coproduct)$ with unit element $\unit$ is given by
   \begin{equation*}
    \primitive\spacefs{B}:=\{b\in\spacefs{B}|\coproduct b=\unit\etensor b+b\etensor\unit\}.
   \end{equation*}
   We put $\primitive S:=S\cap\primitive\spacefs{B}$ for any subspace $S$ of $\spacefs{B}$.
  \end{enumerate}
 \end{enumerate}
\end{defn}
\begin{rmk}\label{rmk:grouplike_primitive_counit}~
 \begin{enumerate}[(i)]
  \item For any $c\in\grouplike C$, where $(C,\coproduct)$ is a coalgebra with counit $\counit$, we have $\counit(c)=1$ since $c\neq 0$ and
  \begin{equation*}
   c=(\counit\etensor\id)\coproduct c=(\counit\etensor\id)(c\etensor c)=\counit(c)c
  \end{equation*}
  due to the counit property \eqref{eq:counit_property}.
  
  \noindent Analogously, in the case of a graded coalgebra $C$, we have $\counit(c)=1$ for all $c\in\grouplike\spacefs{C}$.
  \item For any $b\in\primitive B$, where $(B,\product,\coproduct)$ is a coalgebra, we have $\counit(b)=0$ since
  \begin{equation*}
   b=(\counit\etensor\id)\coproduct b=(\counit\etensor\id)(\unit\etensor b+b\etensor\unit)=\counit(\unit) b+\counit(b)\unit=b+\counit(b)\unit.
  \end{equation*}
  by the counit property \eqref{eq:counit_property} and Remark \ref{rmk:counit_unit}.
  
  \noindent Analogously, in the case of a graded bialgebra $B$, we have $\counit(b)=0$ for all $b\in\primitive\spacefs{B}$.
 \end{enumerate}

\end{rmk}

\begin{rmk}\label{rmk:grouplike_spacefs_grouplike_n}
 For a graded coalgebra $(C,\coproduct)$, we have
 \begin{equation*}
  \quotientmap^n\grouplike\spacefs{C}\subseteq\grouplike_n C.
 \end{equation*}
 Indeed, if $c\in\grouplike\spacefs{C}$, then
 \begin{equation*}
  \coproduct\quotientmap^n c=\quotientmaptwo^n\coproduct c=\quotientmaptwo^n(c\etensor c)=\quotientmaptwo^n(\quotientmap^n c\etensor\quotientmap^n c)
 \end{equation*}
 and also $\quotientmap^n c\neq 0$ since $\counit(c)=1$.

\end{rmk}

A \emphind{character} of an algebra $(A,\product)$ is an element of $\adual{A}\setminus\{0\}$ which is an algebra homomorphism (Section II.4.\@ \cite{manchon06}).
\begin{thm}\label{thm:grouplike_character}\textnormal{(Based on Section 2.3 \cite{hairerkelly14})}
 Let $(A,\product)$ be the dual algebra of a coalgebra $(C,\coproduct)$ under some duality pairing $\dualbracket$.
 \begin{enumerate}
  \item An element $c\in C$ is group-like iff $\langle\,\cdot\,,c\rangle$ is a character of $(A,\product)$.
  \item If both $(A,\product)$ and $(C,\coproduct)$ are graded, then $c_1\in C$ is group-like iff $\langle\,\cdot\,,c_1\rangle$ is a character of $(\spacefs{A},\product)$ and $c_2\in\spacefs{C}$ is group-like iff $\langle\,\cdot\,,c_2\rangle$ is a character of $(A,\product)$.
 \end{enumerate}
\end{thm}
\noindent The proof is a basic application of the duality of products and coproducts (see Theorem \ref{thm:dual_hopf_bi_co_algebras}) and of the considerations in Subsection \ref{subsection:formal_series} and is skipped here.
\begin{thm}\textnormal{(Proposition I.7.3.\@ \cite{manchon06})}
 Let $(H,\product,\coproduct,\antipode)$ be a Hopf algebra with unit map $\unitmap$, unit element $\unit$ and counit $\counit$. Then, $\antipode x=-x$ for all $x\in\primitive H$. If $H$ is furthermore connectedly graded by $(H_i)_i$, we have $\antipode x=-x$ for all $x\in\primitive\spacefs{H}$.
\end{thm}
\begin{proof}(Proof of Proposition I.7.3.\@ \cite{manchon06})
 Let $x\in\primitive H$ be arbitrary. Then,
 \begin{equation*}
  0=\unitmap\counit(x)=\mproduct(\antipode\etensor\id)\coproduct x=\mproduct(\antipode\etensor\id)(\unit\etensor x+x\etensor\unit)=\antipode\unit\product x+\antipode x\product\unit=x+\antipode x,
 \end{equation*}
 thus $\antipode x=-x$.
 
 \noindent The proof for $x\in\primitive\spacefs{H}$ is completely analogous.
 
\end{proof}

\begin{thm}\textnormal{(Proposition I.7.3.\@ \cite{manchon06})}
 For a bialgebra $(B,\product,\coproduct)$, $\primitive B$ is a sub Lie algebra of $(B,\liebracket\productindex)$.
\end{thm}
\begin{proof}(Proof of Proposition I.7.3.\@ \cite{manchon06})
 Let $x,y\in\primitive B$ be arbitrary. Then,
 \begin{align*}
  \coproduct[x,y]\productindex&=\coproduct(x\product y-y\product x)=\coproduct x\producttwo\coproduct y-\coproduct y\producttwo\coproduct x\\
  &=(\unit\etensor x+x\etensor\unit)\producttwo(\unit\etensor y+y\etensor\unit)-(\unit\etensor y+y\etensor\unit)\producttwo(\unit\etensor x+x\etensor\unit)\\
  &=\unit\etensor(x\product y)+y\etensor x+x\etensor y+(x\product y)\etensor\unit-\unit\etensor(y\product x)-x\etensor y -y\etensor x-(y\product x)\etensor\unit\\
  &=\unit\etensor(x\product y-y\product x)+(x\product y-y\product x)\etensor\unit=\unit\etensor[x,y]\productindex-[x,y]\productindex\etensor\unit.  
 \end{align*}
 Thus, $[x,y]\productindex\in\primitive B$.\\
\end{proof}

\begin{thm}\textnormal{(Based on Theorem 1.4 and Lemma 1.5 \cite{reutenauer93})}
 Let $(H,\product,\coproduct,\antipode)$ be a Hopf algebra connectedly graded by $(H_i)_i$ over a field of characteristic zero, $\rnbracketing_{H_1}$ the right norm bracketing operator of $(H,\liebracket\productindex)$ restricted to $\tensoral{H_1}$, $D$ the derivation generated by
 \begin{equation*}
  D h_n=n h_n\quad\forallc h_n\in H_n,
 \end{equation*}
 furthermore $\convoproduct$ the convolution product on $\linear(H,H)$ and $E_{H_1}:(\tensoral{H_1},\itensor)\rightarrow(H,\product)$ the algebra homomorphism generated by
 \begin{equation*}
  E_{H_1}h_1=h_1\quad\forallc h_1\in H_1.
 \end{equation*}
 Then,
 \begin{equation}\label{eq:rnbracketing_D_convoproduct_antipode}
  \rnbracketing_{H_1}=(D\convoproduct\antipode)E_{H_1}
 \end{equation}
 and
 \begin{equation}\label{eq:D_convoproduct_antipode_on_primitive}
  (D\convoproduct\antipode){\restriction_{\primitive{H}}}=D{\restriction_{\primitive{H}}}.
 \end{equation}
\end{thm}
\begin{proof}(Based on the proofs of Theorem 1.4 and Lemma 1.5 \cite{reutenauer93})
 First of all,
 \begin{equation*}
  (D\convoproduct\antipode)E_{H_1}\unit=\mproduct(D\etensor\antipode)\coproduct\unit=D\unit\product\antipode\unit=0=\rnbracketing_{H_1}\unit
 \end{equation*}
 and, for all $v\in H_1$,
 \begin{equation*}
  (D\convoproduct\antipode)E_{H_1}v=\mproduct(D\etensor\antipode)\coproduct v=D\unit\etensor\antipode v+D v\product\antipode\unit=v=\rnbracketing_{H_1}v.
 \end{equation*}
 Assuming we have $(D\convoproduct\antipode)E_{H_1}\hat{u}=\rnbracketing_{H_1}\hat{u}$ for some $\hat{u}\in\tensoral{H_1}$ and putting $u:=E_{H_1}\hat{u}$, we have for all $x\in H_1$ that
 \begin{align*}
  (D\convoproduct\antipode)E_{H_1}(x\itensor\hat{u})&=\mproduct(D\etensor\antipode)\coproduct(x\itensor u)=\sum_{(x)}\sum_{(u)}D(x_1\product u_1)\product\antipode(x_2\product u_2)\\
  &=\sum_{(x)}\sum_{(u)}(D x_1\product u_1+x_1\product D u_1)\product\antipode u_2\product\antipode x_2\\
  &=\sum_{(u)}(D\unit\product u_1+\unit\product D u_1)\product\antipode u_2\product\antipode x+\sum_{(u)}(D x\product u_1+x\product D u_1)\product \antipode u_2\product \antipode\unit\\
  &=-\sum_{(u)}D u_1\product\antipode u_2\product x+\sum_{(u)}x\product u_1\product\antipode u_2+\sum_{(u)}x\product D u_1\product\antipode u_2\\
  &=-(D\convoproduct S)u\product x+x\product\unitmap\counit(u)+x\product(D\convoproduct S)u=-\rnbracketing_{H_1}\hat{u}\product x+x\product\rnbracketing_{H_1}\hat{u}\\
  &=[x,\rnbracketing_{H_1}\hat{u}]_{\product}=\rnbracketing_{H_1}(x\itensor\hat{u}).
 \end{align*}
 Thus, we get \eqref{eq:rnbracketing_D_convoproduct_antipode} via induction.
 \noindent Since for all $y\in\primitive{H}$ we have
 \begin{equation*}
  (D\convoproduct\antipode)y=\mproduct(D\etensor\antipode)\coproduct y=D\unit\etensor\antipode y+D y\etensor\antipode\unit=D y,
 \end{equation*}
 equation \eqref{eq:D_convoproduct_antipode_on_primitive} also follows.
 
\end{proof}

\begin{cor}\label{cor:lie_generate_primitive}\textnormal{(Based on Theorem 1.4 \cite{reutenauer93})}
 For a Hopf algebra $H$ connectedly graded by $(H_i)_i$ over a field of characteristic zero, we have
 \begin{equation*}
  \primitive\generate{H_1}{\product}=\generate{H_1}{\liebracket\productindex}.
 \end{equation*}
\end{cor}
\begin{proof}(Based on the proof of Theorem 1.4 \cite{reutenauer93})
 Since $\primitive\generate{H_1}{\product}$ is a Lie algebra that contains $H_1$, we have $\generate{H_1}{\liebracket\productindex}\subseteq\primitive\generate{H_1}{\product}$. On the other hand,
 \begin{align*}
  \primitive\generate{H_1}{\product}&=D\primitive\generate{H_1}{\product}=(D\convoproduct\antipode)\primitive\generate{H_1}{\product}\\
  &\subseteq(D\convoproduct\antipode)\generate{H_1}{\product}=(D\convoproduct\antipode)E_{H_1}\tensoral{H_1}=\rnbracketing_{H_1}\tensoral{H_1}= \generate{H_1}{\liebracket\productindex}.
 \end{align*}

\end{proof}

\begin{thm}\label{thm:primitive_grouplike}\textnormal{(Based on Theorems 3.1 and 3.2 \cite{reutenauer93})}
 For a Hopf algebra $(H,\product,\coproduct,\antipode)$ connectedly graded by $(H_i)_i$ over a field of characteristic zero, we have
 \begin{equation*}
  \primitive\spacefs{H}=\prod_{n=1}^\infty\primitive H_n\quad\text{and}\quad\exppro(\primitive\spacefs{H})=\grouplike\spacefs{H}.
 \end{equation*}
\end{thm}
\begin{proof}(Based on the proofs of Theorems 3.1 and 3.2 \cite{reutenauer93})
 If $x$ is in $\primitive\spacefs{H}$, then, due to the grading,
 \begin{equation*}
  \coproduct\proj_n x=\projtwo_n\coproduct x=\projtwo_n(\unit\etensor x+x\etensor\unit)=\unit\etensor\proj_n x+\proj_n x\etensor\unit.
 \end{equation*}
 On the other hand, of course, the spaces $\primitive H_n$ are linearly independent and for $x_n\in\primitive H_n$, we have
 \begin{equation*}
  \coproduct\sum_{n=1}^\infty x_n=\sum_{n=1}^\infty\coproduct x_n=\sum_{n=1}^\infty(\unit\etensor x_n+x_n\etensor\unit)=\unit\etensor(\sum_{n=1}^\infty x_n)+(\sum_{n=1}^\infty x_n)\etensor\unit.
 \end{equation*}
 For any $u\in\primitive\spacefs{H}$, using Theorem \ref{thm:exppro_logpro_inverse_addition} 2.\@ based on 
 \begin{equation*}
  (\unit\etensor u)\producttwo(u\etensor\unit)=u\etensor u=(u\etensor\unit)\producttwo(\unit\etensor u)
 \end{equation*}
 and using Theorem \ref{thm:exppro_logpro_morphism} for the algebra homomorphism $\coproduct:\,(\spacefs{H},\product)\rightarrow(\spacefs{(H\etensor H)},\producttwo)$, we have that
 \begin{align*}
  \coproduct\exppro(u)&=\exp_{\producttwo}(\coproduct u)=\exp_{\producttwo}(\unit\etensor u+u\etensor\unit)=\exp_{\producttwo}(\unit\etensor u)\producttwo\exp_{\producttwo}(u\etensor\unit)\\&=(\unit\etensor\exppro(u))\producttwo(\exppro(u)\etensor\unit)=\exppro(u)\etensor\exppro(u).
 \end{align*}
 For any $v\in\grouplike\spacefs{H}$, we have that $\counit(v)=1$ due to the extended counit property and thus $v\in\groupfs{H}$. Hence, we may compute, again using Theorems \ref{thm:exppro_logpro_inverse_addition} 2.\@ and \ref{thm:exppro_logpro_morphism},
 \begin{equation}\label{eq:logpro_grouplike_element}
  \begin{aligned}
    \coproduct\logpro(v)&=\log_{\producttwo}(\coproduct v)=\log_{\producttwo}(v\etensor v)=\log_{\producttwo}((\unit\etensor v)\producttwo(v\etensor\unit))\\
    &=\log_{\producttwo}(\unit\etensor v)+\log_{\producttwo}(v\etensor\unit)=\unit\etensor\logpro(v)+\logpro(v)\etensor\unit.
  \end{aligned}
 \end{equation}

\end{proof}

\begin{cor}\textnormal{(Based on Section 2.4 \cite{hairerkelly14})}\label{cor:primitive_n_grouplike_n}
 For a bialgebra $(B,\product,\coproduct)$ connectedly graded by $(B_i)_i$ over a field of characteristic zero, we have
 \begin{equation*}
  \exp_{\product_n}(\primitive B^n)=\quotientmap^n\grouplike\spacefs{B}=\grouplike_n B.
 \end{equation*}
\end{cor}
\begin{proof}
 For all $v\in\grouplike_n{H}$, we have, by Theorem \ref{thm:exppro_logpro_morphism} and the same calculation as in \eqref{eq:logpro_grouplike_element}, that
 \begin{align*}
  \coproduct\log_{\product_n}(v)&=\log_{\producttwo_n}(\coproduct v)=\log_{\producttwo_n}(\quotientmaptwo^n(v\etensor v))=\quotientmaptwo^n\log_{\producttwo}(v\etensor v)=\quotientmaptwo^n(\unit\etensor\log_\product(v)+\log_\product(v)\etensor\unit)\\
  &=\unit\etensor\quotientmap^n\log_\product(v)+\quotientmap^n\log_\product(v)\etensor\unit=\unit\etensor\log_{\product_n}(v)+\log_{\product_n}(v)\etensor\unit.
 \end{align*}
 Thus, $\log_{\product_n}(\grouplike_n H)\subseteq\primitive B^n$, implying $\grouplike_n H\subseteq\exp_{\product_n}(\primitive B^n)$.
 Again applying Theorem \ref{thm:exppro_logpro_morphism}, this time for the exponential map, we compute
 \begin{equation*}
  \exp_{\product n}(\primitive B^n)=\exp_{\product n}(\quotientmap^n\primitive\spacefs{B})=\quotientmap^n\exppro(\primitive\spacefs{B})=\quotientmap^n\grouplike\spacefs{B}.
 \end{equation*}
 Altogether, using Remark \ref{rmk:grouplike_spacefs_grouplike_n}, we finally get
 \begin{equation*}
  \grouplike_n H\subseteq\exp_{\product_n}(\primitive B^n)=\quotientmap^n\grouplike\spacefs{B}\subseteq\grouplike_n H.
 \end{equation*}
\end{proof}

\begin{example}~
 \begin{enumerate}
  \item In the case of polynomials, due to Corollary \ref{cor:lie_generate_primitive},
  \begin{equation*}
   \primitive\polynoms_d=\primitive\generate{\polynoms_{d,1}}{\propol}=\generate{\polynoms_{d,1}}{\liebracket_{\propol}}=\polynoms_{d,1}
  \end{equation*}
  since $\propol$ is commutative and hence also 
  \begin{equation*}
   \primitive\spacefs{(\polynoms_d)}=\polynoms_{d,1}=\vspan\{\X_i|1\leq i\leq d\}.
  \end{equation*} 
  Theorem \ref{thm:primitive_grouplike} then implies $\grouplike\spacefs{(\polynoms_d)}=\{\exp_{\propol}(\sum_{i=1}^d a_i\X_i)|(a_i)_i\subset\reals\}$.
  
  \item For the concatenation algebra, we again have $\generate{V}{\itensor}=\tensoral{V}$, hence 
  \begin{equation*}
   \primitive(\tensoral{V},\deshuffle)=\generate{V}{\liebracket_{\itensor}}.
  \end{equation*}
  The elements of $\primitive(\tensoral{V},\deshuffle)$ are called \emph{Lie polynomials}\index{Lie!polynomials} over $V$ (Section 1.3 \cite{reutenauer93}). We put $\tensoralfs{V}:=\spacefs{\tensoral{V}}$ and call
  \begin{equation*}
   \primitive(\tensoralfs{V},\deshuffle)=\prod_{i=1}^\infty\big(\generate{V}{\liebracket_{\itensor}}\cap V^{\itensor i}\big)
  \end{equation*}
  the space of \index{Lie!series} over $V$ (Section 3.1 \cite{reutenauer93}\footnote{Note that in the case where $V$ is infinite dimensional, the space of Lie series defined in \cite{reutenauer93} is strictly larger than the one presented here. According to our definition, Lie series always reduce to finite sums when projected down to $\tensoraltc{n}{V}$, which is not the case in \cite{reutenauer93}. The definitions agree if $V$ is finite dimensional.}).
  
  \item For the shuffle algebra, $\generate{V}{\shuffle}$ is a proper subspace of $\tensoral{V}$. Therefore, we may only conclude from Corollary \ref{cor:lie_generate_primitive} that
  \begin{equation*}
   \primitive(\generate{V}{\shuffle},\deconc)=\generate{V}{\liebracket_{\shuffle}}=V,
  \end{equation*}
  which follows from $\shuffle$ being commutative. But, looking at \eqref{eq:defn_deconc}, we see that every summand in the expansion of $\deconc u$ for some element of the form $u=v_1\itensor\cdots\itensor v_n$, $v_i\in V$ is specific to $u$, hence none of those summands can be canceled out by $\deconc w$ for an element $w$ of the same form as $u$, but linearly independent. Hence, there cannot be primitive elements of $\deconc$ in $V^{\itensor n}$ for $n>1$. Therefore,
  \begin{equation*}
   \primitive(\tensoralfs{V},\deconc)=\primitive(\tensoral{V},\deconc)=V
  \end{equation*}
  and, by Theorem \ref{thm:primitive_grouplike}, we get $\grouplike(\tensoralfs{V},\deconc)=\exp_{\shuffle}(V)$.

 \end{enumerate}
\end{example}
Interestingly, the example of $(\tensoralfs{K^2},\itensor,\deshuffle)$, $K$ being a field of characteristic zero, can be used to show the following.
\begin{thm}\textnormal{(Based on Corollaries 3.3 and 3.4 \cite{reutenauer93})}
 For an algebra $(A,\product)$ connectedly graded by $(A_i)_i$ over a field of characteristic zero, $\exppro(W)$ is a subgroup of $(\groupfs{A},\product)$ for any sub Lie algebra $W=\prod_{i=0}^\infty (W\cap A_i)$ of $(\liefs{A},\liebracket\productindex)$.
\end{thm}
\begin{proof}(Based on the proof of Corollary 3.3 \cite{reutenauer93})
 Let $\{e_1,e_2\}$ be a basis of $K^2$. Then, we of course have $e_1,e_2\in\primitive(\tensoralfs{K^2},\deshuffle)$ and hence $\exp(e_1),\exp(e_2)\in\grouplike(\tensoralfs{K^2},\deshuffle)$. We compute
 \begin{align*}
  \deshuffle(\exp_{\itensor}(e_1)\itensor\exp_{\itensor}(e_2))&=\deshuffle\exp_{\itensor}(e_1)\itensortwo\deshuffle\exp_{\itensor}(e_2)\\
  &=(\exp_{\itensor}(e_1)\etensor\exp_{\itensor}(e_1))\itensortwo(\exp_{\itensor}(e_2)\etensor\exp_{\itensor}(e_2))\\
  &=(\exp_{\itensor}(e_1)\itensor\exp_{\itensor}(e_2))\etensor(\exp_{\itensor}(e_1)\itensor\exp_{\itensor}(e_2)).
 \end{align*}
 Thus, 
 \begin{equation*}
  \exp_{\itensor}(e_1)\itensor\exp_{\itensor}(e_2)\in\grouplike(\tensoralfs{K^2},\deshuffle)=\exp_{\itensor}(\primitive\tensoralfs{K^2})=\exp_{\itensor}(\spacefs{(\generate{e_1,e_2}{\liebracket_{\itensor}})})
 \end{equation*}
 Now, let $v_1,v_2\in\exppro(W)$ be arbitrary. Let $\varLambda:\,(\tensoralfs{K^2},\itensor)\rightarrow(\spacefs{A},\product)$ be the unique algebra homomorphism such that $\varLambda e_1=\logpro v_1$, $\varLambda e_2=\logpro v_2$ and $\proj_n\varLambda x=\varLambda\proj_n x$ for all $x\in\tensoralfs{K^2}$. Then,
 \begin{equation*}
  \varLambda\exp_{\itensor}(e_i)=\varLambda\sum_{n=0}^\infty\frac{e_i^{\itensor n}}{n!}=\sum_{n=0}^\infty\frac{\varLambda (e_i^{\itensor n})}{n!}=\sum_{n=0}^\infty\frac{(\varLambda e_i)^{\product n}}{n!}=\exppro(\varLambda e_i)=v_i,
 \end{equation*}
 therefore
 \begin{equation*}
  v_1\product v_2=(\varLambda\exp_{\itensor}(e_1))\product(\varLambda\exp_{\itensor}(e_2))=\varLambda(\exp_{\itensor}(e_1)\itensor\exp_{\itensor}(e_2))
 \end{equation*}
 and finally
 \begin{align*}
  v_1\product v_2\in\varLambda\exp_{\itensor}(\spacefs{(\generate{e_1,e_2}{\liebracket_{\itensor}})})&=\exppro(\varLambda\spacefs{(\generate{e_1,e_2}{\liebracket_{\itensor}})})=\exppro(\spacefs{(\varLambda\generate{e_1,e_2}{\liebracket_{\itensor}})})\\
  &=\exppro(\spacefs{(\generate{\logpro v_1,\logpro v_2}{\liebracket_{\product}})})\subseteq\exppro(W).
 \end{align*}
 Since for any $v\in\exppro(W)$, we have $v^{\product -1}=\exppro(-\logpro(v))\in\exppro(W)$, the proof of $\exppro(W)$ being a subgroup is complete.
 
\end{proof}
\begin{rmk}
 The Lie series $\log_{\itensor}(\exp_{\itensor}(a)\itensor\exp_{\itensor}(b))$ (Corollary 3.4 \cite{reutenauer93}) is called the \emphind{Baker-Campbell-Hausdorff series} of $a,b\in\primitive(\tensoralfs{V},\deshuffle)$. See Chapter 3 \cite{reutenauer93} for more on this series.
\end{rmk}
\begin{rmk}(Based on Corollary 3.3 \cite{reutenauer93})
 We now especially know that for every Hopf algebra $(H,\product,\coproduct,\antipode)$ connectedly graded by $(H_i)_i$, the set of group-like elements $\grouplike\spacefs{H}=\exppro\primitive\spacefs{H}$ is a subgroup of $(\groupfs{H},\product)$. Since
 \begin{equation*}
  x\product\antipode x=\mproduct(\id\etensor\antipode)(x\etensor x)=\mproduct(\id\etensor\antipode)\coproduct x=\unitmap\counit(x)=\unitmap(1)=\unit,
 \end{equation*}
 where we made use of the antipode property \eqref{eq:antipodeproperty} and of Remark \ref{rmk:grouplike_primitive_counit}, we have $x^{\product-1}=\antipode x$ for all $x\in\grouplike\spacefs{H}$.
\end{rmk}
\vfill

\section{Connes-Kreimer Hopf algebra of trees and forests}
\subsection{Trees and forests}
In the following definition, we formally introduce the set of \emphind{forests} $\forests_I$ and the subset of \emphind{trees} $\trees_I$ \emph{labeled}\index{labels} or \emph{decorated}\index{decorations} by indices from some set $I$.
\begin{samepage}
\begin{defn}\textnormal{(Based on Section 2.2 \cite{hairerkelly14})}
 For a non-empty set $I$ let $\trees_I\subset\forests_I$ be sets, $\unit\in\forests_I$ an element, ${\treepro}:\,\forests_I\times\forests_I\rightarrow\forests_I$ an associative and commutative binary operation, $\lvert\cdot\rvert:\forests_I\rightarrow\naturals_0$ a function and $\lroot\cdot\rroot_i:\,\forests_I\rightarrow\trees_I,\,i\in I$ injective functions such that the following conditions are fulfilled.
 \begin{enumerate}[(i)]
  \item $\unit\treepro\zeta=\zeta\treepro\unit=\zeta\quad\forallc\zeta\in\forests_I$,
  \item $\trees_I$ is the disjoint union of the sets $\lroot\forests_I\rroot_i,\,i\in I$,
  \item For every $\zeta\in\forests_I\setminus\{\unit\}$, there is an $n\in\naturals$ and $\tau_1,\ldots,\tau_n\in\trees_I$ such that $\zeta=\tau_1\treepro\cdots\treepro\tau_n$ and this factorization is unique up to commutativity,
  \item $\lvert\zeta_1\treepro\zeta_2\rvert=\lvert\zeta_1\rvert+\lvert\zeta_2\rvert\quad\forallc\zeta_1,\zeta_2\in\forests_I$,
  \item $\lvert\lroot\zeta\rroot_i\rvert=\lvert\zeta\rvert+1\quad\forallc\zeta\in\forests_I,i\in I$,
  \item $\lvert\zeta\rvert=0\implies\zeta=\unit$.
 \end{enumerate}
\end{defn}
\end{samepage}
We call $\unit$ the \emphind{empty forest}. If $I$ consists only of one element $j$, we put $\forests:=\forests_{\{j\}}$, $\trees:=\trees_{\{j\}}$ and $\lroot\cdot\rroot:=\lroot\cdot\rroot_j$ and call $\forests$ the set of \emphind{undecorated forests}\index{forests!undecorated@\textit{undecorated}} and $\trees$ the set of \emph{undecorated trees}\index{trees!undecorated@\textit{undecorated}}.
\begin{table}\label{tab:tonewroot} 
 \centering
 \begin{tabular}{| c V{3} c c c |}
  \hline
  $n$                  & $\forests_n$ & $\rightarrow$              & $\trees_{n+1}$                        \\
  \hlineB{3}
  $0$                  & $\Fempty$    & $\mapsto$                  & $\Fnd$                                \\
  \hline
  $1$                  & $\Fnd$       & $\mapsto$                  & $\Fnndd$\vphantom{$\Fnnnddd$}         \\
  \hline
  \multirow{2}{*}{$2$} & $\Fndnd$     & \multirow{2}{*}{$\mapsto$} & $\Fnndndd$\vphantom{$\Fnnnddd$}       \\
                       & $\Fnndd$     &                            & $\Fnnnddd$\vphantom{$\Fnnnndddd$}     \\
  \hline
  \multirow{4}{*}{$3$} & $\Fndndnd$   & \multirow{4}{*}{$\mapsto$} & $\Fnndndndd$\vphantom{$\Fnnnddd$}     \\
                       & $\Fndnndd$   &                            & $\Fnndnnddd$\vphantom{$\Fnnnndddd$}   \\
                       & $\Fnndndd$   &                            & $\Fnnndnddd$\vphantom{$\Fnnnndddd$}   \\
                       & $\Fnnnddd$   &                            & $\Fnnnndddd$\vphantom{$\Fnnnnnddddd$} \\
  \hline
 \end{tabular}
 \caption{The action of $\tonewroot{\cdot}$.}
\end{table}
\subsection{Commutative Hopf algebra}
Let $\forests_{I,n}:=\big\{\zeta\in\forests_I\big|\lvert\zeta\rvert=n\big\}$ and $\trees_{I,n}:=\trees_I\cap\forests_{I,n}$. We extend $\treepro$ bilinearly from $\forests_I\times\forests_I$ to $\forestspace{I}\times\forestspace{I}$ and all $\lroot\cdot\rroot_i$ linearly from $\forests_I$ to $\lingen{\forests_I}$.
\begin{lemma}
 $(\forestspace{I},\treepro)$ is a commutative unitary algebra connectedly graded by $(\lingen{\forests_{I,n}})_n$.
\end{lemma}
\begin{defn}(Equations (48) to (50) \cite{conneskreimer98}, Equation (2.3) \cite{hairerkelly14})
 Let $\destar:\,\forestspace{I}\rightarrow\forestspace{I}\etensor\forestspace{I}$\nomenclature[00g D star]{$\destar$}{Connes-Kreimer coproduct of cuts of trees and forests} be the linear map recursively generated by
 \begin{equation}\label{eq:defn_destar}
  \destar\unit:=\unit\etensor\unit,\quad\destar\lroot\zeta\rroot_i:=\lroot\zeta\rroot_i\etensor\unit+(\id\etensor\lroot\cdot\rroot_i)\destar\zeta,\quad\destar(\zeta_1\treepro\zeta_2):=\destar\zeta_1\treeprotwo\destar\zeta_2,
 \end{equation}
 where $\zeta,\zeta_1,\zeta_2\in\forests_I$ and $\treeprotwo$ is again the canonical associative product on $\forestspace{I}\etensor\forestspace{I}$ (\cf \eqref{eq:producttwo}) generated by
 \begin{equation*}
  (\zeta_1\etensor\eta_1)\treeprotwo(\zeta_2\etensor\eta_2):=(\zeta_1\treepro\zeta_2)\etensor(\eta_1\treepro\eta_2).
 \end{equation*}
 Furthermore, let $\counit\in\adual{\forestspace{I}}$ be generated by
 \begin{equation*}
  \counit(\zeta):=\deltasymb{\unit,\zeta},\quad\zeta\in\forests_I.
 \end{equation*}
\begin{example}
 In order to find the value of $\destar\LFnnnddnnddd{1.5}{1}{1}{2}{3}{3}{2}$, we first compute generally those for the first two linear trees
 \begin{align*}
  \destar\LFnd{i}&=\LFnd{i}\etensor\unit+(\id\etensor\lroot\cdot\rroot_i)\destar\unit=\LFnd{i}\etensor\unit+\unit\etensor\LFnd{i},\\
  \destar\LFnndd{1}{j}{i}&=\LFnndd{1}{j}{i}\etensor\unit+(\id\etensor\lroot\cdot\rroot_j)\destar\LFnd{i}=\LFnndd{1}{j}{i}\etensor\unit+\LFnd{i}\etensor\LFnd{j}+\unit\etensor\LFnndd{1}{j}{i},
 \end{align*}
 and then specifically
 \begin{align*}
  \destar\LFnnddnndd{1.5}{1}{2}{3}{3}{2}&=\destar\LFnndd{1}{2}{3}\treeprotwo\destar\LFnndd{1}{3}{2}\\
  &=\LFnnddnndd{1.5}{1}{2}{3}{3}{2}\etensor\unit+\LFndnndd{1.5}{1}{3}{3}{2}\etensor\LFnd{2}+\LFnndd{1}{3}{2}\etensor\LFnndd{1}{2}{3}+\LFndnndd{1.5}{1}{2}{2}{3}\etensor\LFnd{3}+\LFndnd{1.5}{2}{3}\etensor\LFndnd{1.5}{2}{3}+\LFnd{2}\etensor\LFndnndd{1.5}{1}{3}{2}{3}+\LFnndd{1}{2}{3}\etensor\LFnndd{1}{3}{2}\\
  &\hphantom{\hphantom{}=\hphantom{}}+\LFnd{3}\etensor\LFndnndd{1.5}{1}{2}{3}{2}+\unit\etensor\LFnnddnndd{1.5}{1}{2}{3}{3}{2},\\
  \destar \LFnnnddnnddd{1.5}{1}{1}{2}{3}{3}{2}&=\LFnnnddnnddd{1.5}{1}{1}{2}{3}{3}{2}\etensor\unit+(\id\etensor\lroot\cdot\rroot_1)\destar\LFnnddnndd{1.5}{1}{2}{3}{3}{2}\\
  &=\LFnnnddnnddd{1.5}{1}{1}{2}{3}{3}{2}\etensor\unit+\LFnnddnndd{1.5}{1}{2}{3}{3}{2}\etensor\LFnd{1}+\LFndnndd{1.5}{1}{3}{3}{2}\etensor\LFnndd{1}{1}{2}+\LFnndd{1}{3}{2}\etensor\LFnnnddd{1}{1}{2}{3}+\LFndnndd{1.5}{1}{2}{2}{3}\etensor\LFnndd{1}{1}{3}+\LFndnd{1.5}{2}{3}\etensor\LFnndndd{1.5}{1}{1}{2}{3}+\LFnd{2}\etensor\LFnndnnddd{1.5}{1}{1}{3}{2}{3}+\LFnndd{1}{2}{3}\etensor\LFnnnddd{1}{1}{3}{2}\\
  &\hphantom{\hphantom{}=\hphantom{}}+\LFnd{3}\etensor\LFnndnnddd{1.5}{1}{1}{2}{3}{2}+\unit\etensor\LFnnnddnnddd{1.5}{1}{1}{2}{3}{3}{2}.
 \end{align*}
 By identification of the labels $2$ and $3$ we furthermore get
 \begin{equation*}
  \destar\LFnnnddnnddd{1.5}{1}{1}{2}{2}{2}{2}=\LFnnnddnnddd{1.5}{1}{1}{2}{2}{2}{2}\etensor\unit+\LFnnddnndd{1.5}{1}{2}{2}{2}{2}\etensor\LFnd{1}+2\,\LFndnndd{1.5}{1}{2}{2}{2}\etensor\LFnndd{1}{1}{2}+2\,\LFnndd{1}{2}{2}\etensor\LFnnnddd{1}{1}{2}{2}+\LFndnd{1.5}{2}{2}\etensor\LFnndndd{1.5}{1}{1}{2}{2}+2\,\LFnd{2}\etensor\LFnndnnddd{1.5}{1}{1}{2}{2}{2}+\unit\etensor\LFnnnddnnddd{1.5}{1}{1}{2}{2}{2}{2},
 \end{equation*}
 and identifying all the labels we have the undecorated case
 \begin{equation*}
  \destar\Fnnnddnnddd=\Fnnnddnnddd\etensor\unit+\Fnnddnndd\etensor\Fnd+2\,\Fndnndd\etensor\Fnndd+2\,\Fnndd\etensor\Fnnnddd+\Fndnd\etensor\Fnndndd+2\,\Fnd\etensor\Fnndnnddd+\unit\etensor\Fnnnddnnddd.
 \end{equation*}

\end{example}

\end{defn}
\begin{thm}\label{thm:treepro_destar_Hopf}\textnormal{(Based on Equations (52) and (53) \cite{conneskreimer98}, Section 5.1 \cite{foissy02}, Section 2.2 \cite{hairerkelly14})}
 $(\forestspace{I},\treepro,\destar,\antitree)$ is a commutative Hopf algebra connectedly graded by $(\lingen{\forests_{I,n}})_n$ with counit $\counit$, where the linear map $\antitree:\,\forestspace{I}\rightarrow\forestspace{I}$ is recursively generated by
 \begin{equation}\label{eq:defn_antitree}
  \antitree\unit:=\unit,\quad\antitree(\zeta_1\treepro\zeta_2):=\antitree\zeta_1\treepro\antitree\zeta_2,\quad\antitree\lroot\zeta\rroot_i:=-\mtreepro(\antitree\etensor\lroot\cdot\rroot_i)\destar\zeta.
 \end{equation}
\end{thm}
\begin{proof}~
 \begin{enumerate}
  \item\emph{$(\forestspace{I},\destar)$ is a coalgebra.} (Section 5.1 \cite{foissy02}) Obviously, $(\destar\etensor\id)\destar\unit=(\id\etensor\destar)\destar\unit$. If $\zeta_1,\zeta_2\in\forests_I$ are such that $(\destar\etensor\id)\destar\zeta_i=(\id\etensor\destar)\destar\zeta$, Lemma \ref{lemma:coassociative_subset} implies that 
  \begin{equation*}
   (\destar\etensor\id)\destar(\zeta_1\treepro\zeta_2)=(\id\etensor\destar)\destar(\zeta_1\treepro\zeta_2)
  \end{equation*}
  since $\destar$ is an algebra homomorphism by definition (see \eqref{eq:defn_destar}). Finally, if $(\destar\etensor\id)\destar\zeta=(\id\etensor\destar)\destar\zeta$, we also have
  \begin{align*}
   (\destar\etensor\id)\destar\lroot\zeta\rroot_i&=\lroot\zeta\rroot_i\etensor\unit\etensor\unit+(\id\etensor\lroot\cdot\rroot_i\etensor\unitmap)\destar\zeta+(\id\etensor\id\etensor\lroot\cdot\rroot_i)(\destar\etensor\id)\destar\zeta\\
   &=\lroot\zeta\rroot_i\etensor\unit\etensor\unit+(\id\etensor\lroot\cdot\rroot_i\etensor\unitmap)\destar\zeta+(\id\etensor\id\etensor\lroot\cdot\rroot_i)(\id\etensor\destar)\destar\zeta\\
   &=\lroot\zeta\rroot_i\etensor\unit\etensor\unit+(\id\etensor\destar)(\id\etensor\lroot\cdot\rroot_i)\destar\zeta=(\id\etensor\destar)\destar\lroot\zeta\rroot_i.
  \end{align*}
  Hence, we get via induction that $(\destar\etensor\id)\destar\zeta=(\id\etensor\destar)\destar\zeta$ for all $\zeta\in\forests_I$ and by linearity we conclude that $\destar$ is coassociative.
  \item\emph{$\counit$ is a counit.}
  First of all, we have $(\counit\etensor\id)\destar\unit=\counit(\unit)\unit=\unit=(\id\etensor\counit)\destar\unit$.
  If we assume that $(\counit\etensor\id)\destar\zeta=\zeta=(\id\etensor\counit)\destar\zeta$ holds for some $\zeta\in\forestspace{I}$, we get
  \begin{equation*}
   (\counit\etensor\id)\destar\lroot\zeta\rroot_i=\counit(\lroot\zeta\rroot_i)\unit+\lroot\cdot\rroot_i(\counit\etensor\id)\destar\zeta=\lroot\zeta\rroot_i
  \end{equation*}
  and
  \begin{equation*}
   (\id\etensor\counit)\destar\lroot\zeta\rroot_i=\counit(\unit)\lroot\zeta\rroot_i+(\id\etensor\counit\lroot\cdot\rroot_i)\destar\zeta=\lroot\zeta\rroot_i,
  \end{equation*}
  since $\counit\lroot\cdot\rroot_i=0$ by definition. Finally, if we assume that for some $\zeta_1,\zeta_2\in\forestspace{I}$ we have $(\counit\etensor\id)\destar\zeta_j=\zeta_j=(\id\etensor\counit)\destar\zeta_j$, we get
  \begin{equation*}
   (\counit\etensor\id)\destar(\zeta_1\treepro\zeta_2)=\zeta_1\treepro\zeta_2=(\id\etensor\counit)\destar(\zeta_1\treepro\zeta_2)
  \end{equation*}
  by again applying Lemma \ref{lemma:coassociative_subset}, since $\counit$ is an algebra homomorphism. Indeed,
  \begin{equation*}
   \counit(\zeta_1)\counit(\zeta_2)=\deltasymb{\unit,\zeta_1}\deltasymb{\unit,\zeta_2}=\deltasymb{\unit,\zeta_1\treepro\zeta_2}=\counit(\zeta_1\treepro\zeta_2).
  \end{equation*}

  Hence, by induction the equation
  \begin{equation*}
   (\counit\etensor\id)\destar\zeta=\zeta=(\id\etensor\counit)\destar\zeta
  \end{equation*}
  holds for all $\zeta\in\forestspace{I}$. The counit property \eqref{eq:counit_property} then follows via linearity.

  \item\emph{$(\forestspace{I},\treepro,\destar)$ is a bialgebra.} We have $\destar(\zeta_1\treepro\zeta_2)=\destar\zeta_1\treeprotwo\destar\zeta_2$ for all $\zeta_1,\zeta_2\in\forestspace{I}$ and $\destar\unit=\unit\etensor\unit$ by definition of $\destar$. Also, $\counit$ is an algebra homomorphism as we already saw. Thus, $(\forestspace{I},\treepro,\destar)$ is a bialgebra.
  \item\emph{$(\forestspace{I},\treepro,\destar,\antitree)$ is a connected graded Hopf algebra.} The fact that $(\lingen{\forests_{I,n}})_n$ is a grading for the bialgebra $(\forestspace{I},\treepro,\destar)$ is again easily shown by an induction over the recursive definition of $\destar$. It is also connected since $\lingen{\forests_{I,0}}=\lingen{\unit}$. By Theorem \ref{thm:antipode_recursion} we thus have that $(\forestspace{I},\treepro,\destar,\antipode)$ is a connected graded Hopf algebra with some unique antipode $\antipode$.
  Using $\counit\lroot\cdot\rroot_i=0$, the antipode property \eqref{eq:antipodeproperty} and \eqref{eq:defn_destar}, we get
  \begin{equation*}
   \antipode\lroot\zeta\rroot_i=-\unitmap\counit(\lroot\zeta\rroot_i)+\antipode\lroot\zeta\rroot_i=-\mtreepro(\antipode\etensor\id)\destar\lroot\zeta\rroot_i+\antipode\lroot\zeta\rroot_i=-\mtreepro(\antipode\etensor\lroot\cdot\rroot_i)\destar\zeta.
  \end{equation*}
  Since the antipode is an algebra antimorphism by Theorem \ref{thm:antipode_antimorphism} and $\treepro$ is commutative, we also have that
  \begin{equation*}
   \antipode(\zeta_1\treepro\zeta_2)=\antipode\zeta_1\treepro\antipode\zeta_2
  \end{equation*}
  and $\antipode\unit=\unit$.
  Thus, $\antipode$ is given by the recursion \eqref{eq:defn_antitree}, \ie $\antipode=\antitree$.

 \end{enumerate}
\end{proof}
\begin{rmk}
 Due to Theorem \ref{thm:antipode_recursion}, we may instead of the second expression in \eqref{eq:defn_antitree} use
 \begin{equation*}
  \antitree\lroot\zeta\rroot_i:=-\lroot\zeta\rroot_i-\mtreepro(\id\etensor\antitree)\rdestar\lroot\zeta\rroot_i,
 \end{equation*}
 which has the advantage that here, we do not need to calculate the antipode values for forests with more than one tree in order to get the antipode values for trees.
\end{rmk}

\begin{thm}\textnormal{(Based on Equation (51) \cite{conneskreimer98}, Equation (4) \cite{foissy02}, Section II.9.3.\@ \cite{manchon06}, Remark 2.9.\@ \cite{hairerkelly14})}
 $\destar$ admits the representation
 \begin{equation}\label{eq:destar_representation}
  \destar\zeta=\sum_{(C,T)\in\cuts{\zeta}}\cutfactor{\zeta}{C}{T}\,C\etensor T,
 \end{equation}
 where $(\cuts{\zeta})_{\zeta\in\forests_I}$ is a family of finite subsets of $\forests_I\times\forests_I$ recursively defined by
 \begin{align*}
  &\cuts{\unit}:=\{(\unit,\unit)\},\quad\cuts{\lroot\zeta\rroot_i}:=\big\{(C,\lroot T\rroot_i)\big|(C,T)\in\cuts{\zeta}\big\}\cup\{(\lroot\zeta\rroot_i,\unit)\},\\
  &\cuts{\zeta_1\zeta_2}:=\big\{(C_1 C_2,T_1 T_2)\big|(C_1,T_1)\in\cuts{\zeta_1},(C_2,T_2)\in\cuts{\zeta_2}\}
 \end{align*}
 and $(\cutfactor{\zeta}{\cdot}{\cdot})_{\zeta\in\forests_I}$ is a family of functions $(\cuts{\zeta}\rightarrow\naturals)_{\zeta\in\forests_I}$ recursively defined by
 \begin{align*}
  &\cutfactor{\tau}{\tau}{\unit}:=1,\quad\cutfactor{\lroot\zeta\rroot_i}{C}{\lroot T\rroot_i}:=\cutfactor{\zeta}{C}{T},\\
  &\cutfactor{\zeta_1\zeta_2}{C}{T}:=\sum_{\substack{(C_1,T_1)\in\cuts{\zeta_1},\,(C_2,T_2)\in\cuts{\zeta_2}:\\C_1 C_2=C,\,T_1 T_2=T}}\cutfactor{\zeta_1}{C_1}{T_1}\cutfactor{\zeta_2}{C_2}{T_2}.
 \end{align*}
\end{thm}
\begin{proof}
 First of all, we have
 \begin{equation*}
  \destar\unit=\unit\etensor\unit=\cutfactor{\unit}{\unit}{\unit}\,\unit\etensor\unit=\sum_{(C,T)\in\cuts{\unit}}\cutfactor{\zeta}{C}{T}\,C\etensor T.
 \end{equation*}
 Assuming $\destar\zeta$ admits the representation \eqref{eq:destar_representation} for some $\zeta\in\forests_I$, we get
 \begin{align*}
  \destar\lroot\zeta\rroot_i&=\lroot\zeta\rroot_i\etensor\unit+(\id\etensor\lroot\cdot\rroot_i)\destar\zeta=\cutfactor{\lroot\zeta\rroot_i}{\lroot\zeta\rroot_i}{\unit}\,\lroot\zeta\rroot_i\etensor\unit+\sum_{(C,T)\in\cuts{\zeta}}\cutfactor{\zeta}{C}{T}\,C\etensor\lroot T\rroot_i\\
  &=\cutfactor{\lroot\zeta\rroot_i}{\lroot\zeta\rroot_i}{\unit}\,\lroot\zeta\rroot_i\etensor\unit+\sum_{(C,T)\in\cuts{\zeta}}\cutfactor{\lroot\zeta\rroot_i}{C}{\lroot T\rroot_i}\,C\etensor\lroot T\rroot_i\\
  &=\sum_{(\bar C,\bar T)\in\cuts{\lroot\zeta\rroot_i}}\cutfactor{\lroot\zeta\rroot_i}{\bar C}{\bar T}\,\bar C\etensor\bar T.
 \end{align*}
 Assuming $\destar\zeta_1$ and $\destar\zeta_2$ admit the representation \eqref{eq:destar_representation} for some $\zeta_1,\zeta_2\in\forests_I$, we get
 \begin{align*}
  \destar(\zeta_1\zeta_2)&=\destar\zeta_1\treeprotwo\destar\zeta_2=\sum_{(C_1,T_1)\in\cuts{\zeta_1}}\sum_{(C_2,T_2)\in\cuts{\zeta_2}}\cutfactor{\zeta_1}{C_1}{T_1}\cutfactor{\zeta_2}{C_2}{T_2}\,C_1 C_2\etensor T_1 T_2\\
  &=\sum_{(C_1,T_1)\in\cuts{\zeta_1}}\sum_{(C_2,T_2)\in\cuts{\zeta_2}}\cutfactor{\zeta_1\zeta_2}{C_1 C_2}{T_1 T_2}\,C_1 C_2\etensor T_1 T_2\\
  &=\sum_{(C,T)\in\cuts{\zeta_1\zeta_2}}\cutfactor{\zeta_1\zeta_2}{C}{T}\,C\etensor T.
 \end{align*}
 The claim then follows inductively.\\
\end{proof}

Referring to Section 2.\@ \cite{conneskreimer98}, Section 4.1 \cite{foissy02} and Section II.9.3.\@ \cite{manchon06}, we may interpret the elements of $\cuts{\zeta}$ as the set of all cuts $(C,T)$ of the forest $\zeta$ into an upper part, the \emphind{crown} $C$, and a lower part, the \emphind{trunk} $T$. Such cuts are called \emph{admissible cuts}\index{admissible cut}. The coefficient $\cutfactor{\zeta}{C}{T}$ then gives the number of possibilities to draw the cut $(C,T)$ into a fixed planar drawing of the forest $\zeta$. The special case $(\unit,\zeta)\in\cuts{\zeta}$ is called the \emphind{empty cut} and $(\zeta,\unit)\in\cuts{\zeta}$ is called the \emphind{full cut}.
\begin{thm}
 $\antitree$ admits the representation
 \begin{equation}\label{eq:antitree_representation}
  \antitree\zeta=\sum_{S\in\splits{\zeta}}\splitfactor{\zeta}{S}S,
 \end{equation}
 where $(\splits{\zeta})_{\zeta\in\forests_I}$ is a family of finite subsets of $\forests_I$ recursively defined by
 \begin{equation*}
  \splits{\unit}:=\{\unit\},\quad\splits{\zeta}:=\big\{CS\big|\existsc T\neq\zeta:\,S\in\splits{T},\,(C,T)\in\cuts{\zeta}\big\}
 \end{equation*}
 and $(\splitfactor{\zeta}{\cdot})_{\zeta\in\forests_I}$ is a family of functions $(\splits{\zeta}\rightarrow\mathbb{Z})_{\zeta\in\forests_I}$ recursively defined by
 \begin{equation*}
  \splitfactor{\unit}{\unit}:=1,\quad\splitfactor{\lroot\zeta\rroot_i}{\bar S}:=-\sum_{\substack{(C,T)\in\cuts{\lroot\zeta\rroot_i},\,S\in\splits{T}:\\T\neq\lroot\zeta\rroot_i,\,C S=\bar S}}\splitfactor{T}{S}\cutfactor{\lroot\zeta\rroot_i}{C}{T},
 \end{equation*}
 \begin{equation*}
  \splitfactor{\zeta_1\zeta_2}{S}:=\sum_{\substack{S_1\in\splits{\zeta_1},\,S_2\in\splits{\zeta_2}:\\S_1 S_2=S}}\splitfactor{\zeta_1}{S_1}\splitfactor{\zeta_2}{S_2}.
 \end{equation*}

\end{thm}
\begin{proof}~
 \begin{enumerate}
 \item\emph{$\splits{\zeta_1\zeta_2}=\splits{\zeta_1}\treepro\splits{\zeta_2}$ for all $\zeta_1,\zeta_2\in\forests_I$.}
 We first have
 \begin{equation*}
  \splits{\unit}\treepro\splits{\zeta}=\{\unit\}\treepro\splits{\zeta}=\splits{\zeta}=\splits{\unit\zeta}
 \end{equation*}
 for all $\zeta\in\forests_I$, in particular for all trees $\zeta$. 
 
 \noindent Assuming the claim holds for all $\zeta_1,\zeta_2\in\forests_I$ with $\lvert\zeta_1\rvert+\lvert\zeta_2\rvert\leq n$, we get for all $\eta_1,\eta_2\in\forests_I$ with $\lvert\eta_1\rvert+\lvert\eta_2\rvert\leq n+1$ that
 \begin{align*}
  \splits{\eta_1\eta_2}&=\big\{CS\big|\existsc T\neq\eta_1\eta_2:\,S\in\splits{T},\,(C,T)\in\cuts{\eta_1\eta_2}\big\}\\
  &=\big\{C_1 C_2 S\big|\existsc T_1 T_2\neq\eta_1\eta_2:\,S\in\splits{T_1 T_2},\,(C_1,T_1)\in\cuts{\eta_1},\\
  &\hphantom{\hphantom{}=\big\{}(C_2,T_2)\in\cuts{\eta_2}\big\}\\
  &=\big\{C_1 C_2 S_1 S_2\big|\existsc T_1 T_2\neq\eta_1\eta_2:\,S_1\in\splits{T_1},\,S_2\in\splits{T_2},\\
  &\hphantom{\hphantom{}=\big\{}(C_1,T_1)\in\cuts{\eta_1},\,(C_2,T_2)\in\cuts{\eta_2}\big\}\\
  &=\big\{C_1 C_2 S_1 S_2\big|\existsc T_1\neq\eta_1,\,T_2\neq\eta_2:\,S_1\in\splits{T_1},\,S_2\in\splits{T_2},\\
  &\hphantom{\hphantom{}=\big\{}(C_1,T_1)\in\cuts{\eta_1},\,(C_2,T_2)\in\cuts{\eta_2}\big\}\\
  &=\splits{\eta_1}\treepro\splits{\eta_2},
 \end{align*}
 where in the third equality we used the inductive assumption with the fact that for $(C,T)\in\cuts{\zeta}$, we have $\lvert T\rvert\geq\lvert\zeta\rvert$ if and only if $T=\zeta$. In order to see the $\subseteq$ part of the fourth equality, the case of either $T_1=\eta_1$ or $T_2=\eta_2$ needs further explanation. Without loss of generality, assume $T_1=\eta_1$, then we have $C_1=\unit$ and since $S_1\in\splits{T_1}=\splits{\eta_1}$, there is $(C_1',T_1')\in\cuts{\eta_1}$ and $S_1'\in\splits{T_1'}$ such that $T_1'\neq\eta_1$ and $C_1'S_1'=S_1=C_1 S_1$. Hence, this case is indeed included in what comes after the fourth equality.
 
 \noindent The claim then follows by induction over $n$.
 \item\emph{$\antitree$ admits the given representation.}
 First of all,
 \begin{equation*}
  \antitree\unit=\unit=\splitfactor{\unit}{\unit}\unit=\sum_{\eta\in\splits{\unit}}\splitfactor{\zeta}{\eta}\eta.
 \end{equation*}
 Assuming $\antitree\zeta$ admits the representation \eqref{eq:antitree_representation} for some $\zeta\in\forests_I$, we get
 \begin{align*}
  \antitree\lroot\zeta\rroot_i&=-\lroot\zeta\rroot_i-\mtreepro(\id\etensor\antitree)\rdestar\lroot\zeta\rroot_i=-\sum_{\substack{(C,T)\in\cuts{\lroot\zeta\rroot_i}:\\T\neq\lroot\zeta\rroot_i}}\cutfactor{\lroot\zeta\rroot_i}{C}{T}C\treepro\antitree T\\
  &=-\sum_{\substack{(C,T)\in\cuts{\lroot\zeta\rroot_i}:\\T\neq\lroot\zeta\rroot_i}}\,\sum_{S\in\splits{T}}\splitfactor{T}{S}\cutfactor{\lroot\zeta\rroot_i}{C}{T}CS\\
  &=-\sum_{\bar S\in\splits{\lroot\zeta\rroot_i}}\,\sum_{\substack{(C,T)\in\cuts{\lroot\zeta\rroot_i},\,S\in\splits{T}:\\T\neq\lroot\zeta\rroot_i,\,C S=\bar S}}\splitfactor{T}{S}\cutfactor{\lroot\zeta\rroot_i}{C}{T}CS\\
  &=\sum_{\bar S\in\splits{\lroot\zeta\rroot_i}}\splitfactor{\lroot\zeta\rroot_i}{\bar S}\bar S.
\end{align*}
 Assuming $\antitree\zeta_1$ and $\antitree\zeta_2$ admit the representation \eqref{eq:antitree_representation} for some $\zeta_1,\zeta_2\in\forests_I$, we get
 \begin{align*}
  \antitree(\zeta_1\zeta_2)&=\antitree\zeta_1\treeprotwo\antitree\zeta_2=\sum_{S_1\in\splits{\zeta_1}}\sum_{S_2\in\splits{\zeta_2}}\splitfactor{\zeta_1}{S_1}\splitfactor{\zeta_2}{S_2}S_1 S_2\\
  &=\sum_{S\in\splits{\zeta_1\zeta_2}}\,\sum_{\substack{S_1\in\splits{\zeta_1},\,S_2\in\splits{\zeta_2}:\\S_1 S_2=S}}\splitfactor{\zeta_1}{S_1}\splitfactor{\zeta_2}{S_2}S_1 S_2=\\
  &=\sum_{S\in\splits{\zeta_1\zeta_2}}\,\sum_{\substack{S_1\in\splits{\zeta_1},\,S_2\in\splits{\zeta_2}:\\S_1 S_2=S}}\splitfactor{\zeta_1}{S_1}\splitfactor{\zeta_2}{S_2}S=\sum_{S\in\splits{\zeta_1\zeta_2}}\splitfactor{\zeta_1\zeta_2}{S}S.
 \end{align*}
 \end{enumerate}
\end{proof}

For a tree $\tau$, each forest $\eta$, except $\tau$ itself, in the set $\splits{\tau}$ is achieved by multiplying all the forests resulting from a finite sequence of consecutive admissible cuts of $\tau$. The coefficient $\splitfactor{\tau}{\eta}$ then gives the number of possibilities to draw such sequences of odd length into a fixed planar drawing of the tree $\tau$, minus the number of possibilities with even length, where the order of the cuts in the sequence doesn't count, each cut may only appear once in the sequence and empty as well as full cuts are excluded. For the special case $\tau\in\splits{\tau}$, we have $\splitfactor{\tau}{\tau}=-1$.
\subsection{Cocommutative dual Hopf algebra}
\begin{defn}
 Let $\star:\,\forestspace{I}\times\forestspace{I}$\nomenclature[ZZs star]{$\star$}{Product on trees and forests dual to the coproduct $\destar$} be the bilinear map generated by
 \begin{equation*}
  C\star T:=\sum_{\substack{\zeta\in\forests_I:\\(C,T)\in\cuts{\zeta}}}\cutfactor{\zeta}{C}{T}\zeta,\quad C,T\in\forests_I,
 \end{equation*}
 furthermore $\detreepro:\,\forestspace{I}\rightarrow\forestspace{I}\etensor\forestspace{I}$ the linear map generated by
 \begin{equation*}
  \detreepro\zeta:=\sum_{\substack{\zeta_1,\zeta_2\in\forests_I:\\\zeta_1\zeta_2=\zeta}}\zeta_1\etensor\zeta_2,\quad\zeta\in\forests_I
 \end{equation*}
 and $\antistar:\,\forestspace{I}\rightarrow\forestspace{I}$ the linear map generated by
 \begin{equation*}
  \antistar\eta:=\sum_{\substack{\zeta\in\forests_I:\\\eta\in\splits{\zeta}}}\splitfactor{\zeta}{\eta}\zeta,\quad\zeta\in\forests_I
 \end{equation*}
 Finally, let $\dualbracket:\,\forestspace{I}\times\forestspace{I}\rightarrow\reals$ be bilinearly generated by
 \begin{equation*}
  \left\langle\zeta,\eta\right\rangle:=\deltasymb{\zeta,\eta},\quad\zeta,\eta\in\forests_I.
 \end{equation*}
\end{defn}
\noindent As noted in Section 2.2 \cite{hairerkelly14}, the product $\star$ is often called \emphind{Grossman-Larson product}, where the name refers to the paper \cite{grossmanlarson89}.
\begin{thm}\textnormal{(Based on Section 2.2 \cite{hairerkelly14})}
 $(\forestspace{I},\star,\detreepro,\antistar)$ is the cocommutative dual Hopf algebra of $(\forestspace{I},\treepro,\destar,\antitree)$ under the duality pairing $\dualbracket$. It is connectedly graded by $(\lingen{\forests_{I,n}})_n$, too.
\end{thm}
\begin{proof}~
 \begin{enumerate}
  \item\emph{$\mstar$ is the dual operator of $\destar$.} For all $C,T,\eta\in\forests_I$, we have
  \begin{align*}
   \langle\mstar(C\etensor T),\eta\rangle&=\sum_{\substack{\zeta\in\forests_I:\\(C,T)\in\cuts{\zeta}}}\cutfactor{\zeta}{C}{T}\deltasymb{\zeta,\eta}=\sum_{(\bar C,\bar T)\in\cuts{\eta}}\cutfactor{\eta}{C}{T}\deltasymb{C,\bar C}\deltasymb{T,\bar T}\\
   &=\sum_{(\bar C,\bar T)\in\cuts{\eta}}\cutfactor{\eta}{\bar C}{\bar T}\deltasymb{C,\bar C}\deltasymb{T,\bar T}=\sum_{(\bar C,\bar T)\in\cuts{\eta}}\cutfactor{\eta}{\bar C}{\bar T}\langle C\etensor T,\bar C\etensor\bar T\rangle\\
   &=\langle C\etensor T,\destar\eta\rangle.
  \end{align*}
  \item\emph{$\detreepro$ is the dual operator of $\mtreepro$.} For all $\zeta,\eta_1,\eta_2\in\forests_I$, we have
  \begin{equation*}
   \langle\detreepro\zeta,\eta_1\etensor\eta_2\rangle=\sum_{\substack{\zeta_1,\zeta_2\in\forests_I:\\\zeta_1\zeta_2=\zeta}}\deltasymb{\zeta_1,\eta_1}\deltasymb{\zeta_2,\eta_2}=\deltasymb{\zeta,\eta_1\eta_2}=\langle\zeta,\eta_1\treepro\eta_2\rangle.
  \end{equation*}
  \item\emph{$\antistar$ is the dual operator of $\antitree$.} For all $\eta,\zeta\in\forests_I$, we have
  \begin{equation*}
   \langle\antistar\eta,\zeta\rangle=\sum_{\substack{\zeta'\in\forests_I:\\\eta\in\splits{\zeta'}}}\splitfactor{\zeta'}{\eta}\deltasymb{\zeta',\zeta}=\sum_{\eta'\in\splits{\zeta}}\splitfactor{\zeta}{\eta}\deltasymb{\eta',\eta}=\sum_{\eta'\in\splits{\zeta}}\splitfactor{\zeta}{\eta'}\deltasymb{\eta',\eta}=\langle\eta,\antitree\zeta\rangle.
  \end{equation*}
  \item\emph{$(\forestspace{I},\star,\detreepro,\antistar)$ is the cocommutative dual Hopf algebra of $(\forestspace{I},\treepro,\destar,\antitree)$.} We obviously have that $\unitmap$ is the dual operator of $\counit$ and vice versa. Since $(\forestspace{I},\treepro,\destar,\antitree)$ is a commutative Hopf algebra with unit map $\unitmap$ and counit $\counit$ according to Theorem \ref{thm:treepro_destar_Hopf}, we get from the first three steps using Theorem \ref{thm:dual_hopf_bi_co_algebras} that $(\forestspace{I},\star,\detreepro,\antistar)$ is its cocommutative dual Hopf algebra with unit map $\unitmap$ and counit $\counit$.
  \item\emph{$(\forestspace{I},\star,\detreepro,\antistar)$ is graded.} $\star$ is graded by $(\lingen{\forests_{I,n}})_n$ since $\lvert C\rvert+\lvert T\rvert=\lvert\zeta\rvert$ for all $\zeta\in\forests_I$ and $(C,T)\in\cuts{\zeta}$. $\detreepro$ is graded since $\lvert\zeta_1\rvert+\lvert\zeta_2\rvert=\lvert\zeta_1\zeta_2\rvert$ for all $\zeta_1,\zeta_2\in\forests_I$. $\antistar$ is graded since $\lvert S\rvert=\lvert\eta\rvert$ for all $\eta\in\forests_I$ and $S\in\splits{\eta}$.

 \end{enumerate}

\end{proof}

Obviously, $\primitive(\forestspace{I},\detreepro)=\lingen{\trees_I}$. Hence, $\grouplike(\spacefs{\forestspace{I}},\detreepro)=\exp_{\star}\spacefs{\lingen{\trees_I}}$. As noted in Section 2.3 \cite{hairerkelly14}, this group is also called the \emphind{Butcher group}.
\subsection{Homomorphisms between words and forests}
Based on the terminologies used in \cite{hairerkelly14} and in \cite{reutenauer93}, we call
\begin{equation*}
 \words_I:=\{\unit\}\cup\{e_{i_1\cdots i_n}|n\in\naturals,i_j\in I\},\quad e_a\neq e_b\text{ if }a\neq b,
\end{equation*}
\nomenclature[W I]{$\words_I$}{Set of words over an alphabet $I$}the set of \emphind{words} over the alphabet $I$, where $\unit$ is called the \emphind{empty word}. Define $\lvert\cdot\rvert:\words_I\rightarrow\naturals_0$ via $\lvert e_{i_1\cdots i_n}\rvert:=n$ and $\lvert\unit\rvert:=0$. The number $\lvert w\rvert$ is called the \emph{length}\index{length of a word} of the word $w$. Put $\words_I^n:=\{w\in\words_I:\lvert w\rvert\leq n\}$. We may then identify $\tensoral{\lingen{I}}$ with $\lingen{\words_I}$ and $\tensoraltc{n}{\lingen{I}}$ with $\lingen{\words_I^n}$ via $e_{i_1\cdots i_m}:=e_{i_1}\itensor\cdots\itensor e_{i_m}$. The duality pairing $\dualbracket:\,\words_I\times\words_I\rightarrow\reals$ generated by
\begin{equation*}
 \langle w_1,w_2\rangle:=\deltasymb{w_1,w_2},\quad w_1,w_2\in\words_I
\end{equation*}
is then also an inner product.
\begin{thm}\label{thm:phi_Hopf_morphism}\textnormal{(Section 4.1 \cite{hairerkelly14})}
 The linear operator 
 \begin{equation*}
  \phi:\,(\forestspace{I},\treepro,\destar,\antitree)\rightarrow(\lingen{\words_I},\shuffle,\deconc,\antishuffle)
 \end{equation*}
 recursively generated by
 \begin{equation}\label{eq:defn_phi_Hopf_morphism}
  \phi(\unit):=\unit,\quad\phi(\lroot\zeta\rroot_i):=\phi(\zeta)\itensor e_i,\quad\phi(\zeta_1\zeta_2):=\phi(\zeta_1)\shuffle\phi(\zeta_2)
 \end{equation}
 is a Hopf algebra homomorphism.
 The linear operator $\hat\phi:\,(\lingen{\words_I},\deconc)\rightarrow(\forestspace{I},\destar)$ recursively generated by
 \begin{equation*}
  \hat\phi(\unit):=\unit,\quad\hat\phi(w\itensor e_i):=\lroot\hat\phi(w)\rroot_i
 \end{equation*}
 is a coalgebra monomorphism. We have $\phi\hat\phi=\id$.
\end{thm}
\begin{proof}
 Putting $L_i x:=x\itensor e_i$ for all $x\in\lingen{\words_I}$, we get directly from the definition \eqref{eq:defn_deconc} that
 \begin{equation*}
  \deconc L_i w=\deconc(w\itensor e_i)=(w\itensor e_i)\etensor\unit+\sum_{(w)}^{\itensor}w_1\etensor(w_2\itensor e_1)=L_i w\etensor\unit+(\id\etensor L_i)\deconc w.
 \end{equation*}
 Since $\phi\hat\phi\unit=\unit$ and $\phi\hat\phi(w\itensor e_i)=\phi(\lroot\hat\phi(w)\rroot_i)=\phi\hat\phi(w)\itensor e_i$, we get $\phi\hat\phi=\id$ inductively. In particular, this shows that $\hat\phi$ is injective.
 \begin{enumerate}
  \item\emph{$\phi$ is a Hopf algebra homomorphism.} $\phi$ is obviously an algebra homomorphism by the third equation in \eqref{eq:defn_phi_Hopf_morphism}. Also, we have $\phi\unitmap=\unitmap$ by definition and $\counit\phi\zeta=\deltasymb{\phi\zeta,\unit}=\deltasymb{\zeta,\unit}=\counit\zeta$ for all $\zeta\in\forests_I$. By Theorem \ref{thm:hopf_morphism}, it only remains to show that $\phi$ is a coalgebra homomorphism.
  For this, we first have
  \begin{equation*}
   \deconc\phi(\unit)=\deconc\unit=\unit\etensor\unit=(\phi\etensor\phi)(\unit\etensor\unit)=(\phi\etensor\phi)\destar\unit.
  \end{equation*}
  Assuming we have $\deconc\phi(\zeta)=(\phi\etensor\phi)\destar\zeta$ for some $\zeta\in\forests_I$, we get
  \begin{align*}
   (\phi\etensor\phi)\destar\lroot\zeta\rroot_i&=(\phi\etensor\phi)(\lroot\zeta\rroot_i\etensor\unit)+(\phi\etensor\phi)(\id\etensor\lroot\cdot\rroot_i)\destar\zeta\\
   &=L_i\phi(\zeta)\etensor\unit+(\id\etensor L_i)(\phi\etensor\phi)\destar\zeta=L_i\phi(\zeta)\etensor\unit+(\id\etensor L_i)\deconc\phi(\zeta)\\
   &=\deconc L_i\phi(\zeta)=\deconc\phi(\lroot\zeta\rroot_i).
  \end{align*}
  Assuming we have $\deconc\phi(\zeta_1)=(\phi\etensor\phi)\destar\zeta_1$ and $\deconc\phi(\zeta_2)=(\phi\etensor\phi)\destar\zeta_2$ for some $\zeta_1,\zeta_2\in\forests_I$, we get
  \begin{align*}
   \deconc\phi(\zeta_1\zeta_2)&=\deconc\phi(\zeta_1)\shuffletwo\deconc\phi(\zeta_2)=(\phi\etensor\phi)\destar\zeta_1\shuffletwo(\phi\etensor\phi)\destar\zeta_2\\&=(\phi\etensor\phi)(\destar\zeta_1\treeprotwo\destar\zeta_2)=(\phi\etensor\phi)\destar(\zeta_1\zeta_2).
  \end{align*}
  \item\emph{$\hat\phi$ is a coalgebra homomorphism.} As always, $\destar\hat\phi(\unit)=(\hat\phi\etensor\hat\phi)\deconc\unit$ follows immediately. Assuming we have $\destar\hat\phi(w)=(\hat\phi\etensor\hat\phi)\deconc w$ for some $w\in\words_I$, we get
  \begin{align*}
   \destar\hat\phi(w\itensor e_i)&=\destar\lroot\hat\phi(w)\rroot_i=\lroot\hat\phi(w)\rroot_i\etensor\unit+(\id\etensor\lroot\cdot\rroot_i)\destar\hat\phi(w)\\
   &=\hat\phi(w\itensor e_i)\etensor\unit+(\id\etensor\lroot\cdot\rroot_i)(\hat\phi\etensor\hat\phi)\deconc w\\
   &=(\hat\phi\etensor\hat\phi)(L_i w\etensor\unit)+(\hat\phi\etensor\hat\phi)(\id\etensor L_i)\deconc w=(\hat\phi\etensor\hat\phi)\deconc L_i w.
  \end{align*}
 \end{enumerate}
\end{proof}

\begin{thm}\textnormal{(Lemma 4.9.\@ \cite{hairerkelly14} and its proof)}\newline
 The map $\psi:\,(\forestspace{I},\treepro,\destar,\antitree)\rightarrow(\tensoral{\lingen{\trees_I}},\shuffle,\deconc,\antishuffle)$ recursively linearly generated by
 \begin{equation}\label{eq:defn_monomorph_forests_tensor}
  \psi(\unit):=\unit,\quad\psi(\lroot\zeta\rroot_i):=\mitensor(\psi\etensor\lroot\cdot\rroot_i)\destar\zeta,\quad\psi(\zeta_1\zeta_2):=\psi(\zeta_1)\shuffle\psi(\zeta_2)
 \end{equation}
 is a Hopf algebra monomorphism.
\end{thm}
\begin{proof}~
 \begin{enumerate}
  \item \emph{$\psi$ is a Hopf algebra homomorphism.} $\psi$ is obviously an algebra homomorphism by the third equation in \eqref{eq:defn_monomorph_forests_tensor}.
  We first have
  \begin{equation*}
   \deconc\psi(\unit)=\deconc\unit=\unit\etensor\unit=(\psi\etensor\psi)(\unit\etensor\unit)=(\psi\etensor\psi)\destar\unit.
  \end{equation*}
  Assuming we have $\deconc\psi(\zeta)=(\psi\etensor\psi)\destar\zeta$ for some $\zeta\in\forests_I$, we get
  \begin{align*}
   \deconc\psi(\lroot\zeta\rroot_i)=\deconc\mitensor(\psi\etensor\lroot\cdot\rroot_i)\destar\zeta&=\mitensortwo(\deconc\etensor\unitmap\etensor\id)(\psi\etensor\lroot\cdot\rroot_i)\destar\zeta+\psi(\lroot\zeta\rroot_i)\etensor\unit\\
   &=\mitensortwo(\psi\etensor\psi\etensor\unitmap\etensor\lroot\cdot\rroot_i)(\destar\etensor\id)\destar\zeta+\psi(\lroot\zeta\rroot_i)\etensor\unit\\
   &=\mitensortwo(\psi\etensor\psi\etensor\unitmap\etensor\lroot\cdot\rroot_i)(\id\etensor\destar)\destar\zeta+\psi(\lroot\zeta\rroot_i)\etensor\unit\\
   &=(\psi\etensor\mitensor)(\id\etensor\psi\etensor\lroot\cdot\rroot_i)(\id\etensor\destar)\destar\zeta+\psi(\lroot\zeta\rroot_i)\etensor\unit\\
   &=(\psi\etensor\psi)(\id\etensor\lroot\cdot\rroot_i)\destar\zeta+\psi(\lroot\zeta\rroot_i)\etensor\unit\\
   &=(\psi\etensor\psi)\destar\lroot\zeta\rroot_i.
  \end{align*}
  Assuming we have $\deconc\psi(\zeta_1)=(\psi\etensor\psi)\destar\zeta_1$ and $\deconc\psi(\zeta_2)=(\psi\etensor\psi)\destar\zeta_2$ for some $\zeta_1,\zeta_2\in\forests_I$, we get
  \begin{align*}
   \deconc\psi(\zeta_1\treepro\zeta_2)&=\deconc\psi(\zeta_1)\shuffletwo\deconc\psi(\zeta_2)=(\psi\etensor\psi)\destar\zeta_1\shuffletwo(\psi\etensor\psi)\destar\zeta_2\\&=(\psi\etensor\psi)(\destar\zeta_1\treeprotwo\destar\zeta_2)=(\psi\etensor\psi)\destar(\zeta_1\treepro\zeta_2).
  \end{align*}
  \item \emph{$\psi$ is injective.} See the proof of Lemma 4.9.\@ \cite{hairerkelly14}.
 \end{enumerate}

\end{proof}

\chapter{Rough paths}\label{chapter:rough_paths}
\section{Geometric rough paths}
We shortly write $\words_d:=\words_{\{1,\ldots,d\}}$ for all $d\in\naturals$.
Define
\begin{align*} 
 \geogroupfs{d}:=\grouplike(\tensoralfs{\reals^d},\deshuffle)=\{g\in\tensoralfs{\reals^d}|\deshuffle g=g\etensor g\}=\exp_{\itensor}(\geoliefs{d}),
\end{align*}
where
\begin{equation*}
 \geoliefs{d}:=\primitive(\tensoralfs{\reals^d},\deshuffle)=\{l\in\tensoralfs{\reals^d}|\deshuffle l=\unit\etensor l+l\etensor\unit\}=\spacefs{\big(\generate{\reals^d}{\liebracket_{\itensor}}\big)},
\end{equation*}
as well as
\begin{align*} 
 \geogroup{n}{d}:=\grouplike^{n}(\tensoral{\reals^d},\deshuffle)&=\{g\in\tensoraltc{n}{\reals^d}|\deshuffle g=\quotientmap^n (g\etensor g)\}=\quotientmap^n\geogroupfs{d}=\exp_{\itensor_n}(\geolie{n}{d}),
\end{align*}
where
\begin{equation*}
 \geolie{n}{d}:=\primitive(\tensoraltc{n}{\reals^d},\deshuffle)=\{l\in\tensoraltc{n}{\reals^d}|\deshuffle l=\unit\etensor l+l\etensor\unit\}=\quotientmap^n\generate{\reals^d}{\liebracket_{\itensor}}.
\end{equation*}
Define $\lVert\cdot\rVert_\geogroup{n}{d}:\,\geogroup{n}{d}\rightarrow[0,\infty)$ by (Section~4 \cite{hairerkelly14})
\begin{equation*}
 \lVert g\rVert_\geogroup{n}{d}:=\sum_{m=1}^n\lVert\proj_m\log_{\itensor_n}(g)\rVert^{1/m},
\end{equation*}
where $\lVert l\rVert:=\sqrt{\langle l,l\rangle}$. Put $\lVert\cdot\rVert_n:=\lVert\proj_n \cdot\rVert_n$.
\begin{defn}\label{defn:geometric_rX}(Equation (4.3) and Definition 4.1.\@ \cite{hairerkelly14})
 Let $\gamma\in(0,1)$. A $d$-dimensional $\gamma$-Hölder \emph{weakly geometric rough path}\index{rough path!geometric@\textit{geometric}}\index{geometric rough path} is a map $\rX:\,\timeT\rightarrow\geogroup{\level}{d}$ with $\rX_0:=\unit$ such that
 \begin{equation*}
  \sup_{s<t}\frac{\lVert\rX_{st}\rVert_\geogroup{\level}{d}}{\lvert t-s\rvert^{\gamma}}<\infty,
 \end{equation*}
 where $\rX_{st}:=\rX_s^{\itensor_{\level}{-1}}\itensor_\level\rX_t=\antitensor\rX_s\itensor_\level\rX_t$ and $\level$ is the integer part of $\tfrac{1}{\gamma}$. The set of all such maps $\rX:\,\timeT\rightarrow\geogroup{\level}{d}$ is denoted by $\geopaths^\gamma([0,T],\reals^d)$.
\end{defn}
\begin{rmk}
 The word 'weakly' comes from the fact that there is a notion of geometric rough paths which is a bit stronger. For details on this distinction, see \cite{frizvictoir06}.
\end{rmk}

We present three more definitions which will turn out to be equivalent.
\begin{subdefn}\label{defn:geometric_fullrX}(Based on Equation (4.3) and Definition 4.1.\@ \cite{hairerkelly14})
 Let $\gamma\in(0,1)$. A $d$-dimensional $\gamma$-Hölder \emph{weakly geometric rough path}\index{rough path!geometric@\textit{geometric}}\index{geometric rough path} is a map $\fullrX:\,\timeT\rightarrow\geogroupfs{d}$ with $\fullrX_0:=\unit$ such that
 \begin{equation*}
  \sup_{s<t}\frac{\lVert\quotientmap^n\fullrX_{st}\rVert_\geogroup{n}{d}}{\lvert t-s\rvert^{\gamma}}<\infty\quad\forallc n\in\naturals,
 \end{equation*}
 where $\fullrX_{st}:=\fullrX_s^{\itensor{-1}}\itensor\fullrX_t=\antitensor\fullrX_s\itensor\fullrX_t$. The set of all such maps $\fullrX:\,\timeT\rightarrow\geogroupfs{d}$ is denoted by $\geopathsfull^\gamma([0,T],\reals^d)$.
\end{subdefn}
\begin{subdefn}\label{defn:geometric_dualrX}(Definition 1.2.\@ \cite{hairerkelly14})
 Let $\gamma\in(0,1)$. A $d$-dimensional $\gamma$-Hölder \emph{weakly geometric rough path} is a map $\dualrX:\,\timeT^2\rightarrow\adual{\lingen{\words_d^\level}}$ such that
 \begin{enumerate}
  \item $\dualrX_{st}(\unit)=1$ and $\dualrX_{st}(w_1\shuffle w_2)=\dualrX_{st}(w_1)\dualrX_{st}(w_2)$ for all words $w_1,w_2\in\words_d^\level$ such that $\lvert w_1\rvert+\lvert w_2\rvert\leq\level$,
  \item $\dualrX_{tt}=\counit_\level$ and $\dualrX_{st}(w)=(\dualrX_{su}\etensor\dualrX_{ut})\deconc w=\sum_{(w)}^{\itensor}\dualrX_{su}(w^1)\dualrX_{ut}(w^2)$ for all words $w\in\words_d^\level$,
  \item $\sup_{s\neq t}\frac{\lvert\dualrX_{st}(w)\rvert}{\lvert t-s\rvert^{\gamma\lvert w\rvert}}<\infty$ for all words $w\in\words_d^\level$,
 \end{enumerate}
 where $\level$ is the integer part of $\tfrac{1}{\gamma}$. Put $\dualrX_t:=\dualrX_{0t}$. The set of all such maps $\dualrX:\,\timeT^2\rightarrow\adual{\lingen{\words_d^\level}}$ is denoted by $\geopathsdual^\gamma([0,T],\reals^d)$.
\end{subdefn}
\begin{subdefn}\label{defn:geometric_dualfullrX}(Based on Definition 1.2.\@ \cite{hairerkelly14})
 Let $\gamma\in(0,1)$. A $d$-dimensional $\gamma$-Hölder \emph{weakly geometric rough path} is a map $\dualfullrX:\,\timeT^2\rightarrow\adual{\lingen{\words_d}}$ such that
 \begin{enumerate}
  \item $\dualfullrX_{st}(\unit)=1$ and $\dualfullrX_{st}(w_1\shuffle w_2)=\dualfullrX_{st}(w_1)\dualfullrX_{st}(w_2)$ for all words $w_1,w_2\in\words_d$,
  \item $\dualfullrX_{tt}=\counit$ and $\dualfullrX_{st}(w)=(\dualfullrX_{su}\etensor\dualfullrX_{ut})\deconc w=\sum_{(w)}^{\itensor}\dualfullrX_{su}(w^1)\dualfullrX_{ut}(w^2)$ for all words $w\in\words_d$,
  \item $\sup_{s\neq t}\frac{\lvert\dualfullrX_{st}(w)\rvert}{\lvert t-s\rvert^{\gamma\lvert w\rvert}}<\infty$ for all words $w\in\words_d$.
 \end{enumerate}
 Put $\dualfullrX_t:=\dualfullrX_{0t}$. The set of all such maps $\dualfullrX:\,\timeT^2\rightarrow\adual{\lingen{\words_d}}$ is denoted by $\geopathsdualfull^\gamma([0,T],\reals^d)$.
\end{subdefn}
\begin{thm}\label{thm:lyons_extension}\textnormal{\textbf{Lyons' Extension Theorem} (Theorem 3.7 \cite{lyonscaruanalevy07})}
 Let $\gamma\in(0,1)$ and $\level$ be the integer part of $\gamma^{-1}$. Let $X:\,[0,T]^2\rightarrow\tensoraltc{\level}{\reals^d}$ be such that $\proj_0 X_{st}=1$ for all $s,t\in[0,T]$, such that
 \begin{equation*}
  X_{su}\itensor_\level X_{ut}=X_{st}\quad\forallc s,u,t\in[0,T]
 \end{equation*}
 and such that
 \begin{equation}\label{eq:lyons_extension_analytic_condition}
  \sup_{s<t}\frac{\lVert X_{su}\rVert_m}{\lvert u-s\rvert^{m\gamma}}<\infty\quad\forallc m\in\{0,\ldots,\level\}.
 \end{equation}
 Then, there is a unique extension $\bar{X}:\,[0,T]^2\rightarrow\tensoralfs{\reals^d}$ such that $\quotientmap^\level\bar{X}_{st}=X_{st}$ for all $s,t\in[0,T]$, such that
 \begin{equation}\label{eq:lyons_extension_analytic_result}
  \bar{X}_{su}\itensor \bar{X}_{ut}=\bar{X}_{st}\quad\forallc s,u,t\in[0,T]
 \end{equation}
 and such that
 \begin{equation*}
  \sup_{s<t}\frac{\lVert\bar{X}_{su}\rVert_m}{\lvert u-s\rvert^{m\gamma}}<\infty\quad\forallc m\in\naturals_0.
 \end{equation*}
\end{thm}
\begin{thm}\label{thm:lyons_extension_geogroup}\textnormal{(Corollary 3.9 \cite{cassdriverlimlitterer15})}
 In the context of Theorem \ref{thm:lyons_extension}, if $X_{st}\in\geogroup{\level}{d}$ for all $s,t\in[0,T]^2$, then $\bar{X}_{st}\in\geogroupfs{d}$ for all $s,t\in[0,T]^2$.
\end{thm}
Note that for our setting of paths in $\reals^d$, this result is actually known for several years already, but in \cite{cassdriverlimlitterer15}, it is shown for the more general case of paths in a real Banach space. 

\begin{lemma}\label{lemma:geo_norms}\textnormal{(Based on Proposition 4 \cite{lyonsvictoir07})}~
 \begin{enumerate}
  \item For all $k\in\naturals$, there is a constant $C_k>0$ such that
  \begin{equation*}
   \lvert\langle w,g\rangle\rvert\leq C_{\lvert w\rvert}\lVert g\rVert_{\geogroup{n}{d}}^{\lvert w\rvert}\quad\forallc n\in\naturals,\,g\in\geogroup{n}{d},\,w\in\words_d^n\setminus\{\unit\}.
  \end{equation*}
  \item For all $n\in\naturals$, there is a constant $C_n'>0$ such that
  \begin{equation*}
   \lVert g\rVert_{\geogroup{n}{d}}\leq C_n'\sup_{w\in\words_d^n\setminus\{\unit\}}\lvert\langle w,g\rangle\rvert^{1/\lvert w\rvert}\quad\forallc g\in\geogroup{n}{d}.
  \end{equation*}
 \end{enumerate}
\end{lemma}
\begin{proof}(Based on the proof of Proposition 4 \cite{lyonsvictoir07})
 \begin{enumerate}
  \item Let $n\in\naturals_0$ and $x\in\mathrm{L}_d^n$ be arbitrary. For every $m,k\in\naturals$ with $m\leq n$, we have
  \begin{align*}
   \lVert x^{\itensor_n m}\rVert_k&=\bigg\lVert\Big(\sum_{v\in\words_d^n\setminus\{\unit\}}v\langle v,x\rangle\Big)^m\bigg\rVert_k\\
   &=\bigg\lVert\sum_{\substack{v_1,\ldots,v_m\in\words_d^k\setminus\{\unit\}:\\\lvert v_1\rvert+\ldots+\lvert v_m\rvert=k}}\langle v_1,x\rangle\cdots\langle v_m,x\rangle\,v_1\itensor_n\cdots\itensor_n v_m\bigg\rVert_k\\
   &\leq\sum_{\substack{v_1,\ldots,v_m\in\words_d^k\setminus\{\unit\}:\\\lvert v_1\rvert+\cdots+\lvert v_m\rvert=k}}\,\prod_{i=1}^m\lvert v_i,x\rvert\leq N_{m,k}\sup_{\substack{v_1,\ldots,v_m\in\words_d^k\setminus\{\unit\}:\\\lvert v_1\rvert+\cdots+\lvert v_m\rvert=k}}\,\prod_{i=1}^m\lvert v_i,x\rvert\\
   &\leq N_{m,k}\sup_{\substack{l_1,\ldots,l_m\in\naturals:\\l_1+\cdots+l_m=k}}\,\prod_{i=1}^m\lVert x\rVert_{l_i}\\
   &\leq N_{m,k}\sup_{\substack{l_1,\ldots,l_m\in\naturals:\\l_1+\cdots+l_m=k}}\,\prod_{i=1}^m\lVert\exp_{\itensor_n}(x)\rVert_{\geogroup{n}{d}}^{l_i}=N_{m,k}\lVert\exp_{\itensor_n}(x)\rVert_{\geogroup{n}{d}}^k,
  \end{align*}
  where $N_{m,k}:=\big\lvert\big\{v_1,\ldots,v_m\in\words_d^k\setminus\{\unit\}\big|\lvert v_1\rvert+\cdots+\lvert v_m\rvert=k\big\}\big\rvert<\infty$. Thus, for any $w\in\words_d^n$, we get
  \begin{align*}
   \big\lvert\langle w,\exp_{\itensor_n}(x)\rangle\big\rvert&=\left\lvert\sum_{m=0}^n\frac{1}{m!}\langle w,x^{\itensor_n m}\rangle\right\rvert=\left\lvert\sum_{m=0}^{\lvert w\rvert}\frac{1}{m!}\langle w,\proj_{\lvert w\rvert}x^{\itensor_n m}\rangle\right\rvert\\
   &\leq\sum_{m=0}^{\lvert w\rvert}\frac{1}{m!}\left\lVert x^{\itensor_n m}\right\rVert_{\lvert w\rvert}\leq\underbrace{\sum_{m=0}^{\lvert w\rvert}\frac{N_{m,\lvert w\rvert}}{m!}}_{:=C_{\lvert w\rvert}}\,\lVert\exp_{\itensor_n}(x)\rVert_{\geogroup{n}{d}}^{\lvert w\rvert}.
  \end{align*}
  \item Let $n\in\naturals_0$ and $\unit+u\in\geogroup{n}{d}$ be arbitrary. For every $m,k\in\naturals$ with $m\leq n$, we have, with $N_{m,k}$ and similar first steps as above,
  \begin{multline*}
   \lVert u^m\lVert_k^{1/k}\leq N_{m,k}^{1/k}\sup_{\substack{v_1,\ldots,v_m\in\words_d^k\setminus\{\unit\}:\\\lvert v_1\rvert+\cdots+\lvert v_m\rvert=k}}\,\prod_{i=1}^m\lvert v_i,u\rvert^{1/k}\\
   \leq N_{m,k}^{1/k}\sup_{\substack{v_1,\ldots,v_m\in\words_d^k\setminus\{\unit\}:\\\lvert v_1\rvert+\cdots+\lvert v_m\rvert=k}}\,\prod_{i=1}^m\big(\sup_{v\in\words_d^k\setminus\{\unit\}}\lvert\langle v,u\rangle\rvert^{1/\lvert v\rvert}\big)^{\lvert v_i\rvert/k}=N_{m,k}^{1/k}\sup_{v\in\words_d^k\setminus\{\unit\}}\lvert\langle v,u\rangle\rvert^{1/\lvert v\rvert}.
  \end{multline*}
  Thus,
  \begin{align*}
   \lVert\unit+u\rVert_{\geogroup{n}{d}}&=\sum_{k=1}^n\lVert\log_{\itensor_n}(\unit+u)\rVert_k^{1/k}=\sum_{k=1}^n\bigg\lVert\sum_{m=1}^k\frac{(-1)^{m-1}u^{\itensor_n m}}{m}\bigg\rVert_k^{1/k}\\
   &\leq\sum_{k=1}^n\Big(\sum_{m=1}^k\tfrac{1}{m}\lVert u^{\itensor_n m}\rVert_k\Big)^{1/k}\stackrel{(\diamondsuit)}{\leq}\sum_{k=1}^n\sum_{m=1}^k\tfrac{1}{\sqrt[k]{m}}\lVert u^{\itensor_n m}\rVert_k^{1/k}\\
   &\leq\underbrace{\sum_{k=1}^n\sum_{m=1}^k\sqrt[k]{\tfrac{N_{m,k}}{m}}}_{:=C'_n}\sup_{v\in\words_d^k\setminus\{\unit\}}\lvert\langle v,u\rangle\rvert^{1/\lvert v\rvert}=C'_n\sup_{v\in\words_d^k\setminus\{\unit\}}\lvert\langle v,\unit+u\rangle\rvert^{1/\lvert v\rvert},
  \end{align*}
  where at $(\diamondsuit)$, we iteratively used the simple fact that $(a+b)^{1/k}\leq a^{1/k}+b^{1/k}$ for all $a,b>0$, $k\in\naturals$.

 \end{enumerate}
\end{proof}

The following theorem shows that the four definitions are indeed equivalent.
\begin{thm}\label{thm:geo_equivalence}
 Let $\gamma\in(0,1)$ and $\level$ denote the integer part of $\tfrac{1}{\gamma}$. The maps
 \begin{align*}
  &I_1^\gamma:\,\geopathsfull^\gamma(\timeT,\reals^d)\rightarrow\geopaths^\gamma(\timeT,\reals^d),\,\fullrX\mapsto\rX,\,\rX_t:=\quotientmap^\level\fullrX_t,\\
  &I_2^\gamma:\,\geopaths^\gamma(\timeT,\reals^d)\rightarrow\geopathsdual^\gamma(\timeT,\reals^d),\,\rX\mapsto\dualrX,\,\dualrX_{st}(w):=\langle w,\rX_{st}\rangle,\\
  &I_3^\gamma:\,\geopathsfull^\gamma(\timeT,\reals^d)\rightarrow\geopathsdualfull^\gamma(\timeT,\reals^d),\,\fullrX\mapsto\dualfullrX,\,\dualfullrX_{st}(w):=\langle w,\fullrX_{st}\rangle
 \end{align*}
 are well-defined and bijective.
\end{thm}
\begin{proof}~
 \begin{enumerate}
  \item\emph{$I_1^\gamma:\,\geopathsfull^\gamma(\timeT,\reals^d)\rightarrow\geopaths^\gamma(\timeT,\reals^d)$ is well-defined.} For each $\fullrX\in\geopathsfull^\gamma(\timeT,\reals^d)$, we have that $\rX:=\quotientmap^\level\fullrX\in\geopaths^\gamma(\timeT,\reals^d)$ since $\quotientmap^\level\geogroupfs{d}=\geogroup{\level}{d}$ (see Corollary \ref{cor:primitive_n_grouplike_n}) and
  \begin{equation*}
   \quotientmap^\level(\fullrX_{st})=\quotientmap^\level(\antitensor\fullrX_s\itensor\fullrX_t)= \quotientmap^\level\antitensor\fullrX_s\itensor_\level\quotientmap^\level\fullrX_t= \antitensor\quotientmap^\level\fullrX_s\itensor_\level\quotientmap^\level\fullrX_t=\antitensor\rX_s\itensor_\level\rX_t,
  \end{equation*}
  where the latter together with the analytic condition of Definition \ref{defn:geometric_fullrX} implies the analytic condition of Definition \ref{defn:geometric_rX}.
  \item\emph{$I_1^\gamma$ is bijective.} Direct consequence of Theorems \ref{thm:lyons_extension} and \ref{thm:lyons_extension_geogroup}. This is due to the fact that the analytic condition of Definition \ref{defn:geometric_rX} implies \eqref{eq:lyons_extension_analytic_condition} and that \eqref{eq:lyons_extension_analytic_result} implies the analytic condition of Definition \ref{defn:geometric_fullrX}. Indeed, by Lemma \ref{lemma:geo_norms}, we have
  \begin{equation*}
   \lVert\rX_{st}\rVert_m=\sqrt{\sum_{w\in\words_{d,m}}\langle w,\rX_{st}\rangle^2}\leq\sqrt{\lvert\words_{d,m}\rvert}\,C_m\lVert\rX_{st}\rVert_{\geogroup{n}{d}}^m
  \end{equation*}
  and
  \begin{equation*}
   \lVert\fullrX_{st}\rVert_{\geogroup{n}{d}}^m\leq (C'_m)^m\sup_{w\in\words_d^n\setminus\{\unit\}}\lvert\langle w,\fullrX_{st}\rangle\rvert\leq(C'_m)^m\lVert \fullrX_{st}\rVert_m.
  \end{equation*}
  \item\emph{$I_2^\gamma:\,\geopaths^\gamma(\timeT,\reals^d)\rightarrow\geopathsdual^\gamma(\timeT,\reals^d)$ is well-defined.} Let $\rX\in\geopaths^\gamma(\timeT,\reals^d)$ be arbitrary. Then, $\dualrX_{st}:=\langle\cdot,\rX_{st}\rangle\in\adual{\lingen{\words_d^\level}}$ for all $s,t\in[0,T]$. We have $\dualrX_{st}(\unit)=\langle\unit,\rX_{st}\rangle=1$ since $\rX_{st}\in\geogroup{\level}{d}\subseteq\groupfs{\lingen{\words_d^\level}}$. For all $w_1, w_2\in\words_d^\level$ with $\lvert w_1\rvert+\lvert w_2\rvert\leq\level$ we get, using the duality of $\shuffle$ and $\deshuffle$,
  \begin{align*}
   \dualrX_{st}(w_1\shuffle w_2)&=\langle w_1\shuffle w_2,\rX_{st}\rangle=\langle w_1\etensor w_2,\deshuffle\rX_{st}\rangle\diamondequal\langle w_1\etensor w_2,\quotientmaptwo^\level(\rX_{st}\etensor\rX_{st})\rangle\\
   &=\langle w_1\etensor w_2,\rX_{st}\etensor\rX_{st}\rangle=\langle w_1,\rX_{st}\rangle\langle w_2,\rX_{st}\rangle=\dualrX_{st}(w_1)\dualrX_{st}(w_2),
  \end{align*}
  where at $(\diamondsuit)$ we used the fact that $\langle w_1\etensor w_2,w_3\etensor w_4\rangle=0$ for all $w_3, w_4\in\words_d$ such that $\lvert w_3\rvert+\lvert w_4\rvert>\level$. Furthermore, using the duality of $\itensor$ and $\deconc$, we have for all $w\in\words_d^\level$ that
  \begin{align*}
   \dualrX_{st}(w)&=\langle w,\rX_{st}\rangle=\langle w,\rX_{su}\itensor_\level\rX_{ut}\rangle\diamondequal\langle w,\quotientmap^\level(\rX_{su}\itensor\rX_{ut})\rangle=\langle w,\rX_{su}\itensor\rX_{ut}\rangle\\
   &=\langle\deconc w,\rX_{su}\etensor\rX_{ut}\rangle=(\dualrX_{su}\etensor\dualrX_{ut})\deconc w,
  \end{align*}
  where at $(\diamondsuit)$ we used the fact that $\langle w,v\rangle=0$ for all $v\in\words_d$ such that $\lvert v\rvert>\level$. Finally, the fact that the analytic condition of Definition \ref{defn:geometric_dualrX} follows from that of Definition \ref{defn:geometric_rX} is a direct consequence of Lemma \ref{lemma:geo_norms}. Thus, $\dualrX\in\geopathsdual^\gamma([0,T],\reals^d)$.
  \item\emph{$I_2^\gamma$ is bijective.} First of all, if $\fullrX^1,\fullrX^2\in\geopathsfull^\gamma(\timeT,\reals^d)$ with $\fullrX^1\neq\fullrX^2$, then there is $t\in[0,T]$ such that $\fullrX^1_t\neq\fullrX^2_t$ and thus, there is $w\in\words_d$ such that $\langle w,\fullrX_{0t}^1\rangle\neq\langle w,\fullrX^2_{0t}\rangle$. Thus, $I_2^\gamma$ is injective.
  
  \noindent Let now $\dualrX\in\geopaths^\gamma([0,T],\reals^d)$ be arbitrary. Put $\rX_{st}:=\sum_{w\in\words_d^\level}\dualrX(w)w$. Then, we obviously have $\dualrX_{st}=\langle\,\cdot\,,\rX_{st}\rangle$ for all $s,t\in[0,T]$. For all $w_1,w_2\in\words_d^\level$ such that $\lvert w_1\rvert+\lvert w_2\rvert\leq\level$, we get, again by the duality of $\shuffle$ and $\deshuffle$,
  \begin{align*}
   \langle w_1\etensor w_2,\deshuffle\rX_{st}\rangle&=\langle w_1\shuffle w_2,\rX_{st}\rangle=\dualrX_{st}(w_1\shuffle w_2)=\dualrX_{st}(w_1)\dualrX_{st}(w_2)\\
   &=\langle w_1,\rX_{st}\rangle\langle w_2,\rX_{st}\rangle=\langle w_1\etensor w_2,\rX_{st}\etensor\rX_{st}\rangle\\
   &=\langle w_1\etensor w_2,\quotientmaptwo^\level(\rX_{st}\etensor\rX_{st})\rangle.
  \end{align*}
  Since $\dualbracket$ on $\lingen{\words_d}\etensor\lingen{\words_d}$ restricted to $[\lingen{\words_d}\etensor\lingen{\words_d}]_\level$ is again an inner product, we conclude $\deshuffle\rX_{st}=\quotientmaptwo^\level(\rX_{st}\etensor\rX_{st}$ and together with $\langle\unit,\rX_{st}\rangle=\dualrX_{st}(\unit)=1$ therefore $\rX_{st}\in\geogroup{\level}{d}$ for all $s,t\in[0,T]$. Furthermore, for all $w\in\words_d^\level$ and all $s,u,t\in[0,T]$, we have
  \begin{align*}
   \langle w,\rX_{st}\rangle&=\dualrX_{st}(w)=(\dualrX_{su}\etensor\dualrX_{ut})\deconc w=\langle\deconc w,\rX_{su}\etensor\rX_{ut}\rangle=\langle w,\rX_{su}\itensor\rX_{ut}\rangle\\
   &=\langle w,\quotientmap^\level(\rX_{su}\itensor\rX_{ut})\rangle=\langle w,\rX_{su}\itensor_\level\rX_{ut}\rangle.
  \end{align*}
  Since $\dualbracket$ on $\lingen{\words_d}$ restricted to $\lingen{\words_d^\level}$ is again an inner product, we conclude $\rX_{st}=\rX_{su}\itensor_\level\rX_{ut}$. Thus, putting $\rX_t:=\rX_{0t}$ for all $t\in[0,T]$ and using $\dualrX_{00}=\counit_\level$, we get
  \begin{equation*}
   \unit=\rX_{00}=\rX_{0t}\itensor_\level\rX_{t0}=\rX_t\itensor_\level\rX_{t0}
  \end{equation*}
  and therefore $\rX_{st}=\rX_{s0}\itensor_\level\rX_{0t}=\rX_s^{\itensor_\level -1}\itensor_\level\rX_t$. Hence, using Lemma \ref{lemma:geo_norms} and the analytic condition of Definition \ref{defn:geometric_dualrX}, we get the analytic condition of Definition \ref{defn:geometric_rX} and conclude that indeed $\rX\in\geopaths^\gamma(\timeT,\reals^d)$. This means we have shown that $I_2^\gamma$ is also surjective.
  \item\emph{$I_l^\gamma:\,\geopathsfull^\gamma(\timeT,\reals^d)\rightarrow\geopathsdualfull^\gamma(\timeT,\reals^d)$ is well-defined.} Let $\rX\in\geopathsfull^\gamma(\timeT,\reals^d)$ be arbitrary. Then, $\dualfullrX_{st}:=\langle\cdot,\fullrX_{st}\rangle\in\adual{\lingen{\words_d}}$ for all $s,t\in[0,T]$. We have $\dualfullrX_{st}(\unit)=\langle\unit,\fullrX_{st}\rangle=1$ since $\fullrX_{st}\in\geogroupfs{d}\subseteq\groupfs{\lingen{\words_d}}$. For all $m\in\naturals$ and $w_1, w_2\in\words_d^m$ with $\lvert w_1\rvert+\lvert w_2\rvert\leq m$ we get, using the duality of $\shuffle$ and $\deshuffle$,
  \begin{align*}
   \dualfullrX_{st}(w_1\shuffle w_2)&=\langle w_1\shuffle w_2,\fullrX_{st}\rangle=\langle w_1\shuffle w_2,\quotientmap^m\fullrX_{st}\rangle=\langle w_1\etensor w_2,\deshuffle\quotientmap^m\fullrX_{st}\rangle\\
   &=\langle w_1\etensor w_2,\quotientmaptwo^m\deshuffle\fullrX_{st}\rangle=\langle w_1\etensor w_2,\quotientmaptwo^m(\fullrX_{st}\etensor\fullrX_{st})\rangle\\
   &=\langle w_1\etensor w_2,\fullrX_{st}\etensor\fullrX_{st}\rangle=\langle w_1,\fullrX_{st}\rangle\langle w_2,\fullrX_{st}\rangle=\dualfullrX_{st}(w_1)\dualfullrX_{st}(w_2).
  \end{align*}
Furthermore, using the duality of $\itensor$ and $\deconc$, we have for all $m\in\naturals$ and $w\in\words_d^m$ that
  \begin{align*}
   \dualfullrX_{st}(w)&=\langle w,\fullrX_{st}\rangle=\langle w,\fullrX_{su}\itensor\fullrX_{ut}\rangle=\langle w,\quotientmap^m(\fullrX_{su}\itensor\fullrX_{ut})\rangle=\langle w,\quotientmap^m\fullrX_{su}\itensor_m\quotientmap^m\fullrX_{ut}\rangle\\
   &=\langle w,\quotientmap^m(\quotientmap^m\fullrX_{su}\itensor\quotientmap^m\fullrX_{ut})\rangle=\langle w,\quotientmap^m\fullrX_{su}\itensor\quotientmap^m\fullrX_{ut}\rangle\\
   &=\langle\deconc w,\quotientmap^m\fullrX_{su}\etensor\quotientmap^m\fullrX_{ut}\rangle=\langle\deconc w,\fullrX_{su}\etensor\fullrX_{ut}\rangle=(\dualfullrX_{su}\etensor\dualfullrX_{ut})\deconc w.
  \end{align*}
Finally, the fact that the analytic condition of Definition \ref{defn:geometric_dualfullrX} follows from that of Definition \ref{defn:geometric_fullrX} is a direct consequence of Lemma \ref{lemma:geo_norms}. Thus, $\dualrX\in\geopathsdualfull^\gamma([0,T],\reals^d)$.
  \item\emph{$I_2^\gamma$ is bijective.} First of all, if $\rX^1,\rX^2\in\geopaths^\gamma(\timeT,\reals^d)$ with $\rX^1\neq\rX^2$, then there is $t\in[0,T]$ such that $\rX^1_t\neq\rX^2_t$ and thus, there is $w\in\words_d^\level$ such that $\langle w,\rX_{0t}^1\rangle\neq\langle w,\rX^2_{0t}\rangle$. Thus, $I_2^\gamma$ is injective.
  
  \noindent Let now $\dualfullrX\in\geopathsdualfull^\gamma([0,T],\reals^d)$ be arbitrary. Put $\fullrX_{st}:=\sum_{w\in\words_d}\dualfullrX(w)w$. Then, we obviously have $\dualfullrX_{st}=\langle\,\cdot\,,\fullrX_{st}\rangle$ for all $s,t\in[0,T]$. For all $m\in\naturals$ and $w_1,w_2\in\words_d^m$ such that $\lvert w_1\rvert+\lvert w_2\rvert\leq m$, we get, again by the duality of $\shuffle$ and $\deshuffle$,
  \begin{align*}
   \langle w_1\etensor w_2,\deshuffle\fullrX_{st}\rangle&=\langle w_1\etensor w_2,\quotientmaptwo^m\deshuffle\fullrX_{st}\rangle=\langle w_1\etensor w_2,\deshuffle\quotientmap^m\fullrX_{st}\rangle=\langle w_1\shuffle w_2,\quotientmap^m\fullrX_{st}\rangle\\
   &=\langle w_1\shuffle w_2,\fullrX_{st}\rangle=\dualfullrX_{st}(w_1\shuffle w_2)=\dualfullrX_{st}(w_1)\dualfullrX_{st}(w_2)\\
   &=\langle w_1,\fullrX_{st}\rangle\langle w_2,\fullrX_{st}\rangle=\langle w_1\etensor w_2,\fullrX_{st}\etensor\fullrX_{st}\rangle
  \end{align*}
  Since $(\lingen{\words_d}\etensor\lingen{\words_d},\spacefs{\lingen{\words_d}\etensor\lingen{\words_d}})$ is a pair of dual vector spaces under $\dualbracket$, we conclude $\deshuffle\fullrX_{st}=\fullrX_{st}\etensor\fullrX_{st}$ and together with $\langle\unit,\fullrX_{st}\rangle=\dualfullrX_{st}(\unit)=1$ therefore $\fullrX_{st}\in\geogroupfs{d}$ for all $s,t\in[0,T]$. Furthermore, for all $m\in\naturals$, $w\in\words_d^m$ and all $s,u,t\in[0,T]$, we have
  \begin{align*}
   \langle w,\fullrX_{st}\rangle&=\dualfullrX_{st}(w)=(\dualfullrX_{su}\etensor\dualfullrX_{ut})\deconc w=\langle\deconc w,\fullrX_{su}\etensor\fullrX_{ut}\rangle=\langle\deconc w,\quotientmap^m\fullrX_{su}\etensor\quotientmap^m\fullrX_{ut}\rangle\\
   &=\langle w,\quotientmap^m\fullrX_{su}\itensor\quotientmap^m\fullrX_{ut}\rangle=\langle w,\quotientmap^m(\quotientmap^m\fullrX_{su}\itensor\quotientmap^m\fullrX_{ut})\rangle=\langle w,\quotientmap^m\fullrX_{su}\itensor_m\quotientmap^m\fullrX_{ut}\rangle\\
   &=\langle w,\quotientmap^m(\fullrX_{su}\itensor\fullrX_{ut})\rangle=\langle w,\fullrX_{su}\itensor_\level\fullrX_{ut}\rangle.
  \end{align*}
  Since $(\lingen{\words_d},\spacefs{\lingen{\words_d}})$ is a pair of dual vector spaces under $\dualbracket$, we conclude $\fullrX_{st}=\fullrX_{su}\itensor\fullrX_{ut}$. Thus, putting $\fullrX_t:=\rX_{0t}$ for all $t\in[0,T]$ and using $\dualfullrX_{00}=\counit$, we get
  \begin{equation*}
   \unit=\fullrX_{00}=\fullrX_{0t}\itensor\fullrX_{t0}=\fullrX_t\itensor\fullrX_{t0}
  \end{equation*}
  and therefore $\fullrX_{st}=\fullrX_{s0}\itensor\fullrX_{0t}=\fullrX_s^{\itensor -1}\itensor\fullrX_t$. Hence, using Lemma \ref{lemma:geo_norms} and the analytic condition of Definition \ref{defn:geometric_dualfullrX}, we get the analytic condition of Definition \ref{defn:geometric_fullrX} and conclude that indeed $\fullrX\in\geopathsfull^\gamma(\timeT,\reals^d)$. This means we have shown that $I_2^\gamma$ is also surjective.
 \end{enumerate}

\end{proof}

\begin{thm}\label{thm:geom_roughpaths_reg_str}\textnormal{(Section 4.4 \cite{hairer14})}\newline
 Consider maps $\Pi:\,\timeT\rightarrow\linear\big(\lingen{\words_d},\mathrm{C}(\timeT,\reals^d)\big)$ and $\Gamma:\,\timeT^2\rightarrow\linear(\lingen{\words_d},\lingen{\words_d})$. They satisfy the conditions
 \begin{enumerate}
  \item $(\Pi_s \unit)(t)=1$ and $\Gamma_{st}(x_1\shuffle x_2)=\Gamma_{st}x_1\shuffle\Gamma_{st}x_2$ for all $s,t\in\timeT,\,x_1,x_2\in\lingen{\words_d}$,
  \item $\Pi_s=\Pi_u\Gamma_{us}$ and $\Gamma_{su}\Gamma_{ut}=\Gamma_{st}$ and $\deconc\Gamma_{st}=(\Gamma_{st}\etensor\id)\deconc$ for all $s,u,t\in\timeT$,
  \item $\lvert(\Pi_s w)(t)\rvert\lesssim\lvert t-s\rvert^{\gamma\lvert w\rvert}$ and $\lVert\Gamma_{st}w\rVert_m\lesssim\lvert t-s\rvert^{\gamma (\lvert w\rvert-m)}$ for all words $w$ and $m<\lvert w\rvert$, uniformly over all $s,t\in[0,T]$
 \end{enumerate}
 if and only if they are given by
 \begin{equation*}
  (\Pi_s x)(t):=\dualfullrX_{st}(x),\quad\Gamma_{st}x:=(\dualfullrX_{ts}\etensor\id)\deconc x=\sum_{(x)}^{\itensor}\dualfullrX_{ts}(x_1)x_2
 \end{equation*}
 for some geometric rough path $\dualfullrX\in\geopathsdualfull^\gamma([0,T],\reals^d)$.
 In this case, we furthermore have $\Gamma_{st}w-w\in\lingen{\words_d^{n-1}}$ for all words $w$ with $\lvert w\rvert=n$ and $\Pi_t(x_1\shuffle x_2)=\Pi_t x_1\Pi_t x_2$ for all $t\in\timeT$ and $x_1, x_2\in\lingen{\words_d}$.
\end{thm}
\begin{proof}~
 \begin{enumerate}[a)]
  \item ${\implies}$: Let $\dualfullrX_{st}:=\counit\Gamma_{ts}$, where $\counit$ is the counit of $(\lingen{\words_d},\deconc)$. We have
  \begin{equation*}
   (\Pi_t\,\cdot\,)(t)=\counit(\cdot)
  \end{equation*}
   for all $t\in\timeT$ since $(\Pi_t \unit)(t)=1$ and $(\Pi_t w)(t)=0$ for all words $w$ such that $\lvert w\rvert>0$. The latter is due to 3.\@ and the continuity of $\Pi_t w$. Hence, using 2.\@ we get
  \begin{equation*}
   \dualfullrX_{st}(x)=\counit\Gamma_{ts}x=(\Pi_t\Gamma_{ts}x)(t)=(\Pi_s x)(t).
  \end{equation*}
  Due to the counit property \eqref{eq:counit_property} and 2.\@ it holds that
  \begin{equation*}
   \Gamma_{st}=(\counit\etensor\id)\deconc\Gamma_{st}=(\counit\etensor\id)(\Gamma_{st}\etensor\id)\deconc=(\counit\Gamma_{st}\etensor\id)\deconc=(\dualfullrX_{ts}\etensor\id)\deconc.
  \end{equation*}
  1.\@ and the fact that the counit is an algebra homomorphism imply
  \begin{equation*}
   \dualfullrX_{st}(w_1\shuffle w_2)=\counit\Gamma_{ts}(w_1\shuffle w_2)=\counit(\Gamma_{ts}w_1\shuffle\Gamma_{ts}w_2)=\counit\Gamma_{ts}w_1\shuffle\counit\Gamma_{ts}w_2=\dualfullrX_{st}(w_1)\shuffle\dualfullrX_{st}(w_2),
  \end{equation*}
  while from 2.\@ we conclude
  \begin{equation*}
   \dualfullrX_{st}=\counit\Gamma_{ts}=\counit\Gamma_{tu}\Gamma_{us}=\dualfullrX_{tu}\Gamma_{us}=\dualfullrX_{ut}(\dualfullrX_{su}\etensor\id)\deconc=(\dualfullrX_{su}\etensor\dualfullrX_{ut})\deconc.
  \end{equation*}
  Finally, using the fact that $\lVert\cdot\rVert_0=\lvert\counit(\cdot)\rvert$, 3.\@ leads to
  \begin{equation*}
   \lvert\dualfullrX_{st}(w)\rvert=\lvert\counit\Gamma_{ts}w\rvert=\lVert\Gamma_{ts}w\rVert_0\lesssim\lvert t-s\rvert^{\gamma\lvert w\rvert}.
  \end{equation*}
  Hence, $\dualfullrX\in\geopathsdualfull^\gamma([0,T],\reals^d)$ is a geometric rough path according to Definition \ref{defn:geometric_dualfullrX}.
  \item ${\impliedby}$: First of all, continuity of $\Pi_s x$ for all $s\in[0,T]$ and all $x\in\words_d$ follows from the fact that for all $u\in[0,T]$, we have
  \begin{align*}
   \lim_{t\rightarrow u}(\Pi_s x)(t)&=\lim_{t\rightarrow u}\dualfullrX_{st}(x)=\sum_{(x)}^{\itensor}\dualfullrX_{su}(x_1)\lim_{t\rightarrow u}\dualfullrX_{ut}(x_2)=\sum_{(x)}^{\itensor}\dualfullrX_{su}(x_1)\counit(x_2)\\
   &=\dualfullrX_{su}(x)=(\Pi_s x)(u),
  \end{align*}
  where we used Definition \ref{defn:geometric_dualfullrX} 2.\@ and 3.\@ as well as the counit property \eqref{eq:counit_property}. Due to Definition \ref{defn:geometric_dualfullrX} 1.\@ we have
  \begin{equation*}
   (\Pi_t \unit)(t)=\dualfullrX_{tt}(\unit)=1.
  \end{equation*}
  Again using Definition \ref{defn:geometric_dualfullrX} 1. and the fact that $\deconc$ is an algebra homomorphism, we compute
  \begin{align*}
   \Gamma_{st}(x\shuffle y)&=(\dualfullrX_{ts}\etensor\id)\deconc(x\shuffle y)=\sum_{(x)}^{\itensor}\sum_{(y)}^{\itensor}\dualfullrX_{ts}(x_1\shuffle y_1)\,(x_2\shuffle y_2)\\
   &=\sum_{(x)}^{\itensor}\sum_{(y)}^{\itensor}\dualfullrX_{ts}(x_1)\,\dualfullrX_{ts}(y_1)\,(x_2\shuffle y_2)=\Big(\sum_{(x)}^{\itensor}\dualfullrX_{ts}(x_1)x_2\Big)\shuffle\Big(\sum_{(y)}^{\itensor}\dualfullrX_{ts}(y_1)y_2\Big)\\
   &=\Gamma_{st}x\shuffle\Gamma_{st}y
  \end{align*}
  Likewise,
  \begin{equation*}
   (\Pi_s(x_1\shuffle x_2))(t)=\dualfullrX_{st}(x_1\shuffle x_2)=\dualfullrX_{st}(x_1)\dualfullrX_{st}(x_2)=(\Pi_s x_1)(t)(\Pi_s x_2)(t)=(\Pi_s x_1\Pi_s x_2)(t).
  \end{equation*}
  By Definition \ref{defn:geometric_dualfullrX} 2.\@ we conclude
  \begin{equation*}
   (\Pi_u\Gamma_{us}x)(t)=\dualfullrX_{ut}(\Gamma_{us}x)=\dualfullrX_{ut}(\dualfullrX_{su}\etensor\id)\deconc x=(\dualfullrX_{su}\etensor\dualfullrX_{ut})\deconc x=\dualfullrX_{st}(x)=(\Pi_s x)(t),
  \end{equation*}
  as well as, using coassociativity,
  \begin{align*}
   \Gamma_{su}\Gamma_{ut}&=(\dualfullrX_{us}\etensor\id)\deconc(\dualfullrX_{tu}\etensor\id)\deconc=(\dualfullrX_{us}\etensor\id)(\dualfullrX_{tu}\etensor\deconc)\deconc\\
   &=(\dualfullrX_{tu}\etensor\dualfullrX_{us}\etensor\id)(\id\etensor\deconc)\deconc=(\dualfullrX_{tu}\etensor\dualfullrX_{us}\etensor\id)(\deconc\etensor\id)\deconc\\
   &=(\dualfullrX_{ts}\etensor\id)\deconc=\Gamma_{st}
  \end{align*}
  and similarly
  \begin{align*}
   \deconc\Gamma_{st}&=\deconc(\dualfullrX_{ts}\etensor\id)\deconc=(\dualfullrX_{ts}\etensor\id\etensor\id)(\id\etensor\deconc)\deconc=(\dualfullrX_{ts}\etensor\id\etensor\id)(\deconc\etensor\id)\deconc\\
   &=(\Gamma_{st}\etensor\id)\deconc.
  \end{align*}
  Since for every word $w$, there are words $w_1^i,w_2^i,\lvert w_1^i\rvert=\lvert w_2^i\rvert=i$ such that $\deconc w=\sum_{i=0}^{\lvert w\rvert}w_1^{\lvert w\rvert-i} \etensor w_2^i$, we have
  \begin{align}\label{eq:translation_bounds_geo}
   \lVert\Gamma_{st}w\rVert_m\leq\sum_{i=0}^{\lvert w\rvert}\big\lvert\dualfullrX_{ts}\big(w_1^{\lvert w\rvert-i}\big)\big\rvert\lVert w_2^i\rVert_m=\big\lvert\dualfullrX_{ts}\big(w_1^{\lvert w\rvert-m}\big)\big\rvert\lesssim\lvert t-s\rvert^{\gamma (\lvert w\rvert-m)}
  \end{align}
  due to Definition \ref{defn:geometric_dualfullrX} 3., which also directly implies
  \begin{equation*}
   \lvert(\Pi_s(w))(t)\rvert=\lvert\dualfullrX_{st}(w)\rvert\lesssim\lvert t-s\rvert^{\gamma\lvert w\rvert}
  \end{equation*}
  for all words w.
  Finally,
  \begin{align*}
   \Gamma_{st}w&=\sum_{(w)}^{\itensor}\dualfullrX_{ts}(w_1)w_2=\dualfullrX_{ts}(\unit)w+\dualfullrX_{ts}(w)\unit+\sum_{(w)}^{\itensor}\dualfullrX_{ts}(w')w''\\
   &=w+\dualfullrX_{ts}(w)\unit+\sum_{(w)}^{\itensor}\dualfullrX_{ts}(w')w'',
  \end{align*}
  thus $\Gamma_{st}w-w\in\lingen{\words_d^{n-1}}$ for all words $w$ with $\lvert w\rvert=n$.

 \end{enumerate}

\end{proof}

\section{Branched rough paths}
Put $\forests_d:=\forests_{\{1,\ldots,d\}}$
and
\begin{align*} 
 \forestgroupfs{d}:=\grouplike(\spacefs{\lingen{\forests_d}},\detreepro)&=\{g\in\spacefs{\lingen{\forests_d}}|\detreepro g=g\etensor g\}=\exp_\star(\spacefs{\lingen{\trees_d}}),
\end{align*}
where
\begin{equation*}
 \spacefs{\lingen{\trees_d}}=\primitive(\spacefs{\lingen{\forests_d}},\detreepro)=\{l\in\spacefs{\lingen{\forests_d}}|\detreepro l=\unit\etensor l+l\etensor\unit\},
\end{equation*}
as well as
\begin{align*} 
 \forestgroup{n}{d}:=\grouplike^{n}(\lingen{\forests_d^n},\detreepro)&=\{g\in\lingen{\forests_d^n}|\detreepro g=\quotientmap^n (g\etensor g)\}=\quotientmap^n\forestgroupfs{d}=\exp_{\star_n}(\lingen{\trees_d^n}),
\end{align*}
where
\begin{equation*}
 \lingen{\trees_d^n}=\primitive(\lingen{\forests_d^n},\detreepro)=\{l\in\lingen{\forests_d^n}|\detreepro l=\unit\etensor l+l\etensor\unit\}.
\end{equation*}
Define $\lVert\cdot\rVert_\forestgroup{n}{d}:\,\forestgroup{n}{d}\rightarrow[0,\infty)$ by (Remark~2.15.\@ \cite{hairerkelly14})
\begin{equation*}
 \bigg\lVert\exp_{\star_n}\Big(\sum_{\tau\in\trees_d^n}b_\tau\tau\Big)\bigg\rVert_\forestgroup{n}{d}:=\sum_{\tau\in\trees_I^n}\lvert b_\tau\rvert^{1/{\lvert\tau\rvert}}.
\end{equation*}
Furthermore, put $\lVert x\rVert:=\sqrt{\langle x,x\rangle}$ for all $x\in\forestspace{d}$ and $\lVert x\rVert_n:=\lVert\proj_n x\rVert$ for all $x\in\spacefs{\forestspace{d}}$.
\begin{defn}\label{defn:branched_rX}(Definition 2.13., Remark 2.15.\@ \cite{hairerkelly14})
 Let $\gamma\in(0,1)$. A $d$-dimensional $\gamma$-Hölder \emph{branched rough path}\index{rough path!branched@\textit{branched}}\index{branched rough path} is a map $\rX:\,\timeT\rightarrow\forestgroup{\level}{d}$ with $\rX_0:=\unit$ such that
 \begin{equation*}
  \sup_{s<t}\frac{\lVert\rX_{st}\rVert_\forestgroup{\level}{d}}{\lvert t-s\rvert^{\gamma}}<\infty,
 \end{equation*}
 where $\rX_{st}:=\rX_s^{\star_{\level}{-1}}\star_\level\rX_t=\antistar\rX_s\star_\level\rX_t$ and $\level$ is the integer part of $\tfrac{1}{\gamma}$. The set of all such maps $\rX:\,\timeT\rightarrow\forestgroup{\level}{d}$ is denoted by $\branchedpaths^\gamma([0,T],\reals^d)$.
\end{defn}

Again, we present three more equivalent definitions.
\begin{subdefn}\label{defn:branched_fullrX}(Based on Definition 2.13., Remark 2.15.\@ \cite{hairerkelly14})
 Let $\gamma\in(0,1)$. A $d$-dimensional $\gamma$-Hölder \emph{branched rough path}\index{rough path!branched@\textit{branched}}\index{branched rough path} is a map $\fullrX:\,\timeT\rightarrow\forestgroupfs{d}$ with $\fullrX_0:=\unit$ such that
 \begin{equation*}
  \sup_{s<t}\frac{\lVert\quotientmap^n\fullrX_{st}\rVert_\forestgroup{n}{d}}{\lvert t-s\rvert^{\gamma}}<\infty\quad\forallc n\in\naturals,
 \end{equation*}
 where $\fullrX_{st}:=\fullrX_s^{\star{-1}}\star\fullrX_t=\antistar\fullrX_s\star\fullrX_t$. The set of all such maps $\fullrX:\,\timeT\rightarrow\forestgroupfs{d}$ is denoted by $\branchedpathsfull^\gamma([0,T],\reals^d)$.
\end{subdefn}
\begin{subdefn}\label{defn:branched_dualrX}(Definition 1.6.\@ \cite{hairerkelly14})
 Let $\gamma\in(0,1)$. A $d$-dimensional $\gamma$-Hölder \emph{branched rough path} is a map $\dualrX:\,\timeT^2\rightarrow\adual{\lingen{\forests_d^\level}}$ such that
 \begin{enumerate}
  \item $\dualrX_{st}(\unit)=1$ and $\dualrX_{st}(\zeta_1\zeta_2)=\dualrX_{st}(\zeta_1)\dualrX_{st}(\zeta_2)$ for all forests $\zeta_1,\zeta_2\in\forests_d^\level$ such that $\lvert\zeta_1\rvert+\lvert\zeta_2\rvert\leq\level$,
  \item $\dualrX_{tt}=\counit_\level$ and $\dualrX_{st}(\zeta)=(\dualrX_{su}\etensor\dualrX_{ut})\destar\zeta=\sum_{(\zeta)}^{\star}\dualrX_{su}(\zeta^1)\dualrX_{ut}(\zeta^2)$ for all forests $\zeta\in\forests_d^\level$,
  \item $\sup_{s\neq t}\frac{\lvert\dualrX_{st}(\zeta)\rvert}{\lvert t-s\rvert^{\gamma\lvert\zeta\rvert}}<\infty$ for all forests $\zeta\in\forests_d^\level$,
 \end{enumerate}
 where $\level$ is the integer part of $\tfrac{1}{\gamma}$. Put $\dualrX_t:=\dualrX_{0t}$. The set of all such maps $\dualrX:\,\timeT^2\rightarrow\adual{\lingen{\forests_d^\level}}$ is denoted by $\branchedpathsdual^\gamma([0,T],\reals^d)$.
\end{subdefn}
\begin{subdefn}\label{defn:branched_dualfullrX}(Based on Definition 1.6.\@ \cite{hairerkelly14})
 Let $\gamma\in(0,1)$. A $d$-dimensional $\gamma$-Hölder \emph{branched rough path} is a map $\dualfullrX:\,\timeT^2\rightarrow\adual{\lingen{\forests_d}}$ such that
 \begin{enumerate}
  \item $\dualfullrX_{st}(\unit)=1$ and $\dualfullrX_{st}(\zeta_1\zeta_2)=\dualfullrX_{st}(\zeta_1)\dualfullrX_{st}(\zeta_2)$ for all forests $\zeta_1,\zeta_2\in\forests_d$,
  \item $\dualfullrX_{tt}=\counit$ and $\dualfullrX_{st}(\zeta)=(\dualfullrX_{su}\etensor\dualfullrX_{ut})\destar\zeta=\sum_{(\zeta)}^{\star}\dualfullrX_{su}(\zeta^1)\dualfullrX_{ut}(\zeta^2)$ for all forests $\zeta\in\forests_d$,
  \item $\sup_{s\neq t}\frac{\lvert\dualfullrX_{st}(\zeta)\rvert}{\lvert t-s\rvert^{\gamma\lvert\zeta\rvert}}<\infty$ for all forests $\zeta\in\forests_d$.
 \end{enumerate}
 Put $\dualfullrX_t:=\dualfullrX_{0t}$. The set of all such maps $\dualfullrX:\,\timeT^2\rightarrow\adual{\lingen{\forests_d}}$ is denoted by $\branchedpathsdualfull^\gamma([0,T],\reals^d)$.
\end{subdefn}
\begin{lemma}\label{lemma:branched_norms}\textnormal{(Based on Remark 2.15. \cite{hairerkelly14})}
 \begin{enumerate}
  \item For all $k\in\naturals$, there is a constant $C_k>0$ such that
  \begin{equation*}
   \lvert\langle\zeta,g\rangle\rvert\leq C_{\lvert w\rvert}\lVert g\rVert_{\forestgroup{n}{d}}^{\lvert\zeta\rvert}\quad\forallc n\in\naturals,\,g\in\forestgroup{n}{d},\,\zeta\in\forests_d^n\setminus\{\unit\}.
  \end{equation*}
  \item For all $n\in\naturals$, there is a constant $C_n'>0$ such that
  \begin{equation*}
   \lVert g\rVert_{\forestgroup{n}{d}}\leq C_n'\sup_{\zeta\in\forests_d^n\setminus\{\unit\}}\lvert\langle\zeta,g\rangle\rvert^{1/\lvert\zeta\rvert}\quad\forallc g\in\forestgroup{n}{d}.
  \end{equation*}
 \end{enumerate}
\end{lemma}
\begin{proof}(Based on the proof of Proposition 4 \cite{lyonsvictoir07})
 \begin{enumerate}
  \item Let $n\in\naturals_0$ and $x=\sum_{\tau\in\trees_d^n}b_\tau\tau\in\lingen{\trees_d^n}$ be arbitrary. For every $m,k\in\naturals$ with $m\leq n$, we have
  \begin{align*}
   \lVert x^{\star_n m}\rVert_k&=\bigg\lVert\Big(\sum_{\tau\in\trees_d^n\setminus\{\unit\}}b_\tau\tau\Big)^m\bigg\rVert_k\\
   &=\bigg\lVert\sum_{\substack{\tau_1,\ldots,\tau_m\in\trees_d^k\setminus\{\unit\}:\\\lvert\tau_1\rvert+\ldots+\lvert\tau_m\rvert=k}}b_{\tau_1}\cdots b_{\tau_m}\,\tau_1\star_n\cdots\star_n\tau_m\bigg\rVert_k\\
   &\leq\underbrace{\sup_{\substack{\tau_1,\ldots,\tau_m\in\trees_d^k\setminus\{\unit\}:\\\lvert\tau_1\rvert+\cdots+\lvert\tau_m\rvert=k}}\lVert\tau_1\star_n\cdots\star_n\tau_m\rVert_k}_{:=D_{m,k}}\sum_{\substack{\tau_1,\ldots,\tau_m\in\trees_d^k\setminus\{\unit\}:\\\lvert\tau_1\rvert+\cdots+\lvert\tau_m\rvert=k}}\,\prod_{i=1}^m\lvert b_{\tau_i}\rvert\\
   &\leq D_{m,k} N_{m,k}\sup_{\substack{\tau_1,\ldots,\tau_m\in\trees_d^k\setminus\{\unit\}:\\\lvert\tau_1\rvert+\cdots+\lvert\tau_m\rvert=k}}\,\prod_{i=1}^m\lvert b_{\tau_i}\rvert\\
   &\leq D_{m,k} N_{m,k}\sup_{\substack{l_1,\ldots,l_m\in\naturals:\\l_1+\cdots+l_m=k}}\,\prod_{i=1}^m\lVert\exp_{\star_n}(x)\rVert_{\forestgroup{n}{d}}^{l_i}=D_{m,k} N_{m,k}\lVert\exp_{\star_n}(x)\rVert_{\forestgroup{n}{d}}^k,
  \end{align*}
  where $N_{m,k}:=\big\lvert\big\{\tau_1,\ldots,\tau_m\in\trees_d^k\setminus\{\unit\}\big|\lvert\tau_1\rvert+\cdots+\lvert\tau_m\rvert=k\big\}\big\rvert<\infty$. Thus, for any $\zeta\in\forests_d^n$, we get
  \begin{align*}
   \big\lvert\langle\zeta,\exp_{\star_n}(x)\rangle\big\rvert&=\left\lvert\sum_{m=0}^n\frac{1}{m!}\langle\zeta,x^{\star_n m}\rangle\right\rvert=\left\lvert\sum_{m=0}^{\lvert\zeta\rvert}\frac{1}{m!}\langle\zeta,\proj_{\lvert\zeta\rvert}x^{\star_n m}\rangle\right\rvert\\
   &\leq\sum_{m=0}^{\lvert\zeta\rvert}\frac{1}{m!}\left\lVert x^{\star_n m}\right\rVert_{\lvert\zeta\rvert}\leq\underbrace{\sum_{m=0}^{\lvert\zeta\rvert}\frac{D_{m,\lvert\zeta\rvert}N_{m,\lvert\zeta\rvert}}{m!}}_{:=C_{\lvert \zeta\rvert}}\,\lVert\exp_{\star_n}(x)\rVert_{\forestgroup{n}{d}}^{\lvert\zeta\rvert}.
  \end{align*}
  \item Let $n\in\naturals_0$ and $u\in\liefs{\lingen{\words_d^n}}$ be arbitrary. For every $m,k\in\naturals$ with $m\leq n$, we have, with 
  \begin{equation*}
   \tilde{D}_{m,k}:=\sup_{\substack{\zeta_1,\ldots,\zeta_m\in\forests_d^k\setminus\{\unit\}:\\\lvert\zeta_1\rvert+\cdots+\lvert\zeta_m\rvert=k}}\lVert\zeta_1\star_n\cdots\star_n\zeta_m\rVert_k,
  \end{equation*}
  $\tilde{N}_{m,k}:=\big\lvert\big\{\zeta_1,\ldots,\zeta_m\in\forests_d^k\setminus\{\unit\}\big|\lvert\zeta_1\rvert+\cdots+\lvert\zeta_m\rvert=k\big\}\big\rvert<\infty$ and similar first steps as above,
  \begin{multline*}
   \lVert u^{\star_n m}\lVert_k^{1/k}\leq\underbrace{\tilde{D}_{m,k}^{1/k}\tilde{N}_{m,k}^{1/k}}_{:=R_{m,k}}\sup_{\substack{\zeta_1,\ldots,\zeta_m\in\forests_d^k\setminus\{\unit\}:\\\lvert\zeta_1\rvert+\cdots+\lvert\zeta_m\rvert=k}}\,\prod_{i=1}^m\lvert\zeta_i,u\rvert^{1/k}\\
   \leq R_{m,k}^{1/k}\sup_{\substack{\zeta_1,\ldots,\zeta_m\in\forests_d^k\setminus\{\unit\}:\\\lvert\zeta_1\rvert+\cdots+\lvert \zeta_m\rvert=k}}\,\prod_{i=1}^m\big(\sup_{\zeta\in\forests_d^k\setminus\{\unit\}}\lvert\langle\zeta,u\rangle\rvert^{1/\lvert\zeta\rvert}\big)^{\lvert \zeta_i\rvert/k}=R_{m,k}^{1/k}\sup_{\zeta\in\forests_d^k\setminus\{\unit\}}\lvert\langle\zeta,u\rangle\rvert^{1/\lvert\zeta\rvert}.
  \end{multline*}
  Thus,
  \begin{align*}
   \lVert\unit+u\rVert_{\forestgroup{n}{d}}&=\sum_{k=1}^n\lVert\log_{\star_n}(\unit+u)\rVert_k^{1/k}=\sum_{k=1}^n\bigg\lVert\sum_{m=1}^k\frac{(-1)^{m-1}u^{\star_n m}}{m}\bigg\rVert_k^{1/k}\\
   &\leq\sum_{k=1}^n\Big(\sum_{m=1}^k\tfrac{1}{m}\lVert u^{\star_n m}\rVert_k\Big)^{1/k}\stackrel{(\diamondsuit)}{\leq}\sum_{k=1}^n\sum_{m=1}^k\tfrac{1}{\sqrt[k]{m}}\lVert u^{\star_n m}\rVert_k^{1/k}\\
   &\leq\underbrace{\sum_{k=1}^n\sum_{m=1}^k\sqrt[k]{\tfrac{R_{m,k}}{m}}}_{:=C'_n}\sup_{\zeta\in\forests_d^k\setminus\{\unit\}}\lvert\langle\zeta,u\rangle\rvert^{1/\lvert \zeta\rvert}=C'_n\sup_{\zeta\in\forests_d^k\setminus\{\unit\}}\lvert\langle\zeta,\unit+u\rangle\rvert^{1/\lvert\zeta\rvert},
  \end{align*}
  where at $(\diamondsuit)$, we iteratively used the simple fact that $(a+b)^{1/k}\leq a^{1/k}+b^{1/k}$ for all $a,b>0$, $k\in\naturals$.
 \end{enumerate}
\end{proof}
For each $\gamma\in(0,1)$, let $q_\gamma:\,\forestspace{d}\rightarrow\reals$ be the linear map recursively generated by $q_\gamma(\zeta):=1$ for all $\zeta\in\forests_d$ such that $\lvert\zeta\rvert\leq 1/\gamma$
and
\begin{equation*}
 q_\gamma(\zeta):=\frac{1}{2^{\gamma\lvert\zeta\rvert}-2}\sum_{(\zeta)}^{\star}q_\gamma(\zeta')q_\gamma(\zeta''),\quad q_\gamma(\zeta_1\zeta_2):=q_\gamma(\zeta_1)q_\gamma(\zeta_2)
\end{equation*}
for all $\zeta,\zeta_1,\zeta_2\in\forests_d$ such that $\lvert\zeta\rvert>1/\gamma$ and $\lvert\zeta_1\zeta_2\rvert>1/\gamma$ (Equation (34) \cite{gubinelli10}).

\begin{thm}\label{thm:gubinelli_extension}\textnormal{\textbf{Gubinelli's Extension Theorem} (Theorem 7.3.\@ \cite{gubinelli10})}
 For $\gamma\in(0,1)$, let $\dualrX\in\branchedpathsdual^\gamma([0,T],\reals^d)$ and $B\in[0,1], A>0$ such that
 \begin{equation*}
  \sup_{s<t}\frac{\lvert\dualrX_{st}(\tau)\rvert}{\lvert t-s\rvert^{\gamma\lvert\tau\rvert}}\leq B A^{\lvert\tau\rvert}q_\gamma(\tau)\quad\forallc\tau\in\trees_d^\level.
 \end{equation*}
 Then, there is an extension $\dualfullrX\in\branchedpathsdualfull^\gamma([0,T],\reals^d)$ such that $\dualfullrX_{st}{\restriction_{\lingen{\forests_d^\level}}}=\dualrX_{st}$ for all $s,t\in[0,T]$, where $\level$ is the integer part of $\gamma$, and 
 \begin{equation*}
  \sup_{s<t}\frac{\lvert\dualfullrX_{st}(\tau)\rvert}{\lvert t-s\rvert^{\gamma\lvert\tau\rvert}}\leq B A^{\lvert\tau\rvert}q_\gamma(\tau)\quad\forallc\tau\in\trees_d.
 \end{equation*}
\end{thm}
\begin{lemma}\label{lemma:gubinelli_extension_uniqueness}
 If for an $\rX\in\branchedpaths^\gamma([0,T],\reals^d)$ there exists an extension $\fullrX\in\branchedpathsfull^\gamma([0,T],\reals^d)$ such that $\quotientmap^\level\fullrX=\rX$, then this extension is unique.
\end{lemma}
\begin{proof}
 Assume for $\fullrX^1,\,\fullrX^2\in\branchedpathsfull^\gamma([0,T],\reals^d)$, we have $\quotientmap^m\fullrX^1=\rX=\quotientmap^m\fullrX^2$ for some $m\geq\level$.
 
 \noindent Put $A_t:=\log_{\star_{m+1}}(\quotientmap^{m+1}\fullrX_t^1)$, $a_t:=A_t-\quotientmap^m A_t$, $B_t:=\log_{\star_{m+1}}(\quotientmap^{m+1}\fullrX_t^2)$ and $b_t:=B_t-\quotientmap^m B_t$. Then, since $\quotientmap^m A_t\star_{m+1} a_t=0$ due to $a_{st}\in\lingen{\trees_{d,m+1}}$ and $\counit{A_t}=0$, we have $A_t^{\star_{m+1} k}=(\quotientmap^m A_t)^{\star_{m+1} k}$ for $k>1$ and thus
 \begin{equation*}
  \quotientmap^{m+1}\fullrX_t^a=\exp_{\star_{m+1}}(A_t)=\exp_{\star_{m+1}}(\quotientmap^m A_t)+a_t.
 \end{equation*}
 Since both $\quotientmap^m A_s$ and $a_s$ are primitive elements, we have $\antistar\exp_{\star_{m+1}}(\quotientmap^m A_s)=\exp_{\star_{m+1}}(-\quotientmap^m A_s)$ and $\antistar a_s=-a_s$, and thus
 \begin{align*}
  \quotientmap^{m+1}\fullrX_{st}^1&=\quotientmap^{m+1}(\fullrX_t\star\antistar\fullrX_s)=\quotientmap^{m+1}\fullrX_t\star_{m+1}\antistar\quotientmap^{m+1}\fullrX_s\\
  &=(\exp_{\star_{m+1}}(\quotientmap^m A_t)+a_t)\star_{m+1}\antipode(\exp_{\star_{m+1}}(\quotientmap^m A_s)+a_s)\\
  &=(\exp_{\star_{m+1}}(\quotientmap^m A_t)+a_t)\star_{m+1}(\exp_{\star_{m+1}}(-\quotientmap^m A_s)-a_s)\\
  &=\exp_{\star_{m+1}}(\quotientmap^m A_t)\star_{m+1}\exp_{\star_{m+1}}(-\quotientmap^m A_s)+a_t-a_s
 \end{align*}
 Likewise, $\quotientmap^{m+1}\fullrX_{st}^2=\exp_{\star_{m+1}}(\quotientmap^m B_t)\star_{m+1}\exp_{\star_{m+1}}(\quotientmap^m-B_s)+b_t-b_s$. Since
 \begin{align*}
  \quotientmap^m A_t&=\quotientmap^m\log_{\star_{m+1}}(\quotientmap^{m+1}\fullrX_t^1)=\log_{\star_{m}}(\quotientmap^m\fullrX_t^1)=\log_{\star_{m}}(\quotientmap^m\fullrX_t^2)=\quotientmap^m\log_{\star_{m+1}}(\quotientmap^{m+1}\fullrX_t^2)\\
  &=\quotientmap^m B_t,
 \end{align*}
 we have
 \begin{equation*}
  \quotientmap^{m+1}\fullrX_{st}^1-\quotientmap^{m+1}\fullrX_{st}^2=a_t+b_s-a_s-b_t.
 \end{equation*}
 Using Lemma \ref{lemma:branched_norms}, we get from the analytic condition in Definition \ref{defn:branched_fullrX} that
 \begin{equation*}
  \sup_{s<t}\frac{\lvert\langle\zeta,(a_t-b_t)-(a_s-b_s)\rangle\rvert}{\lvert t-s\rvert^{(m+1)\gamma}}<\infty\quad\forallc\zeta\in\forests_{d,m+1},
 \end{equation*}
 thus $a-b$ is constant since $(m+1)\gamma\geq(\level+1)\gamma>1$, and hence $\quotientmap^{m+1}\fullrX_{st}^1=\quotientmap^{m+1}\fullrX_{st}^2$.
 
 \noindent By inductive application of the argumentation above, we get $\fullrX^1=\fullrX^2$.

\end{proof}

\begin{thm}
 Let $\gamma\in(0,1)$ and $\level$ denote the integer part of $\tfrac{1}{\gamma}$. The maps
 \begin{align*}
  &I_{1}^\gamma:\,\branchedpathsfull^\gamma(\timeT,\reals^d)\rightarrow\branchedpaths^\gamma(\timeT,\reals^d),\,\fullrX\mapsto\rX,\,\rX_t:=\quotientmap^\level\fullrX_t,\\
  &I_{2}^\gamma:\,\branchedpaths^\gamma(\timeT,\reals^d)\rightarrow\branchedpathsdual^\gamma(\timeT,\reals^d),\,\rX\mapsto\dualrX,\,\dualrX_{st}(y):=\langle y,\rX_{st}\rangle,\\
  &I_{3}^\gamma:\,\branchedpathsfull^\gamma(\timeT,\reals^d)\rightarrow\branchedpathsdualfull^\gamma(\timeT,\reals^d),\,\fullrX\mapsto\dualfullrX,\,\dualfullrX_{st}(y):=\langle y,\fullrX_{st}\rangle
 \end{align*}
 are well-defined and bijective.
\end{thm}
\begin{proof}
 The proof that the functions $I_{2}^\gamma$ and $I_{3}^\gamma$ are well-defined and bijective is completely analogous to steps 3.\@ to 6.\@ of the proof of Theorem \ref{thm:geo_equivalence}. It remains to look at $I_1^\gamma$.
 \begin{enumerate}
 \item\emph{$I_1^\gamma:\,\branchedpathsfull^\gamma(\timeT,\reals^d)\rightarrow\branchedpaths^\gamma(\timeT,\reals^d)$ is well-defined.} For each $\fullrX\in\branchedpathsfull^\gamma(\timeT,\reals^d)$, we have that $\rX:=\quotientmap^\level\fullrX\in\branchedpaths^\gamma(\timeT,\reals^d)$ since $\quotientmap^\level\forestgroupfs{d}=\forestgroup{\level}{d}$ (see Corollary \ref{cor:primitive_n_grouplike_n}) and
  \begin{equation*}
   \quotientmap^\level(\fullrX_{st})=\quotientmap^\level(\antistar\fullrX_s\star\fullrX_t)= \quotientmap^\level\antistar\fullrX_s\star_\level\quotientmap^\level\fullrX_t= \antistar\quotientmap^\level\fullrX_s\star_\level\quotientmap^\level\fullrX_t=\antistar\rX_s\star_\level\rX_t,
  \end{equation*}
  where the latter together with the analytic condition of Definition \ref{defn:branched_fullrX} implies the analytic condition of Definition \ref{defn:branched_rX}.
 \item\emph{$I_1^\gamma$ is bijective.} Injectivity was already proven in Lemma \ref{lemma:gubinelli_extension_uniqueness}. To show surjectivity, let $\rX\in\branchedpaths^\gamma(\timeT,\reals^d)$ be arbitrary. Then, $\dualrX:=I_2^\gamma(\rX)\in\branchedpathsdual^\gamma(\timeT,\reals^d)$ is given by $\dualrX_{st}(y)=\langle y,\rX_{st}\rangle$ for all $y\in\lingen{\forests_d^\level},\,s,t\in[0,T]$. Since $\trees_d^\level$ is a finite set, for any $B\in(0,1]$ there is some $A>0$ such that
 \begin{equation*}
  \sup_{s<t}\frac{\lvert\dualrX_{st}(\tau)\rvert}{\lvert t-s\rvert^{\gamma\lvert\tau\rvert}}\leq B A^{\lvert\tau\rvert}q_\gamma(\tau)\quad\forallc\tau\in\trees_d^\level.
 \end{equation*}
 Thus, applying Theorem \ref{thm:gubinelli_extension}, we find that there is an extension $\dualfullrX\in\branchedpathsdualfull^\gamma([0,T],\reals^d)$ such that $\dualfullrX_{st}{\restriction_{\lingen{\forests_d^\level}}}=\dualrX_{st}$. By surjectivity of $I_3^\gamma$, there finally is $\fullrX\in\branchedpathsfull^\gamma([0,T],\reals^d)$ such that $\dualfullrX_{st}(y)=\langle y,\fullrX_{st}\rangle$ for all $y\in\forestspace{d},\,s,t\in[0,T]$. Then, again for all $s,t\in[0,T]$,
 \begin{align*}
  I_3^\gamma(\fullrX_{st})&=\quotientmap^\level\fullrX_{st}=\quotientmap^\level\sum_{\zeta\in\forests_d}\langle\zeta,\fullrX_{st}\rangle\zeta=\quotientmap^\level\sum_{\zeta\in\forests_d}\dualfullrX_{st}(\zeta)\zeta=\sum_{\zeta\in\forests_d^\level}\dualfullrX_{st}(\zeta)\zeta=\sum_{\zeta\in\forests_d^\level}\dualrX_{st}(\zeta)\zeta\\
  &=\sum_{\zeta\in\forests_d^\level}\langle\zeta,\rX_{st}\rangle\zeta=\rX_{st}.
 \end{align*}
 Since $\rX\in\branchedpaths^\gamma(\timeT,\reals^d)$ was arbitrary, we have proven that $I_1^\gamma$ is surjective.
 \end{enumerate}
\end{proof}

Putting $\lVert\zeta\rVert:=\sqrt{\langle\zeta,\zeta\rangle}$ and $\lVert\zeta\rVert_m:=\lVert\proj_m\zeta\rVert$, we have the following.
\begin{thm}\label{thm:branchedrough_models}\textnormal{(Based on Section 4.4 \cite{hairer14})}\newline
 Consider maps $\Pi:\,\timeT\rightarrow\linear\big(\forestspace{d},\mathrm{C}(\timeT,\reals^d)\big)$ and $\Gamma:\,\timeT^2\rightarrow\linear(\forestspace{d},\forestspace{d})$. They satisfy the conditions
 \begin{enumerate}
  \item $(\Pi_s(\unit))(t)=1$ and $\Gamma_{st}(x_1\treepro x_2)=\Gamma_{st}x_1\treepro\Gamma_{st}x_2$ for all $s,t\in\timeT,\,x_1,x_2\in\forestspace{d}$,
  \item $\Pi_s=\Pi_u\Gamma_{us}$ and $\Gamma_{su}\Gamma_{ut}=\Gamma_{st}$ and $\destar\Gamma_{st}=(\Gamma_{st}\etensor\id)\destar$ for all $s,u,t\in\timeT$,
  \item $\lvert(\Pi_s(\zeta))(t)\rvert\lesssim\lvert t-s\rvert^{\gamma\lvert\zeta\rvert}$ and $\lVert\Gamma_{st}\zeta\rVert_m\lesssim\lvert t-s\rvert^{\gamma (\lvert\zeta\rvert-m)}$ for all forests $\zeta$ and $m<\lvert\zeta\rvert$, uniformly over all $s,t\in[0,T]$
 \end{enumerate}
 if and only if they are given by
 \begin{equation*}
  (\Pi_s(x))(t):=\dualfullrX_{st}(x),\quad\Gamma_{st}x:=(\dualfullrX_{ts}\etensor\id)\destar x=\sum_{(x)}^{\star}\dualfullrX_{ts}(x_1)x_2
 \end{equation*}
 for some branched rough path $\dualfullrX\in\branchedpathsdualfull^\gamma([0,T],\reals^d)$.
 In this case, we furthermore have $\Gamma_{st}\zeta-\zeta\in\lingen{\forests_d^{n-1}}$ for all forests $\zeta$ with $\lvert\zeta\rvert=n$ and $\Pi_t(x_1\treepro x_2)=\Pi_t x_1\Pi_t x_2$ for all $t\in\timeT$ and $x_1, x_2\in\forestspace{d}$.
\end{thm}
\begin{proof}
 Analogous to the proof of Theorem \ref{thm:geom_roughpaths_reg_str}. Just replace $(\lingen{\words_d},\shuffle,\deconc)$ with $(\forestspace{d},\treepro,\destar)$ everywhere, except that instead of \eqref{eq:translation_bounds_geo} we have now
 \begin{align*}
   \lVert\Gamma_{st}\zeta\rVert_m&\leq\sum_{(C,T)\in\cuts{\zeta}}\cutfactor{\zeta}{C}{T}\,\lvert\dualfullrX_{ts}(C)\rvert\lVert T\rVert_m=\sum_{\substack{(C,T)\in\cuts{\zeta}:\\\lvert C\rvert=\lvert\zeta\rvert-m}}\cutfactor{\zeta}{C}{T}\,\lvert\dualfullrX_{ts}(C)\rvert\\&\lesssim\lvert t-s\rvert^{\gamma (\lvert\zeta\rvert-m)}.
 \end{align*}

\end{proof}

It turns out that weakly geometric rough paths can be seen as a special kind of branched rough paths satisfying an additional condition, which is nothing but the integration by parts rule. This is made precise by the following result, which is due to \cite{hairerkelly14}. We added a proof for the 'conversely' part which was just stated as an observation there.

Let $\phi:\,\forestspace{d}\rightarrow\lingen{\words_d}$ be the Hopf algebra homomorphism and $\hat\phi:\,\lingen{\words_d}\rightarrow\lingen{\trees_d}$ be the coalgebra monomorphism introduced in Theorem \ref{thm:phi_Hopf_morphism}.

\begin{thm}\textnormal{(Section 4.1, Proposition 4.6.\@ \cite{hairerkelly14})}\label{thm:geom_are_branched}
 For every weakly geometric rough path $\dualrX^\mathrm{g}\in\geopaths^\gamma_2(\timeT,\reals^d)$, there is a branched rough path $\dualrX^\mathrm{b}\in\branchedpaths^\gamma_2(\timeT,\reals^d)$ given by
 \begin{equation*}
  \dualrX^\mathrm{b}_{st}(x):=\dualrX^\mathrm{g}_{st}(\phi(x)).
 \end{equation*}
 $\dualrX^\mathrm{g}$ can be recovered from $\dualrX^\mathrm{b}$ by
 \begin{equation*}
  \dualrX^\mathrm{g}_{st}(x)=\dualrX^\mathrm{b}_{st}(\hat\phi(x)).
 \end{equation*}
 Conversely, for every branched rough path $\dualrX^{\mathrm{b}'}\in\branchedpaths^\gamma_2(\timeT,\reals^d)$ with the additional property that
 \begin{equation}\label{eq:branched_geometric_additional_property}
  \dualrX_{st}^{\mathrm{b}'}(x)=0\quad\forallc x\in\lingen{\forests_d^\level}\cap\ker\phi,\,s,t\in\timeT,
 \end{equation}
 there is a weakly geometric rough path $\dualrX^{\mathrm{g}'}\in\geopaths^\gamma_2(\timeT,\reals^d)$ given by
 \begin{equation*}
  \dualrX^{\mathrm{g}'}_{st}(x):=\dualrX^{\mathrm{b}'}_{st}(\hat\phi(x)).
 \end{equation*}
 $\dualrX^{\mathrm{b}'}$ can be recovered from $\dualrX^{\mathrm{g}'}$ by
 \begin{equation*}
  \dualrX^{\mathrm{b}'}_{st}(x)=\dualrX^{\mathrm{g}'}_{st}(\phi(x)).
 \end{equation*}
\end{thm}
\begin{proof} Again, put $\level:=\lfloor\tfrac{1}{\gamma}\rfloor$.
 \begin{enumerate}
  \item Let $\dualrX^\mathrm{g}\in\geopaths^\gamma_2(\timeT,\reals^d)$ and $\dualrX^\mathrm{b}$ defined as above. The algebraic conditions 1.\@ and 2.\@ of Definition \ref{defn:branched_dualrX} follow then directly from the algebraic conditions 1.\@ and 2.\@ of Definition \ref{defn:geometric_dualrX} due to the fact that $\phi$ is a bialgebra homomorphism which preserves the grading. We also have
  \begin{equation*}
   \sup_{s<t}\frac{\lvert\dualrX_{st}^\mathrm{b}(\zeta)\rvert}{\lvert t-s\rvert^{\gamma\lvert\zeta\rvert}}=\sup_{s<t}\frac{\lvert\dualrX_{st}^\mathrm{g}(\phi(\zeta))\rvert}{\lvert t-s\rvert^{\gamma\lvert\zeta\rvert}}<\infty\quad\forallc\zeta\in\lingen{\forests_d^\level}
  \end{equation*}
  due to condition 3.\@ of Definition \ref{defn:geometric_dualrX} since $\phi(\zeta)\in\lingen{\words_{d,\lvert\zeta\rvert}}$, and thus we indeed have $\dualrX^\mathrm{b}\in\branchedpaths^\gamma_2(\timeT,\reals^d)$. The recovery property is clear from the fact that $\phi\hat\phi=\id$.
  \item Let $\dualrX^{\mathrm{b}'}\in\branchedpaths^\gamma_2(\timeT,\reals^d)$ be such that it satisfies the additional property and let $\dualrX^{\mathrm{g}'}$ be defined as above. We have that $\hat\phi(w_1)\treepro\hat\phi(w_2)-\hat\phi(w_1\shuffle w_2)\in\ker\phi$ for all $w_1,w_2\in\words_d$, since 
  \begin{align*}
   \phi(\hat\phi(w_1)\treepro\hat\phi(w_2)-\hat\phi(w_1\shuffle w_2))&=\phi(\hat\phi(w_1)\treepro\hat\phi(w_2))-\phi\hat\phi(w_1\shuffle w_2)\\
   &=\phi\hat\phi(w_1)\shuffle\phi\hat\phi(w_2)-w_1\shuffle w_2=0.
  \end{align*}
  Hence, using property 1.\@ of Definition \ref{defn:branched_dualrX}, we get for all $w_1,w_2\in\words_d$ with $\lvert w_1\rvert+\lvert w_2\rvert\leq\level$ that
  \begin{align*}
   \dualrX^{\mathrm{g}'}(w_1\shuffle w_2)&=\dualrX^{\mathrm{b}'}(\hat\phi(w_1\shuffle w_2))=\dualrX^{\mathrm{b}'}(\hat\phi(w_1)\treepro\hat\phi(w_2))=\dualrX^{\mathrm{b}'}(\hat\phi(w_1))\dualrX^{\mathrm{b}'}(\hat\phi(w_2))\\
   &=\dualrX^{\mathrm{g}'}(w_1)\dualrX^{\mathrm{g}'}(w_2).
  \end{align*}
  Since $\hat\phi$ is a coalgebra homomorphism, condition 2.\@ of Definition \ref{defn:geometric_dualrX} for $\dualrX^{\mathrm{g}'}$ follows directly from condition 2.\@ of Definition \ref{defn:branched_dualrX}.
  Furthermore,
  \begin{equation*}
   \sup_{s<t}\frac{\lvert\dualrX_{st}^{\mathrm{g}'}(w)\rvert}{\lvert t-s\rvert^{\gamma\lvert w\rvert}}=\sup_{s<t}\frac{\lvert\dualrX_{st}^{\mathrm{b}'}(\hat\phi(w))\rvert}{\lvert t-s\rvert^{\gamma\lvert\hat\phi(w)\rvert}}<\infty\quad\forallc w\in\lingen{\words_d^\level}
  \end{equation*}
  due to condition 3.\@ of Definition \ref{defn:branched_dualrX}, and thus we indeed have $\dualrX^{\mathrm{g}'}\in\geopaths^\gamma_2(\timeT,\reals^d)$.
  
  \noindent We also have that $\hat\phi\phi(x)-x\in\ker\phi$ for all $x\in\forestspace{d}$, since
  \begin{equation*}
   \phi(\hat\phi\phi(x)-x)=\phi\hat\phi\phi(x)-\phi(x)=\phi(x)-\phi(x)=0.
  \end{equation*}
  Thus,
  \begin{equation*}
   \dualrX^{\mathrm{b}'}(x)=\dualrX^{\mathrm{b}'}(\hat\phi\phi(x))=\dualrX^{\mathrm{g}'}(\phi(x))\quad\forallc x\in\lingen{\forests_d^\level}.
  \end{equation*}

 \end{enumerate}

\end{proof}

\chapter{Regularity structures}\label{chapter:regularity_structures}
\section{General concept}
\begin{defn}\label{defn:reg_struc}(Definition 3.1 \cite{hairer15}, Definition 2.1 \cite{hairer14}, Definition 13.1 \cite{frizhairer14})\newline
 Let 
 \begin{enumerate}
  \item $A\subseteq\reals$ be bounded from below and locally finite, \ie $A\cap(-\infty,r)$ finite for all $r\in\reals$,
  \item $(T_\alpha)_{\alpha\in A}$ be a family of non-trivial (\ie $\dim T_\alpha\neq0\,\,\,\,\forallc\alpha\in A$) Banach spaces over $\reals$ with norms $(\lVert\cdot\rVert_\alpha)_{\alpha\in A}$ and let $T:=\bigoplus_{\alpha\in A}T_\beta$ be endowed with the norm $\lVert\cdot\rVert:=\sum_{\alpha\in A}\lVert\pi_\alpha\cdot\rVert_\alpha$, where $\pi_\alpha$ is the canonical projection on $T_\alpha$,
  \item $G\subseteq\linearcont(T,T)$ be a group of invertible continuous linear operators with composition law, \ie $\id\in G$, $\varGamma_1\varGamma_2\in G$ and $\varGamma^{-1}\in G$ for all $\varGamma_1,\varGamma_2,\varGamma\in G$, with the property that
  \begin{equation}\label{eq:structuregroup_grading}
   \varGamma x-x\in T_{<\alpha}:=\bigoplus_{\beta\in A:\,\beta<\alpha}T_\beta\quad\forallc\varGamma\in G,\alpha\in A,x\in T_\alpha
  \end{equation}
\end{enumerate}
 Then, $\regstruc:=(A,T,G)$ is called a \emphind{regularity structure} with \emphind{index set} $A$, \emphind{model space} $T$ and \emphind{structure group} $G$.  \end{defn}
We say that a regularity structure $\bar\regstruc=(\bar A,\bar T,\bar G)$ is a \emphind{sub regularity structure} of a regularity structure $\regstruc=(A,T,G)$ if $\bar A\subseteq A$, $\bar G\subseteq G$ and $\bar T$ is a closed subspace of $T$ such that $\bar T_\alpha$ is a sub Banach space of $T_\alpha$ for all $\alpha\in\bar A$ (loosely based on Section 2.1 \cite{hairer14}). We put $T_{\geq\alpha}:=\bigoplus_{\beta\in A:\,\beta\geq\alpha} T_\beta$ and shortly write $\lVert\cdot\rVert_\alpha$ also for the seminorms $\lVert\pi_\alpha\cdot\rVert_\alpha$ on $T$.
\begin{rmk}
For each $A$ and $(T_\alpha)_\alpha$ fulfilling Properties 1.\@ and 2.\@ there is a maximal structure group given by
  \begin{equation*}
   G^\text{max}:=\{\varGamma\in\linearcont(T,T)|\varGamma x-x\in T_{<\alpha}\,\forallc\alpha\in A,x\in T_\alpha\},
  \end{equation*}
  which contains all possible choices of structure groups. Hence, choosing a structure group for a given model space means choosing certain additional properties that the elements have to fulfill, \eg multiplicity with respect to a given product on $T$. \end{rmk}
\begin{rmk}\label{rmk:reg_struc_with_unit}
 All of the examples of regularity structures $(A,T,G)$ that we present later will satisfy the additional property that $0\in A$ and that there is a unique element $\unit\in T_0$ such that $\lVert\unit\rVert_0=1$ and $\varGamma\unit=\unit$ for all $\varGamma\in G$.
\end{rmk}
The original definition of a regularity structure in \cite{hairer14} contained an even stronger requirement for the subspace $T_0$, namely that $T_0=\lingen{\unit}$, where $\unit$ is such that $\lVert\unit\rVert_0=1$ and $\varGamma\unit=\unit$ for all $\varGamma\in G$. Definition \ref{defn:reg_struc} here is exactly the one given in \cite{hairer15}.

\begin{defn}(Definition 2.17 \cite{hairer14})
 Let $\regstruc=(A,T,G)$ be a regularity structure and $(S,d)$ be a metric space. A map $\Gamma:\,S\times S\rightarrow G$ is called a \emph{$G$-bimap}\footnote{This nonstandard terminology just used for the purpose of this thesis.} over $S$ if
 \begin{enumerate}
  \item $\Gamma_{xx}=\id$ for all $x\in S$,
  \item $\Gamma_{xy}\Gamma_{yz}=\Gamma_{xz}$ for all $x,y,z\in S$,
  \item For every compact subset $\compact\subseteq S$ and every $\beta\in A$, there is a constant $C_{\compact,\beta}$ such that
  \begin{equation*}
   \lVert\Gamma_{xy}a\lVert_\theta\leq C_{\compact,\beta}\lVert a\rVert d(x,y)^{\beta-\theta}\quad\forallc x,y\in\compact,a\in T_\beta,\theta<\beta.
  \end{equation*}

 \end{enumerate}
\begin{defn}
 For a regularity structure $\regstruc=(A,T,G)$, let $V$ be a subspace of $T$ such that $V_\alpha:=V\cap T_\alpha$ is a (possibly $\{0\}$) closed subspace of $T_\alpha$ for all $\alpha\in A$. Then, $V$ is called a \emphind{sector} of $\regstruc$ if $V$ is invariant under $G$, \ie
 \begin{equation*}
  \varGamma V\subseteq V\quad\forallc\varGamma\in G.
 \end{equation*}
 We call the sector \emph{function-like}\index{function-like sector} if $V_\alpha=\{0\}$ for all $\alpha\in A$ with $\alpha<0$ and $V_0=\lingen{\unit}$, where $\unit$ is some element such that $\varGamma\unit=\unit$ for all $\varGamma\in G$.

\end{defn}
\end{defn}
\begin{defn}(Definition 3.1 \cite{hairer14})
 For a given regularity structure $(A,T,G)$, a $G$-bimap $\Gamma$ over $(S,d)$, $\gamma>\min A$, a closed subset $M\subseteq S$ and a sector $V$ of $(A,T,G)$, the space of \emphind{modelled distributions} $\modeld_\Gamma^\gamma(M,V)$ of order $\gamma$ consists of all functions $f:\,M\rightarrow V_{<\gamma}$ such that
 \begin{equation*}
  \lVERT f\rVERT_\compact:=\sup_{x\in\compact}\sup_{\alpha<\gamma}\lVert f(x)\rVert_\alpha+\sup_{\substack{x,y\in\compact:\\0<d(x,y)\leq 1}}\sup_{\alpha<\gamma}\frac{\lVert f(x)-\Gamma_{xy}f(y)\rVert_\alpha}{d(x,y)^{\gamma-\alpha}}<\infty
 \end{equation*}
 for all compact $\compact\subseteq M$. We shortly write $\modeld_\Gamma^\gamma:=\modeld_\Gamma^\gamma(S,T)$.
\end{defn}

The following definition is inspired by the definition of automorphisms in Section 2.4 \cite{hairer14}.
\begin{defn} 
 Let $\regstruc=(A,T,G)$ and $\bar{\regstruc}=(\bar A,\bar T,\bar G)$ be regularity structures. Furthermore, let $\varphi:\,T\rightarrow\bar T$ be a continuous linear map such that $\varphi T_\beta\subseteq\bar{T}_\beta^+$ for all $\beta\in A$ and $\psi:\,\bar G\rightarrow G$ be a group homomorphism. If
 \begin{enumerate}
  \item $\varGamma\varphi=\varphi\psi(\varGamma)\quad\forallc\varGamma\in\bar G$,
  \item for every $G$-bimap $\Gamma$ the function $(x,y)\mapsto\psi(\Gamma_{xy})$ is a $\bar G$-bimap,
 \end{enumerate}
 then $\varphi$ is called a model space morphism and $\psi$ is called a structure group morphism. The pair $(\varphi,\psi)$ is called a regularity structure morphism from $\bar{\regstruc}$ to $\regstruc$. The morphisms $\varphi$, $\psi$ and $(\varphi,\psi)$ are called homogeneous if $\varphi T_\beta\subseteq\bar{T}_\beta$.
\end{defn}
For $x\in\reals^d$ and $\delta>0$ let $\mathfrak{S}_x^\delta\in\linearcont(\schwartz(\reals^d),\schwartz(\reals^d))$ be given by
\begin{equation*}
 (\mathfrak{S}_x^\delta u)(y):=\delta^{-d}u(\delta^{-1}(y-x)).
\end{equation*}
Furthermore, for any $r\in\naturals$, let $\testfcts{r}{d}$ denote the set of all $u\in\schwartz(\reals^d)$ compactly supported in $\mathrm{B}(0,1)$ such that $\lVert u\rVert_{\mathrm{C}^r}\leq 1$. 
\begin{defn}\label{defn:model}(Definition 2.17 \cite{hairer14})
 Let $\regstruc=(A,T,G)$ be a regularity structure, $\Pi:\,\reals^d\rightarrow\linearcont(T,\schwartzd(\reals^d))$ be a function and $\Gamma$ be a $G$-bimap over $\reals^d$. Furthermore, let $r$ be the smallest integer such that $r>\lvert\min A\rvert$. We call $(\Pi,\Gamma)$ a model for $\regstruc$ on $\reals^d$ if
 \begin{enumerate}
  \item $\Pi_x\Gamma_{xy}=\Pi_y$ for all $x,y\in\reals^d$,
  \item For every compact subset $\compact\subseteq\reals^d$ and every $\beta\in A$, there is a constant $C_{\compact,\beta}$ such that
  \begin{equation*}
   \lvert(\Pi_x a)(\mathfrak{S}_x^\delta u)\rvert\leq C_{\compact,\beta}\lVert a\rVert\delta^\beta\quad\forallc x\in\compact,\,\delta\in(0,1],\,a\in T_\beta,\,u\in\testfcts{r}{d}.
  \end{equation*}
 \end{enumerate}
 Let $\spaceofmodels(\reals^d)$ denote the set of all models for $\regstruc$ on $\reals^d$. If $\regstruc$ satisfies the additional property of Remark \ref{rmk:reg_struc_with_unit}, then let $\spaceofmodels_1(\reals^d)\subseteq\spaceofmodels(\reals^d)$ denote those models with the additional property that $\Pi_x\unit$ is the constant one function for all $x\in\reals^d$.
\end{defn}

\begin{defn}(Definition 14.3 \cite{frizhairer14}, based on Definition 4.1 and Definition 4.6 \cite{hairer14})
 Let $V,\bar V$ be two sectors of a regularity structure $\regstruc=(A,T,G)$. A \emph{product}\index{product!in a regularity structure} on $V\times\bar V$ is a continuous bilinear map $\product:\,V\times\bar V\rightarrow T$ such that
 \begin{enumerate}
  \item $V_\alpha\product V_\beta\subseteq T_{\alpha+\beta}$,
  \item $\Gamma(v\product\bar v)=\Gamma v\product\Gamma\bar v\quad\forallc\Gamma\in G,\,v\in V,\,\bar v\in\bar V$.
 \end{enumerate}

\end{defn}
\begin{thm}\label{thm:modeld_composition}\textnormal{(Section 4.2 and Theorem 4.16 \cite{hairer14}, Section 14.2 and Proposition 14.7. \cite{frizhairer14})}
 Let $V$ be a function-like sector of a regularity structure $\regstruc=(A,T,G)$ with $\vartheta$ being the lowest non-zero homogeneity of $V$ and $\product$ a product on $V\times V$ such that $V\product V\subseteq V$. Furthermore, let $\gamma>0$, $h\in\contdiff^k(\reals,\reals)$ for some $k\in\naturals$ with $k\geq\gamma/\vartheta+1$ and $\Gamma$ a $G$-bimap over $\reals^d$.
 Then, $F\mapsto(h\circ F)_\gamma$, where
 \begin{equation*}
  (h\circ F)_\gamma(x):=\sum_{n=0}^{\lfloor\gamma/\vartheta\rfloor}\quotientmap_{<\gamma}\frac{h^{(n)}(F^\unit(x))}{n!}(F(x)-F^\unit(x)\unit)^{\product n},
 \end{equation*}
 $v^{\product 0}:=\unit$ and $F^\unit(x)\unit:=\pi_0 F(x)$, is a continuous map from $\modeld_\Gamma^\gamma$ into itself. More precisely, for any compact $\compact\subseteq\reals^d$ and any $D>0$, there is a $C_{\compact,D}>0$ such that
 \begin{equation*}
  \lVERT (h\circ F)_\gamma-(h\circ G)_\gamma\rVERT_{\gamma,\compact}\leq C_{\compact,D}\lVERT F-G\rVERT_{\gamma,\compact}\quad\forallc F,G\in\modeld_\Gamma^\gamma:\,\lVERT F\rVERT_{\gamma,\compact}+\lVERT G\rVERT_{\gamma,\compact}\leq D.
 \end{equation*}
\end{thm}
\begin{defn}(Section 2.2 and Definition 3.7 \cite{hairer14})
 \begin{enumerate}
  \item For $n\in\naturals_0$ and $\alpha\in(0,1]$, let $\hoelder^{n+\alpha}(\reals^d)$ denote those functions in $\mathrm{C}(\reals^d,\reals)$ which are $n$ times differentiable and whose $n$-th derivative is Hölder continuous with exponent $\alpha$. We endow $\hoelder^{n+\alpha}(\reals^d)$ with the family of seminorms
  \begin{equation*}
   \lVert f\rVert_{\hoelder^{n+\alpha},\compact}:=\sup_{x,y\in\compact}\frac{\lvert f^{(n)}(x)-f^{(n)}(y)\rvert}{\lVert x-y\rVert^\alpha}+\sum_{k=0}^n\,\sup_{x\in\compact}\lvert f^{(k)}(x)\rvert
  \end{equation*}
  indexed by all compact $\compact\subset\reals^d$.
  \item Put $\hoelder^0(\reals^d):=\lebesgue_{\text{loc}}^\infty(\reals^d,\reals)$.
  \item For $n\in\naturals$ and $\alpha\in[0,1)$, let $\hoelder^{\alpha-n}(\reals^d)$ denote those distributions $\varrho\in\tdual{\mathrm{D}}(\reals^d)$ such that for each $\compact\subset\reals^d$, we have
  \begin{equation*}
   \lVert\varrho\rVert_{\hoelder^{\alpha-n},\compact}:=\sup_{u\in\testfcts{n}{d}}\sup_{x\in\compact}\sup_{\delta\in(0,1]}\frac{\varrho(\mathfrak{S}_x^\delta u)}{\delta^{\alpha-n}}<\infty.
  \end{equation*}
  We endow $\hoelder^{\alpha-n}(\reals^d)$ with this family of seminorms.
 \end{enumerate}
\end{defn}

\begin{thm}\index{reconstruction theorem}\textnormal{\textbf{Reconstruction Theorem} (Theorem 3.10 and Corollary  \cite{hairer14})}
 Let $(\Pi,\Gamma)$ be a model over $\reals^d$ for a regularity structure $\regstruc=(A,T,G)$ and $\gamma>\alpha:=\min A$. Then, there is a continuous linear map $\reconstr^\alpha:\,\modeld_\Gamma^\gamma\rightarrow\hoelder^{\alpha}(\reals^d,\reals)$ with the property that for every compact $\compact\subset\reals^d$ there is a constant $C_\compact>0$ such that
 \begin{equation}\label{eq:reconstr_thm}
  \lvert(\reconstr^\alpha F-\Pi_x F(x))(\mathfrak{S}_x^\delta u)\rvert\leq C_\compact\lVERT F\rVERT_{\compact'}\delta^\gamma\quad\forallc x\in\compact,\,\delta\in(0,1],\,F\in\modeld_\Gamma^\gamma,\,u\in\testfcts{r}{d},
 \end{equation}
 where $\compact':=\bigcup_{x\in\compact}\bar{\mathrm{B}}(x,1)$. Such a map is called a \emph{reconstruction map}. We furthermore have that for every compact $\compact\subset\reals^d$ there is a $C'_\compact$ such that
 \begin{equation*}
  \lVert\reconstr^\alpha F\rVert_{\hoelder^\alpha,\compact}\leq C'_\compact\lVERT F\rVERT_{\gamma,\compact'}.
 \end{equation*}
 If $\gamma>0$, then the reconstruction map is unique and for each $F\in\modeld_\Gamma^\gamma$, the distribution $\reconstr F$ is uniquely characterized by equation \eqref{eq:reconstr_thm}.
\end{thm}
\begin{thm}\label{thm:reconstr_fctlike}\textnormal{(Special case of Proposition 3.28 \cite{hairer14})}
 Let $V$ be a function-like sector of a regularity structure $\regstruc=(A,T,G)$ with $\vartheta$ being the lowest non-zero homogeneity of $V$. Let $(\Pi,\Gamma)$ be a model for $\regstruc$ on $\reals^d$ such that $\Pi_x\unit$ is the constant one function for all $x\in\reals^d$. Then, for $\gamma>\vartheta$, $\reconstr^\gamma$ maps $\modeld_\Gamma^\gamma(V)$ into $\hoelder^\vartheta(\reals^d)$ and we have
 \begin{equation*}
  (\reconstr^\gamma F)(x)\unit=\pi_0 F(x)\quad\forallc F\in\modeld_\Gamma^\gamma(V).
 \end{equation*}
 In particular, we have $\Pi_x a\in\hoelder^\vartheta(\reals^d)$ for all $a\in V$.
\end{thm}
\noindent Note that the 'in particular' part indeed follows from the general statement since $\Gamma_{\cdot x}a\in\modeld_\Gamma^\gamma(V)$ for all $\gamma$ large enough and
\begin{equation*}
 \Pi_x a-\Pi_y\Gamma_{y x}a=0\quad\forallc y\in\reals^d,
\end{equation*}
thus $\Pi_x a=\reconstr^\gamma(\Gamma_{\cdot x}a)\in\hoelder^\vartheta(\reals^d)$ since the reconstruction of a modeled distribution is characterized by equation \eqref{eq:reconstr_thm}.

\section{Construction of a regularity structure based on a Hopf algebra}
\markright{CONSTRUCTION OF A REGULARITY STRUCTURE BASED ON A H.\@ ALGEBRA}
\subsection{General construction}

We closely follow Section 4.3 \cite{hairer14} in this subsection. Let $(H,\product,\coproduct,\antipode)$ be a locally finite dimensional connected graded Hopf algebra over $\reals$ whose grading differentiates into a $d$-dimensional grading
\begin{equation*}
 H=\bigoplus_{n\in\naturals_0^d}H_n
\end{equation*}
for some $d\in\naturals$, \ie
\begin{equation*}
 H_n\product H_m\subseteq H_{n+m},\quad\coproduct H_n\subseteq\bigoplus_{0\leq k\leq n}H_k\etensor H_{n-k},\quad\antipode H_n\subseteq H_n\quad\forallc n,m\in\naturals_0^d.
\end{equation*}
With locally finite dimensional we mean that $H_n$ is finite dimensional for all $n\in\naturals_0^d$. The notion of a $d$-dimensional grading of course reduces to the usual notion of a grading in the case $d=1$ (see Definition \ref{defn:grading}). Also, we obviously get back a familiar one dimensional grading $(\check{H}_i)_{i\in\naturals_0}$ by putting
\begin{equation*}
 \check{H}_i:=\bigoplus_{n\in\naturals_0^d:\,\lvert n\rvert=i}H_n.
\end{equation*}
Now, choose scaling factors $\alpha=(\alpha_i)_{i=1}^d$, where $\alpha_i\in\reals^{+}$, and put $\langle\alpha,n\rangle:=\sum_{i=1}^d \alpha_i n_i$ for all $n\in\naturals_0^d$. Then, we may introduce the index set $A$ and the model space $T$ as
\begin{equation*}
 A:=\bigcup_{n\in\naturals_0^d}\{\langle\alpha,n\rangle\},\quad T_\gamma:=\bigoplus_{n\in\naturals_0^d:\,\langle\alpha,n\rangle=\gamma}H_n,\quad T:=\bigoplus_{\gamma\in A}T_\gamma.
\end{equation*}
Let $G$ denote the set of characters on $T=H$, \ie
\begin{equation*}
 G:=\{g\in\adual{H}|g(x\product y)=g(x)g(y)\,\forallc x,y\in H\}.
\end{equation*}
$G$ is a group together with the convolution product $\convoproduct$ on $\adual{H}=\linear(H,\reals)$.
Then, define linear maps $\structaction{}^{l/r}:\,\adual{H}\rightarrow\linear(H,H)$ by
\begin{equation*}
 \structaction{g}^l:=(g\etensor\id)\coproduct
\end{equation*}
and
\begin{equation*}
 \structaction{g}^r:=(\id\etensor g)\coproduct.
\end{equation*}
In both cases, we have the following.
\begin{thm}\label{thm:regstruc_for_hopfalg}\textnormal{(Based on Section 4.3 \cite{hairer14})}
 $(A,T,\structaction{g}^l)$ and $(A,T,\structaction{g}^r)$ are regularity structures with $\min A=0$. $\bar\Gamma^l:\,G\rightarrow\structaction{g}^l$ is a group antiisomorphism and $\bar\Gamma^r:\,G\rightarrow\structaction{g}^r$ is a group isomorphism. The structure groups have the characterizations
 \begin{equation*}
  \structaction{G}^l=\{\varGamma\in\linear(H,H)|\varGamma\mproduct=\mproduct(\varGamma\etensor\varGamma),\,\coproduct\varGamma=(\varGamma\etensor\id)\coproduct\}
 \end{equation*}
 \resp
 \begin{equation*}
  \structaction{G}^r=\{\varGamma\in\linear(H,H)|\varGamma\mproduct=\mproduct(\varGamma\etensor\varGamma),\,\coproduct\varGamma=(\id\etensor\varGamma)\coproduct\}.
 \end{equation*}
 In particular, $\product$ is a product on the whole $T\times T$.
\end{thm}
\begin{proof}(Based on Section 4.3 \cite{hairer14})
 We only show the case of $\structaction{G}^l$, the other case is completely analogous.
 \begin{enumerate}
  \item\emph{$\structaction{}^l$ is a group antiisomorphism.} For all $g,h\in G$, we have
  \begin{align*}
   \structaction{g}^l\structaction{h}^l&=(g\etensor\id)\coproduct(h\etensor\id)\coproduct=(h\etensor g\etensor\id)(\id\etensor\coproduct)\coproduct=(h\etensor g\etensor\id)(\coproduct\etensor\id)\coproduct\\
   &=((h\convoproduct g)\etensor\id)\coproduct=\structaction{h\convoproduct g}^l.
  \end{align*}
  \item\emph{$(A,T,\structaction{g}^l)$ is a regularity structure.} The only thing left to show is \eqref{eq:structuregroup_grading}. Let $g\in G$, $n\in\naturals_0^d$ and $x\in H_n$ be arbitrary. Then, $x\in T_{\langle\alpha,n\rangle}$ and by the multi grading of $\coproduct$, 
  \begin{equation*}
   \structaction{g}^l x-x=(g\etensor\id)\coproduct x-x=((g-\counit)\etensor\id)\coproduct x\in\bigoplus_{m<n}H_m=T_{<\langle\alpha,n\rangle},
  \end{equation*}
  since $(g-\counit)\unit=0$ because $g$ is in $G$.
  \item\emph{The given characterization of $\structaction{G}^l$ holds.} First of all, for any $g\in G$, we have
  \begin{align*}
   \structaction{g}^l\mproduct&=(g\etensor\id)\coproduct\mproduct=(g\etensor\id)(\mproduct\etensor\mproduct)\tau_{1324}(\coproduct\etensor\coproduct)=\mproduct(g\etensor g\etensor\id\etensor\id)\tau_{1324}(\coproduct\etensor\coproduct)\\
   &=\mproduct(g\etensor\id\etensor g\etensor\id)(\coproduct\etensor\coproduct)=\mproduct(\structaction{g}^l\etensor\structaction{g}^l).
  \end{align*}
  and
  \begin{equation*}
   \coproduct\structaction{g}^l=\coproduct(g\etensor\id)\coproduct=(g\etensor\id\etensor\id)(\id\etensor\coproduct)\coproduct=(g\etensor\id\etensor\id)(\coproduct\etensor\id)\coproduct=(\structaction{g}^l\etensor\id)\coproduct.
  \end{equation*}
  Thus, we have shown the '$\subseteq$' part of the characterization. To show the '$\supseteq$' part, let $\varGamma\in\linear(H,H)$ such that $\varGamma\mproduct=\mproduct(\varGamma\etensor\varGamma)$ and $\coproduct\varGamma=(\varGamma\etensor\id)\coproduct$ be arbitrary. Put $h:=\counit\varGamma$. Then,
  \begin{equation*}
   h\mproduct=\counit\varGamma\mproduct=\counit\mproduct(\varGamma\etensor\varGamma)=(\counit\etensor\counit)(\varGamma\etensor\varGamma)=(h\etensor h),
  \end{equation*}
  which implies that $h\in G$. Also,
  \begin{equation*}
   \varGamma=(\counit\etensor\id)\coproduct\varGamma=(\counit\etensor\id)(\varGamma\etensor\id)\coproduct=(h\etensor\id)\coproduct=\structaction{h}^l.
  \end{equation*}
   Hence, the '$\supseteq$' part is also shown.
 \end{enumerate}
\end{proof}

\subsection{Polynomial case}
 This subsection is a more explicit formulation of what is described mostly in words in Remark 4.19 \cite{hairer14}. Considering the polynomial Hopf algebra $\polynoms_d$, we have an obvious $d$-dimensional grading given by $(\lingen{\X^n})_{n\in\naturals_0^d}$. But, as its natural to choose all scaling factors $\alpha_i=1$, we get $A=\naturals_0$ and $T=\bigoplus_{k\in A}T_k$ with $T_k=\polynoms_{d,k}$. We identify $\adual{\polynoms_d}$ with the space of formal series of differential operators $\spacefs{(\diffs_d)}$ via the duality pairing \eqref{eq:dualitypairing_diff_pol}, writing 
 \begin{equation*}
  \bra{D}(p):=\langle D,p\rangle
 \end{equation*}
 for $D\in\spacefs{(\diffs_d)},\,p\in\polynoms_d$.
 \begin{align*}
  \structaction{\bra{\partial^n}}\X^m&=(\bra{\partial^n}\etensor\id)\copropol\X^m=\sum_{0\leq l\leq m}\binom{m}{l}\underbrace{(\partial^n\X^l)(0)}_{=\deltasymb{n,l}n!}\X^{m-l}=\begin{cases}\frac{m!}{(m-n)!}\X^{m-n}&\text{if }n\leq m\\0&\text{if }n\nleq m\end{cases}\\
  &=\partial^n\X^m
 \end{align*}
 Thus, for a formal series $D$ of differential operators, $D\in\spacefs{(\diffs_d)}\subset\linear(\polynoms_d,\polynoms_d)$, we have
 \begin{equation*}
  \structaction{\bra{D}}=D.
 \end{equation*}
 Since the group of characters $G$ can be written as
 \begin{equation*}
  G=\big\{\bra{D}\big|D\in\exp_\circ(\diffs_{d,1})\big\},
 \end{equation*}
 we may write the structure group as
 \begin{equation*}
  \structaction{G}=\exp_\circ(\diffs_{d,1}),
 \end{equation*}
 \ie the structure group consists of the composition exponentials of first order differential operators. Since
 \begin{equation*}
  \exp_\circ(\sum_{i=1}^d a_i\partial_i)=\exp_\circ(a_1\partial_1)\cdots\exp_\circ(a_d\partial_d)
 \end{equation*}
 due to commutativity of $(\diffs_d,\circ)$,
 \begin{equation*}
  \exp_\circ(a_i\partial_i)\X_j=\sum_{k=0}^\infty\frac{a_i^k}{k!}\partial_i^k\X_j=\X_j+\deltasymb{i,j}a_i
 \end{equation*}
 and $\Gamma(x\propol y)=\Gamma x\propol\Gamma y$ for all $\Gamma\in\structaction{G}$, we have
 \begin{equation*}
  \exp_\circ(\sum_{i=1}^d a_i\partial_i)\X^n=\prod_{i=1}^d(\X_i+a_i)^{n_i}=:(\X+a)^n\quad\forallc a\in\reals^d,
 \end{equation*}
 \ie the structure group $\structaction{G}$ is nothing but the group of translations of the $d$-dimensional polynomials.
\subsection{Connes-Kreimer Hopf algebra case}
This section is basically a short formulation of what was suggested in Remark 4.25 \cite{hairer14}. Starting from the Hopf algebra $(\forestspace{d},\treepro,\destar,\antitree)$, the one dimensional grading and the scaling factor $\gamma$, we get a regularity structure $\regstruc^\mathrm{f}=(A^\mathrm{f},T^\mathrm{f},\structaction{\butchergroup}^\mathrm{f})$, where $\butchergroup$ is the group of characters on $\forestspace{d}$, $A^\mathrm{f}=\gamma\naturals_0$ and of course $T^\mathrm{f}=\forestspace{d}$ with $T_{k\gamma}^\mathrm{f}=\lingen{\forests_{d,k}}$. The group $\butchergroup$ can be identified with $\grouplike(\forestspace{d},\detreepro)$ via Theorem \ref{thm:grouplike_character}. The map $\structaction{}^\mathrm{f}$ is then recursively given by
\begin{equation*}
 \structaction{g}^\mathrm{f}\unit=\unit,\quad\structaction{g}^\mathrm{f}\lroot\zeta\rroot_i=g(\lroot\zeta\rroot_i)\unit+\lroot\structaction{g}^\mathrm{f}\zeta\rroot_i,\quad\structaction{g}^\mathrm{f}(\zeta_1\treepro\zeta_2)=\structaction{g}^\mathrm{f}\zeta_1\treepro\structaction{g}^\mathrm{f}\zeta_2.
\end{equation*}
Let $\spaceofmodels_1^\mathrm{f}([0,T])$ denote those models $(\Pi,\Gamma)\in\spaceofmodels_1^\mathrm{f}(\reals)$ such that $\Gamma_{st}=\id$ for $s,t\in(-\infty,0]$ and for $s,t\in[T,\infty)$.
\begin{thm}\label{thm:bijection_branched_model}\textnormal{(Based on Section 4.4 \cite{hairer14})}
 There is a bijective map $I:\,\branchedpathsdualfull^\gamma(\reals,\reals^d)\rightarrow\spaceofmodels_1^\mathrm{f}([0,T])$ which maps a branched rough path $\dualfullrX$ to the unique model $(\Pi,\Gamma)\in\spaceofmodels_1^\mathrm{f}([0,T])$ such that
 \begin{equation}\label{eq:Pi_Gamma_for_dualfullrX}
  (\Pi_s x)(t):=\dualfullrX_{st}(x),\quad\Gamma_{st}x:=(\dualfullrX_{ts}\etensor\id)\destar x=\sum_{(x)}^{\star}\dualfullrX_{ts}(x_1)x_2\quad\forallc s,t\in[0,T],\,x\in\forestspace{d}
 \end{equation}
\end{thm}
\begin{proof}
 Let $\dualfullrX\in\branchedpathsdualfull^\gamma([0,T],\reals^d)$ be arbitrary. By Theorem \ref{thm:branchedrough_models}, there is a unique $\structaction{G}^\mathrm{f}$-bimap $\Gamma$ over $\reals$ such that it fulfills \eqref{eq:Pi_Gamma_for_dualfullrX} and such that $\Gamma_{st}=\id$ for $s,t\in(-\infty,0]$ and for $s,t\in[T,\infty)$ (just put $\Gamma_{st}=\Gamma_{0t}$ if $s<0$ and $\Gamma_{st}=\Gamma_{Tt}$ if $s>T$). Furthermore, there is a unique $\Pi$ on $\reals$ such that \eqref{eq:Pi_Gamma_for_dualfullrX} holds and such that $\Pi_s=\Pi_t\Gamma_{ts}$ for $s,t\in\reals$, in particular $\Pi_s=\Pi_0$ if $s<0$ and $\Pi_s=\Pi_T$ if $s>T$. For this $\Pi$ and for every $k\in\naturals_0$, there is a $C_k>0$ such that
 \begin{equation*}
  \lvert(\Pi_s\zeta)(t)\rvert\leq C_k\lvert t-s\rvert^{\gamma\lvert\zeta\rvert}\quad\forallc\zeta\in\forests_{d,k},\,t,s\in\reals
 \end{equation*}
 Thus, putting $f_s:=\Pi_s\zeta$ for some $\zeta\in\forests_d$, we get
 \begin{align*}
  \lvert(\Pi_s x)(\mathfrak{S}_s^\delta u)\rvert&=\left\lvert\delta^{-1}\int_\reals f_s(t)u(\delta^{-1}(t-s))\intd t\right\rvert=\left\lvert\int_\reals f_s(s+\delta t)u(t)\intd t\right\rvert\leq\int_\reals\lvert f_s(s+\delta t)\rvert\lvert u(t)\rvert\intd t\\
  &\leq C_{\lvert\zeta\rvert}\delta^\alpha T
 \end{align*}
 for all $s\in\reals$ and $\delta\in(0,1]$. By linearity and local finite dimensionality, the analytic condition for $\Pi$ in Definition \ref{defn:model} of a model follows. The map $I$ is therefore well-defined. It is also injective, since $\dualfullrX_{st}=(\Pi_s\cdot)(t)$ for all $s,t\in[0,T]$. It only remains to show surjectivity.
 
 \noindent Let $(\Pi,\Gamma)\in\spaceofmodels_1^\mathrm{f}([0,T])$ be arbitrary. Then, due to Theorem \ref{thm:reconstr_fctlike}, $\Pi_s x$ is a continuous function for all $s\in\reals$ and $x\in\forestspace{d}$. Due to continuity and the analytic condition for $\Pi$ in Definition \ref{defn:model}, we have $(\Pi_t \zeta)(t)=0$ for all $\zeta\in\forests_d\setminus\{\unit\}$ and $t\in\reals$. Since $\Gamma$ is a $\structaction{G}^\mathrm{f}$-bimap, for all $k\in\naturals_0$ there is a $C_k>0$ such that
 \begin{equation*}
  \lvert(\Pi_s\zeta)(t)\rvert=\lvert(\Pi_t\Gamma_{ts}\zeta)(t)=\lVert\Gamma_{ts}\zeta\rVert_0\leq C_{\lvert\zeta\rvert}\lvert t-s\rvert^{\gamma\lvert\zeta\rvert}.
 \end{equation*}
 Finally, due to Theorem \ref{thm:branchedrough_models} there is a $\dualfullrX\in\branchedpathsdualfull^\gamma([0,T],\reals^d)$ such that $\eqref{eq:Pi_Gamma_for_dualfullrX}$ holds for $\Pi$ and $\Gamma$.
 
 \end{proof}

\begin{defn}(Reformulation of Definition 3.2.\@ \cite{hairerkelly14})
 For a branched rough path $\dualfullrX\in\branchedpathsdualfull^\gamma([0,T],\reals^d)$, the model $(\Pi,\Gamma)=I(\dualfullrX)$ and $\level=\lfloor\gamma^{-1}\rfloor$, we call $\modeld_\Gamma^{\level\gamma}$ the space of \emph{controlled rough paths}\index{controlled rough path} of $\dualfullrX$.
\end{defn}
Applying Theorem \ref{thm:modeld_composition} to $\modeld_\Gamma^{\level\gamma}$ yields exactly the one-dimensional case of the composition formula for controlled rough paths given in Equation (3.8) \cite{hairerkelly14}.
\subsection{Shuffle-deconcatenation Hopf algebra case}
 The same can be done for the Hopf algebra $(\lingen{\words_d},\shuffle,\deconc,\antishuffle)$, yielding a regularity structure $\regstruc^\mathrm{w}=(A^\mathrm{w},T^\mathrm{w},\structaction{{G^\mathrm{w}}}^\mathrm{w})$, where $G^\mathrm{w}$ denotes the group of characters on $\lingen{\words_d}$ and where again $A^\mathrm{w}=\gamma\naturals_0$ and $T^\mathrm{w}=\lingen{\words_d}$ with $T_{k\gamma}^\mathrm{w}=\lingen{\words_{d,k}}$. This structure was already discussed in detail in Section 4.4 \cite{hairer14}.
 
 Here, we have the recursion
 \begin{equation*}
  \structaction{g}^\mathrm{w}\unit=\unit,\quad\structaction{g}^\mathrm{w}(w\itensor e_i)=g(w\itensor e_i)\unit+\structaction{g}^\mathrm{w}w\itensor e_i.
 \end{equation*}
 
Due to Theorem \ref{thm:geom_roughpaths_reg_str}, there is a straightforward analogue to Theorem \ref{thm:bijection_branched_model} in this case.
\section{Rough path structures revisited}
\subsection{Motivation}
Another way to construct regularity structures for rough paths is to perform an abstract fixed point iteration in the spirit of Section 15.2 \cite{frizhairer14}.
The problem we look at is, for smooth functions $f_i$ and drivers $x^i$ with $x_t^i=0$ for $t\leq 0$,
\begin{equation*}
 \diffd y_t=\sum_{i=1}^d f_i(y_t)\diffd x_t^i
\end{equation*}
Putting $\xi_t^i:=\dot x_t^i$, this may formally be reformulated as
\begin{equation*}
 \dot y_t=\sum_{i=1}^d f_i(y_t)\xi_t^i
\end{equation*}
or, as a formal integral equation,
\begin{equation*}
 y_t-y_0=\sum_{i=1}^d\int_0^t f_i(y_t)\xi_t^i\intd t.
\end{equation*}
Introducing the integration kernel $K:=\chi_{[0,\infty)}$, we have
\begin{equation*}
 y_t-y_0=\left[K\ast\Big(\sum_{i=1}^d f_i(y)\xi^i\Big)\right](t).
\end{equation*}
This may be translated to an abstract equation
\begin{equation}\label{eq:abstract_rde}
 Y=y_0\unit+\abstractkernel\dot Y\quad\text{and}\quad\dot Y=\sum_{i=1}^d f_i(Y)\Xi_i.
\end{equation}
for modeled distributions $Y$ and $\dot Y$. The linear operator $\abstractkernel$ acting on certain modeled distributions corresponds to the convolution with the kernel $K$.
It will turn out that $\abstractkernel$ can be written as
\begin{equation*}
 (\abstractkernel F)(t)=\abstractint\,F(t)+(K\ast\reconstr F)(t)\unit
\end{equation*}
for some linear operator $\abstractint$ acting on a sector of the model space which we want to introduce. In view of Theorem \ref{thm:modeld_composition}, we will define $f_i(Y)\Xi_i$ as
\begin{equation*}
 \sum_{k=0}^\infty\frac{f_i^{(k)}(Y^\unit)}{k!}(Y-Y^\unit)^k\Xi_i.
\end{equation*}
and later cut off that infinite series at some level. Let therefore $\mathcal{Y}=\bigcup_{n=0}^\infty\mathcal{Y}_n$ be the set of symbols necessary to describe $Y$ for any family of smooth functions $(f_i)_{i=1}^d$ and $\dot{\mathcal{Y}}=\bigcup_{n=0}^\infty\dot{\mathcal{Y}}_n$ the set of symbols necessary to describe $\dot Y$. Starting from $\mathcal{Y}_0=\{\unit\}$, we obtain the sets by iteratively applying \eqref{eq:abstract_rde} and collecting the symbols we have not yet included, \ie
\begin{alignat*}{3}
 \dot{\mathcal{Y}}_0&=\{\Xi_i|i\in I\},&\quad\mathcal{Y}_1&=\{\abstractint(\Xi_i)|i\in I\},\\
 \dot{\mathcal{Y}}_1&=\{\abstractint(\Xi_{i_1})\cdots\abstractint(\Xi_{i_m})\Xi_j|m\in\naturals,\,i_k,j\in I\},&\,\,\,\quad\mathcal{Y}_2&=\{\abstractint(\abstractint(\Xi_{i_1})\cdots\abstractint(\Xi_{i_m})\Xi_j)|m\in\naturals,\,i_k,j\in I\},\\
 &\,\,\,\vdots&\quad&\,\,\,\vdots\\
 \dot{\mathcal{Y}}_n&=\{\zeta_1\cdots\zeta_m\Xi_i|m\in\naturals,\,\zeta_j\in\mathcal{Y}_n,\,i\in I\},&\mathcal{Y}_{n+1}&=\{\abstractint(\zeta)|\zeta\in\dot{\mathcal{Y}}_n\},
\end{alignat*}
where $I:=\naturals\cap[1,d]$.
Obviously, we may identify
\begin{equation*}
 \mathcal{Y}\equiv\trees_d\cup\{\unit\},\quad\dot{\mathcal{Y}}\equiv\bigcup_{i=1}^d\forests_d\Xi_i
\end{equation*}
via
\begin{equation*}
 \abstractint(\tau_1\cdots\tau_m\Xi_i)\equiv\lroot\tau_1\cdots\tau_m\rroot_i
\end{equation*}
\subsection{Branched rough paths structure}
Thus, as a model space, we put
\begin{equation*}
 \branchedmodelspace=\langle\mathcal{Y}\rangle\oplus\langle\dot{\mathcal{Y}}\rangle=\langle\unit\rangle\oplus\langle\trees_d\rangle\oplus\bigoplus_{i=1}^d\langle\forests_d\Xi_i\rangle
\end{equation*}
with homogenities
\begin{equation*}
 \lvert\unit\rvert=0,\quad\lvert\Xi_i\rvert=\gamma-1,\quad\lvert\tau_1\cdots\tau_m\Xi_i\rvert=\lvert\tau_1\rvert+\cdots+\lvert\tau_m\rvert+\gamma-1,\quad\lvert\lroot\zeta\rroot_i\rvert=\lvert\abstractint(\zeta\Xi_i)\rvert=\lvert\zeta\rvert+\gamma
\end{equation*}
and hence index set
\begin{equation*}
 A^\mathrm{b}:=\gamma\naturals_0\cup(\gamma\naturals-1).
\end{equation*}

$\branchedmodelspace$ then can be turned into a comodule of $(\forestspace{d},\destar)$ via
\begin{equation}\label{eq:defn_branched_coproduct}
 \coproduct^\mathrm{b}:\,\branchedmodelspace\rightarrow\forestspace{d}\etensor \branchedmodelspace,\quad\coproduct^\mathrm{b}\zeta\Xi_i=\sum_{(\zeta)}^\star\zeta_1\etensor\zeta_2\Xi_i,\quad\coproduct^\mathrm{b}\lroot\zeta\rroot_i=\destar\lroot\zeta\rroot_i
\end{equation}
which allows us to define the action of some $g\in\adual{\forestspace{d}}$ on $\branchedmodelspace$ as
\begin{equation}\label{eq:defn_branched_structaction}
 \structaction{g}^\mathrm{b}:=(g\etensor\id)\coproduct^\mathrm{b}.
\end{equation}
Denoting the group of characters on $\forestspace{d}$ as $\butchergroup$, we get a regularity structure $\regstruc^\mathrm{b}:=(A^\mathrm{b},\branchedmodelspace,\structaction{\butchergroup}^\mathrm{b})$ for branched rough paths. It is some kind of a generalisation of the rough paths structure given in Definition 13.5.\@ \cite{frizhairer14} for the case of $\gamma\in(1/3,1/2]$, though they work with a finite dimensional model space there, which would also be possible here via truncation.
\begin{thm}\label{thm:branched_regstruc}
 $\regstruc^\mathrm{b}$ is a regularity structure. $\sectorY$ and $\sectorYdot$ are sectors of $\regstruc^\mathrm{b}$.
\end{thm}
\begin{proof}
 The fact that $\structaction{\butchergroup}^\mathrm{b}$ is a structure group follows again from the grading of $\destar$ and the group antimorphism property of $\structaction{}^\mathrm{b}$ (\cf Theorem \ref{thm:regstruc_for_hopfalg}). Since $\coproduct^\mathrm{b}$ by definition maps $\sectorY$ to $\forestspace{d}\etensor\sectorY$ and $\sectorYdot$ to $\forestspace{d}\etensor\sectorYdot$, they are indeed sectors.
\end{proof}

There is an obvious group isomorphism $\psi$ given by $\psi(\structaction{g}^\mathrm{f})=\structaction{g}^\mathrm{b}$, which is also a structure group morphism. It can be expanded to an injective map
\begin{equation*}
 \varPhi:\,\spaceofmodels^\mathrm{f}(\reals)\rightarrow\spaceofmodels^\mathrm{b}(\reals),\,\varPhi(\Pi,\Gamma):=(\varphi(\Pi),\psi(\Gamma))
\end{equation*}
via
\begin{equation*}
 \varphi(\Pi)_x\zeta:=\Pi_x\zeta\quad\forallc\zeta\in\sectorY,\quad\varphi(\Pi)_x\zeta:=\tfrac{\diffd}{\diffd t}\Pi_x\zeta\quad\forallc\zeta\in\sectorYdot.
\end{equation*}
We have that
\begin{equation*}
 \varPhi(\spaceofmodels_1^\mathrm{f}(\reals))=\big\{(\Pi,\Gamma)\in\spaceofmodels_1^\mathrm{b}(\reals)\big|\tfrac{\diffd}{\diffd t}\Pi_x{\restriction_{\sectorY}}=\Pi_x\partial\big\}=:\spaceofmodels_\textnormal{ad}^\mathrm{b}(\reals),
\end{equation*}
where $\partial\in\linear(\sectorY,\sectorYdot)$ is generated by
\begin{equation*}
 \partial\unit:=0,\quad\partial\lroot\zeta\rroot_i:=\zeta\Xi_i.
\end{equation*}
Note that $\partial\varGamma{\restriction_{\sectorY}}=\Gamma\partial$ for all $\varGamma$ in $\structaction{G}^\mathrm{b}$. The models in $\spaceofmodels_\textnormal{ad}^\mathrm{b}(\reals)$ are called \emph{admissible models}.
\subsection{Integration against the kernel and solving RDEs}
We introduce $\abstractint\in\linear(\sectorYdot,\sectorY)$ as
\begin{equation*}
 \abstractint(\zeta\Xi_i)=\lroot\zeta\rroot_i.
\end{equation*}.
\begin{thm}\label{thm:abstractint}
 $\abstractint:\,\sectorYdot\rightarrow\sectorY$ is an \emph{abstract integration map of order} $1$ in the sense of Definition 5.7 \textnormal{\cite{hairer14}} since
 \begin{enumerate}[(i)]
  \item $\abstractint\sectorYdot_\alpha\subseteq\sectorY_{\alpha+1}$,
  \item $\abstractint\varGamma a-\varGamma\abstractint a\in\lingen{\unit}$ for every $a\in\sectorYdot$ and $\varGamma\in\structaction{\butchergroup}^\mathrm{b}$.
 \end{enumerate}

\end{thm}
\begin{proof}
 We have
 \begin{equation*}
  \lvert\abstractint(\zeta\Xi_i)\rvert=\lvert\lroot\zeta\rroot_i\rvert=\lvert\zeta\rvert+\gamma=\lvert\zeta\Xi_i\rvert+1
 \end{equation*}
 and
 \begin{equation*}
  \abstractint\structaction{g}^\mathrm{b}(\zeta\Xi_i)-\structaction{g}^\mathrm{b}\abstractint(\zeta\Xi_i)=\abstractint\big[(\structaction{g}^\mathrm{b}\zeta)\Xi_i\big]-\structaction{g}^\mathrm{b}\lroot\zeta\rroot_i=\lroot\structaction{g}^\mathrm{b}\zeta\rroot_i-g(\lroot\zeta\rroot_i)\unit-\lroot\structaction{g}^\mathrm{b}\zeta\rroot_i=-g(\lroot\zeta\rroot_i)\unit.
 \end{equation*}

\end{proof}
In the following, let $\treepro:\,\sectorY\times\sectorYdot\rightarrow\sectorYdot$ be generated by
\begin{equation*}
 \tau\treepro(\zeta\Xi_i):=(\tau\treepro\zeta)\Xi_i\quad\forallc\tau\in\trees_d\cup\{\unit\},\,\zeta\in\forests_d,\,i\in\{1,\ldots,d\}.
\end{equation*}

\begin{thm}
 A linear map $\varGamma\in\linear(\branchedmodelspace,\branchedmodelspace)$ is in $\structaction{\butchergroup}^\mathrm{b}$ if and only if all of the following properties are satisfied:
 \begin{enumerate}[(i)]
  \item $\varGamma x-x\in\branchedmodelspace_{<\alpha}$ for all $\alpha\in A^\mathrm{b}$ and $x\in\branchedmodelspace_\alpha$,
  \item $\varGamma\sectorY\subseteq\sectorY$ and $\varGamma\sectorYdot\subseteq\sectorYdot$,
  \item $\abstractint\varGamma b-\varGamma\abstractint b\in\lingen{\unit}$ for every $b\in\sectorYdot$,
  \item $\varGamma a\treepro\varGamma b=\varGamma(a\treepro b)$ for all $a\in\sectorY$ and $b\in\sectorYdot$.
 \end{enumerate}

\end{thm}
\begin{proof}~
 \begin{enumerate}
  \item Let $\varGamma\in\structaction{\butchergroup}^\mathrm{b}$ be arbitrary. Let $g\in\butchergroup$ be such that $\varGamma=\structaction{g}^\mathrm{b}$. Properties (i) and (ii) hold by Theorem \ref{thm:branched_regstruc}, property (iii) holds by Theorem \ref{thm:abstractint} and property (iv) holds since
  \begin{align*}
   \varGamma(\tau\treepro\zeta\Xi_i)&=(g\etensor\id)\coproduct^\mathrm{b}(\tau\treepro\zeta\Xi_i)=\big((g\etensor\id)\destar(\tau\treepro\zeta)\big)\Xi_i=\big((g\etensor\id)(\destar\tau\treeprotwo\destar\zeta)\big)\Xi_i\\
   &=\big((g\etensor\id)\destar\tau\big)\treepro\big((g\etensor\id)\destar\zeta\big)\Xi_i=\big((g\etensor\id)\coproduct^\mathrm{b}\tau\big)\treepro\big((g\etensor\id)\coproduct^\mathrm{b}(\zeta\Xi_i)\big)\\
   &=\varGamma\tau\treepro\varGamma(\zeta\Xi_i)
  \end{align*}

  \item Let $\varGamma\in\linear(\branchedmodelspace,\branchedmodelspace)$ be such that it satisfies all of the given properties. Denote by $g$ the element in $\adual{\forestspace{d}}$ recursively generated by
  \begin{equation*}
   g(\unit):=1,\quad g(\lroot\zeta\rroot_i)\unit:=\varGamma\abstractint(\zeta\Xi_i)-\abstractint\varGamma(\zeta\Xi_i),\quad g(\zeta_1\treepro\zeta_2):=g(\zeta_1)g(\zeta_2).
  \end{equation*}
  By the first and last equation, we get that $g\in\butchergroup$. The claim we want to show is $\varGamma=\structaction{g}^\mathrm{b}$. By \eqref{eq:defn_branched_structaction} and \eqref{eq:defn_branched_coproduct} this is equivalent to the identities
  \begin{equation}\label{eq:branched_structaction_on_sectorY}
   \varGamma\tau=(g\etensor\id)\destar\tau
  \end{equation}
  for all $\tau\in\trees_d\cup\{\unit\}$ and
  \begin{equation}\label{eq:branched_structaction_on_sectorYdot}
   \varGamma(\zeta\Xi_i)=\big((g\etensor\id)\destar\zeta\big)\Xi_i
  \end{equation}
  for all $\zeta\in\forests_d$.
  
  \noindent First of all, we indeed have
  \begin{equation*}
   \varGamma\unit=\unit=(g\etensor\id)\destar\unit
  \end{equation*}
  by properties (i) and (ii) as well as
  \begin{equation*}
   \varGamma\Xi_i=\Xi_i=\big((g\etensor\id)\destar\unit\big)\Xi_i
  \end{equation*}
  by property (i) and the fact that the $\Xi_i$ have the lowest homogeneity of the regularity structure. Thus, recalling the iterative scheme from the motivation, \eqref{eq:branched_structaction_on_sectorY} holds on $\mathcal{Y}_0$ and \eqref{eq:branched_structaction_on_sectorYdot} holds on $\dot{\mathcal{Y}}_0$.
  
  \noindent Assuming \eqref{eq:branched_structaction_on_sectorYdot} holds on $\dot{\mathcal{Y}}_n$, we get
  \begin{equation*}
   \varGamma\lroot\zeta\rroot_i=\varGamma\abstractint(\zeta\Xi_i)=g(\lroot\zeta\rroot_i)\unit+\abstractint\varGamma(\zeta\Xi_i)=g(\lroot\zeta\rroot_i)\unit+(g\etensor\lroot\cdot\rroot_i)\destar\zeta=(g\etensor\id)\destar\lroot\zeta\rroot_i
  \end{equation*}
  for all $\zeta\in\dot{\mathcal{Y}}_n$, using in the second equality the definition of $g$ and in the last equality the definition of $\destar$ given in \eqref{eq:defn_destar}. Thus, it follows that \eqref{eq:branched_structaction_on_sectorY} holds on $\mathcal{Y}_{n+1}$.
  
  \noindent Assuming that \eqref{eq:branched_structaction_on_sectorY} holds on $\mathcal{Y}_{n+1}$, we show that \eqref{eq:branched_structaction_on_sectorYdot} holds on $\dot{\mathcal{Y}}_{n+1}$ by an induction over the number of trees of a symbol in $\dot{\mathcal{Y}}_{n+1}$. Indeed, we first have for all $\tau\in\mathcal{Y}_{n+1}$ that
  \begin{equation*}
   \varGamma(\tau\Xi_i)=\varGamma\tau\treepro\varGamma\Xi_i=(\varGamma\tau)\Xi_i=\big((g\etensor\id)\destar\tau\big)\Xi_i,
  \end{equation*}
  where we used property (iv). Then, if \eqref{eq:branched_structaction_on_sectorYdot} holds for some $\zeta\Xi_i\in\dot{\mathcal{Y}}_{n+1}$, we get for all $\tau\in\mathcal{Y}_{n+1}$ that
  \begin{align*}
   \varGamma(\tau\treepro\zeta\Xi_i)&=\varGamma\tau\treepro\varGamma(\zeta\Xi_i)=\big(g\etensor\id)\destar\tau\big)\treepro\big((g\etensor\id)\destar\zeta\big)\Xi_i=\big((g\etensor\id)(\destar\tau\treeprotwo\destar\zeta)\big)\Xi_i\\&=\big((g\etensor\id)\destar(\tau\treepro\zeta)\big)\Xi_i,
  \end{align*}
  where we used property (iv) again in the first identity, the definition of $g$ in the third identity and in the last identity the definition of $\destar$ given in \eqref{eq:defn_destar}. Thus, the induction over the number of trees is complete.
  
  \noindent Finally, by an induction over $n$, we obtain the claim.
 \end{enumerate}
\end{proof}

\begin{thm}
 Let $\alpha>0$, $f\in\contdiff^k(\reals,\reals)$ for some $k\in\naturals$ with $k\geq\alpha/\gamma+1$ and $(\Pi,\Gamma)\in\spaceofmodels^\mathrm{b}(\reals)$.
 Then, $Y\mapsto(f(Y)\Xi_i)_{\alpha+\gamma-1}$, where
 \begin{equation*}
  (f(Y)\Xi_i)_{\alpha+\gamma-1}(t):=\sum_{n=0}^{\lfloor\alpha/\gamma\rfloor}\quotientmap_{<\alpha+\gamma-1}\frac{f^{(n)}(Y^\unit(t))}{n!}(Y(t)-Y^\unit(t)\unit)^{\treepro n}\Xi_i
 \end{equation*}
 and $Y^\unit(t)\unit:=\pi_0 Y(t)$, is a continuous map from $\modeld_\Gamma^{\alpha}(\reals,\sectorY)$ into $\modeld_\Gamma^{\alpha+\gamma-1}(\reals,\sectorYdot)$. More precisely, for any compact $\compact\subseteq\reals^d$ and any $D>0$, there is a $C_{\compact,D}>0$ such that
 \begin{equation}\label{eq:Lipschitz_f_Y_Xi}
  \lVERT (h(F)\Xi_i)_{\alpha+\gamma-1}-(h(F)\Xi_i)_{\alpha+\gamma-1}\rVERT_{\alpha+\gamma-1,\compact}\leq C_{\compact,D}\lVERT F-G\rVERT_{\alpha,\compact}
 \end{equation}
 for all $F,G\in\modeld_\Gamma^{\alpha}(\reals,\sectorY)$ such that $\lVERT F\rVERT_{\alpha,\compact}+\lVERT G\rVERT_{\alpha,\compact}\leq D$.
\end{thm}
\begin{proof}
 First extend $\branchedmodelspace$ to $\forestspace{d}\oplus\bigoplus_{i=1}^d\langle\forests_d\Xi_i\rangle$ with $\structaction{g}^\mathrm{b}{\restriction_{\forestspace{d}}}=\structaction{g}^\mathrm{f}$ and apply Theorem \ref{thm:modeld_composition} to $\modeld_\Gamma^{\alpha}(\reals,\forestspace{d})$ to get $Y\mapsto(f(Y))_{\alpha}$. For all $G\in\modeld_\Gamma^{\alpha}(\reals,\forestspace{d})$ and all compact $\compact\subset\reals$, we have
 \begin{align*}
  \sup_{\substack{s,t\in\compact:\\0<\lvert t-s\rvert\leq 1}}\sup_{\beta<\alpha+\gamma-1}\frac{\lVert G(t)\Xi_i-\Gamma_{ts}(G(s)\Xi_i)\rVert_\beta}{\lvert t-s\rvert^{\alpha+\gamma-1-\beta}}&=\sup_{\substack{s,t\in\compact:\\0<\lvert t-s\rvert\leq 1}}\sup_{\beta<\alpha+\gamma-1}\frac{\lVert(G(t)-\Gamma_{ts}G(s))\Xi_i\rVert_\beta}{\lvert t-s\rvert^{\alpha+\gamma-1-\beta}}\\
  &=\sup_{\substack{s,t\in\compact:\\0<\lvert t-s\rvert\leq 1}}\sup_{\beta<\alpha+\gamma-1}\frac{\lVert G(t)-\Gamma_{ts}G(s)\rVert_{\beta-\gamma+1}}{\lvert t-s\rvert^{\alpha+\gamma-1-\beta}}\\
  &=\sup_{\substack{s,t\in\compact:\\0<\lvert t-s\rvert\leq 1}}\sup_{\theta<\alpha}\frac{\lVert G(t)-\Gamma_{ts}G(s)\rVert_{\theta}}{\lvert t-s\rvert^{\alpha-\theta}}
 \end{align*}
 and
 \begin{equation*}
  \sup_{t\in\compact}\sup_{\beta<\alpha+\gamma-1}\lVert G(t)\Xi_i\rVert_{\beta}=\sup_{t\in\compact}\sup_{\beta<\alpha+\gamma-1}\lVert G(t)\rVert_{\beta-\gamma+1}=\sup_{t\in\compact}\sup_{\theta<\alpha}\lVert G(t)\rVert_{\theta},
 \end{equation*}
 and therefore $\lVERT G\Xi_i\rVERT_{\alpha+\gamma-1,\compact}=\lVERT G\rVERT_{\alpha,\compact}$. Thus, $(f(Y)\Xi_i)_{\alpha+\gamma-1}$ indeed is an element of the space $\modeld_\Gamma^{\alpha+\gamma-1}(\reals,\sectorYdot)$ and we get the Lipschitz bounds \eqref{eq:Lipschitz_f_Y_Xi} from those in Theorem \ref{thm:modeld_composition}.
 
\end{proof}

For the remainder of the discussion, we confine ourselves to a basic outline. The idea is to convolute the kernel $K$ with the reconstructions of modeled distributions in $\modeld_\Gamma^\alpha(\reals,\sectorYdot)$ for some $\alpha>0$. Since those reconstructions live in $\hoelder^{\gamma-1}(\reals)$ and $K$ is a 1-regularizing kernel in the sense of Assumption 5.1 \cite{hairer14}, it is indeed possible to perform such a convolution and to obtain a continuous function in $\hoelder^\gamma(\reals)$ as a result, if we further restrict the set of models in order to get sufficient decay at $-\infty$. We do this by simply demanding that $\Gamma_{st}=\id$ for $s,t\leq 0$, thus this restricted set will still contain the usual setting of rough paths where everything happens on a compact interval of the form $[0,T]$. Now, we can define an abstract kernel operator $\abstractkernel^\alpha:\,\modeld_\Gamma^\alpha([0,\infty),\sectorYdot)\rightarrow\modeld_\Gamma^{\alpha+1}([0,\infty),\sectorY)$ via
 \begin{equation*}
 (\abstractkernel^\alpha F)(t):=\abstractint\,F(t)+(K\ast\reconstr^\alpha F)(t)\,\unit.
\end{equation*}
This operator then obviously fulfills the property that
\begin{equation*}
 \reconstr^{\alpha+1}\abstractkernel^\alpha F=K\ast\reconstr^\alpha F\quad\forallc F\in\modeld_\Gamma^\alpha(\sectorYdot).
\end{equation*}
By showing appropriate Lipschitz bounds for $\abstractkernel^\alpha F$ relative to $F$, we then see that there is $T>0$ such that
\begin{equation*}
 \Psi:\,\modeld_\Gamma^{\alpha+1-\gamma}([0,T],\sectorY)\rightarrow\modeld_\Gamma^{\alpha+1-\gamma}([0,T],\sectorY),\quad\Psi(Z):=y_s\unit+\abstractkernel^{\alpha}\sum_{i=1}^d(f_i(Z)\Xi_i)_{\alpha}
\end{equation*}
is a contraction, and therefore has a unique fixed point.

\subsection{Geometric rough paths structure}
For weakly geometric rough paths, we consider the model space
\begin{equation*}
 \geomodelspace=\lingen{\words_d}\oplus\bigoplus_{i=1}^d\lingen{\words_d\Xi_i},
\end{equation*}
since the kernel of the linear map
\begin{equation*}
 \phi_T:\,\branchedmodelspace\rightarrow \geomodelspace,\quad \phi_T(\unit):=\unit,\quad\phi_T(\tau):=\phi(\tau)\,\,\forallc\tau\in\lingen{\trees_d},\quad\phi_T(\zeta\Xi_i):=\phi(\zeta)\Xi_i\,\,\forallc\zeta\in\forestspace{d},
\end{equation*}
consists of those linear combinations of symbols which will vanish through integration by parts. Here, $\phi:\,\forestspace{d}\rightarrow\lingen{\words_d}$ is again the Hopf algebra endomorphism from Theorem \ref{thm:geom_are_branched}.
As for the case of branched rough paths, $\geomodelspace$ then can be turned into a comodule of now $(\lingen{\words_d},\deconc)$ via
\begin{equation*}
 \coproduct^\mathrm{g}:\,\geomodelspace\rightarrow\lingen{\words_d}\etensor \geomodelspace,\quad\coproduct^\mathrm{g} w\Xi_i=\sum_{(w)}^{\itensor} w_1\etensor w_2\Xi_i,\quad\coproduct^\mathrm{g} w=\deconc w
\end{equation*}
which allows us to define the action of some $g\in\adual{\lingen{\words_d}}$ on $\geomodelspace$ as
\begin{equation*}
 \structaction{g}^\mathrm{g}:=(g\etensor\id)\coproduct^\mathrm{g}.
\end{equation*}
Recalling the group $G^\mathrm{w}$ of characters on $\lingen{\words_d}$, we get a regularity structure $(A^\mathrm{g},\geomodelspace,\structaction{G^\mathrm{w}}^\mathrm{g})$ for weakly geometric rough paths. In the case of $\gamma\in(1/3,1/2]$, it again reduces to the rough paths structure given in Definition 13.5.\@ \cite{frizhairer14} after a suitable truncation. With $\psi(\structaction{g}^\mathrm{g}):=\structaction{g\phi}^\mathrm{b}$, $(\phi_T,\psi)$ can be shown to be a regularity structure morphism and since $\phi_T$ is surjective and $\psi$ injective, there is a sub regularity structure of $\regstruc^\mathrm{b}$ we may identify $\regstruc^\mathrm{g}$ with. Thus, all the concepts developed for $\regstruc^\mathrm{b}$ in the previous subsections can be translated to the case of $\regstruc^\mathrm{g}$.

\FloatBarrier
\bibliography{masterthesis}
\bibliographystyle{alpha}
\small
\printnomenclature[1.75cm]
\markboth{}{}
\scriptsize
\twocolumn
\cleardoublepage
\phantomsection
\printindex
\end{document}